\documentclass{amsart}
\usepackage{graphicx}
\usepackage{amsthm}
\usepackage{tikz-cd}
\usepackage{amssymb}
\usepackage{amsmath}
\usepackage{amsthm}
\usepackage{MnSymbol}
\usepackage{hyperref}
\usepackage{mathtools}
\usepackage{graphicx}
\usepackage{mathrsfs}
\usepackage{bbm}
\usepackage[dvipsnames]{xcolor}
\usetikzlibrary{fit,positioning,calc}
\usepackage{stmaryrd}
\usepackage{booktabs}
\usepackage{tabularx}
\usepackage{array}

\theoremstyle{plain}
\newtheorem{theorem}{Theorem}[section]
\newtheorem*{theorem*}{Theorem}

\theoremstyle{plain}
\newtheorem{proposition}[theorem]{Proposition}

\theoremstyle{plain}

\theoremstyle{plain}
\newtheorem{corollary}[theorem]{Corollary}

\theoremstyle{definition}
\newtheorem{example}[theorem]{Example}

\theoremstyle{definition}
\newtheorem{definition}[theorem]{Definition}

\theoremstyle{definition}
\newtheorem{remark}[theorem]{Remark}

\theoremstyle{plain}
\newtheorem{lemma}[theorem]{Lemma}

\theoremstyle{definition}
\newtheorem{assumption}[theorem]{Assumption}

\theoremstyle{plain}
\newtheorem{introthm}{Theorem}

\hypersetup{colorlinks=true,
            linkcolor=RoyalBlue,
            citecolor=ForestGreen,
            urlcolor=Blue}

\def\AA{\mathbb{A}}

\def\DD{\mathbb{D}}
\def\FF{\mathbb{F}}
\def\GG{\mathbb{G}}
\def\NN{\mathbb{N}}
\def\QQ{\mathbb{Q}}
\def\VV{\mathbb{V}}
\def\ZZ{\mathbb{Z}}
\newcommand\cA{\mathcal{A}}
\newcommand\cC{\mathcal{C}}
\newcommand\cV{\mathcal{V}}
\newcommand\cO{\mathcal{O}}
\newcommand\cP{\mathcal{P}}
\newcommand\cT{\mathcal{T}}
\newcommand\cF{\mathcal{F}}
\newcommand\frakm{\mathfrak{m}}
\newcommand\frp{\mathfrak{p}}
\newcommand\tilG{\widetilde{G}}
\newcommand\tilH{\widetilde{H}}
\newcommand\alg{\textup{alg}}
\newcommand{\coker}{\textup{coker}}
\newcommand{\Fil}{\textup{Fil}}

\newcommand\Frac{\textup{Frac}}
\newcommand\Frob{\textup{Frob}}
\newcommand\Gal{\textup{Gal}}
\newcommand\id{\textup{id}}

\newcommand\pr{\textup{pr}}

\newcommand{\red}{\textup{red}}
\newcommand{\Rep}{\textup{Rep}}
\newcommand{\Res}{\textup{Res}}
\newcommand\res{\textup{res}}
\newcommand\Aut{\textup{Aut}}
\newcommand\Hom{\textup{Hom}}
\newcommand\End{\textup{End}}
\newcommand\GL{\textup{GL}}
\newcommand{\perf}{\textup{perf}}
\newcommand{\cris}{\textup{cris}}
\newcommand{\Ainf}{\mathbb{A}_{\textup{inf}}}
\newcommand{\Bcris}{\mathbb{B}_{\textup{cris}}}
\newcommand{\AinfK}{\mathbb{A}_{\textup{inf},K}}
\newcommand{\BcrisK}{\mathbb{B}_{\textup{cris},K}}
\newcommand{\BdR}{\mathbb{B}_{\textup{dR}}}
\newcommand{\Acris}{\mathbb{A}_{\textup{cris}}}
\newcommand{\AcrisK}{\mathbb{A}_{\textup{cris},K}}
\newcommand{\cts}{\mathrm{cts}}
\newcommand{\phimod}{\textup{$\varphi$-Mod}}
\newcommand{\sep}{\textup{sep}}
\newcommand{\LT}{\textup{LT}}
\newcommand{\sh}{\textup{sh}}
\newcommand{\OC}{\cO_C}
\newcommand{\OCb}{\cO_{C^\flat}}
\newcommand{\Gmhat}{\widehat{\GG}_m} 

\newcommand{\Fp}{\FF_p}
\newcommand{\Zp}{\ZZ_p}
\newcommand{\Qp}{\QQ_p}
\newcommand{\Fq}{\FF_q}
\newcommand{\cc}{{\circ\circ}}
\newcommand{\isomto}{\xrightarrow{\sim}}

\newcommand{\BR}{R^\sep}
\newcommand{\Fpbar}{\overline{\FF}_p}
\newcommand{\Cb}{C^\flat}
\newcommand{\sA}{\mathscr{A}}

\newcommand{\Comm}{\mathrm{Comm}}
\newcommand{\tsigma}{\tilde{\sigma}}
\newcommand{\Fit}{\mathrm{Fit}}
\newcommand{\Id}{\mathrm{Id}}

\makeatletter
\newcommand{\notoc@start}{%
  \global\let\notoc@orig@tocwrite\@tocwrite
  \global\let\@tocwrite\notoc@tocwrite
}
\newcommand{\notoc@tocwrite}[2]{%
  \global\let\@tocwrite\notoc@orig@tocwrite
}
\newcommand{\sectionnotoc}[1]{%
  \notoc@start
  \section*{#1}%
}
\newcommand{\subsectionnotoc}[1]{%
  \notoc@start
  \subsection*{#1}%
}
\makeatother

\title{Formal groups and $(\varphi,\Gamma)$-modules}
\author{Daishi Kiyohara}
\date{\today}
\begin{document}
\maketitle
\begin{abstract}
Let $K/E$ be a finite unramified extension of $p$-adic local fields, and let $H$ be a one-dimensional formal $\cO_E$-module of finite height over $\cO_K$.
We introduce the exponential period map of $H$ and use it to construct a complete regular local ring $R_{H,K}$ with imperfect residue field, an endomorphism $\varphi_E$, and a commuting action of $\Gamma=\Gal(K(H[p^\infty](\overline{K}))/K)$.
We prove an equivalence of categories between finitely generated $\cO_E$-modules with a continuous $\Gal_K$-action and \'{e}tale $(\varphi_E,\Gamma)$-modules over $R_{H,K}$.
This recovers the classical cyclotomic and Lubin-Tate equivalences in the corresponding cases.
In general, $R_{H,K}$ can have Krull dimension greater than one, and $\varphi_E$ need not lift the $q_E$-power Frobenius modulo $\pi_E$.
The proof uses $F$-dynamical systems over $\cO_E$, which combine contraction modulo $\pi_E$ with Frobenius on the residue field.
For every flat $F$-dynamical system, we establish an equivalence of categories between \'{e}tale $\varphi$-modules and continuous representations of the absolute Galois group of the residue field on finitely generated $\cO_E$-modules.
\end{abstract}
\tableofcontents

\newpage
\section{Introduction}
\subsection{Main theorem}
Fontaine's theory of \'{e}tale $(\varphi,\Gamma)$-modules \cite{Fontaine1990Representations} is a cornerstone of the theory of Galois representations of a $p$-adic local field $K$.
It is a powerful tool for studying Galois representations through their relation to $p$-adic Hodge theory \cite{Cherbonnier1998Representations,Berger2002Representationsa} and has played a crucial role in the $p$-adic local Langlands correspondence for $\GL_2(\Qp)$ \cite{Colmez2010Representations}.
Fontaine's construction involves the cyclotomic extension $K(\zeta_{p^\infty})/K$ with $\Gamma=\Gal(K(\zeta_{p^\infty})/K)$.
In his original paper \cite[Introduction]{Fontaine1990Representations}, Fontaine proposed extending the theory to a wider class of extensions $K_{\infty}$ over $K$ whose Galois group $\Gamma$ is a $p$-adic Lie group.
A central difficulty in extending the classical construction is lifting the action on the field of norms \cite{Wintenberger1983corps} to a Cohen ring together with a compatible Frobenius.
Berger subsequently proved that, in the coefficient-linear one-variable setting, the existence of such a lift satisfying the finite-height condition forces $\Gamma$ to be abelian \cite[Theorem B]{Berger2014Lifting}.

In the present paper, we develop a generalization of $(\varphi,\Gamma)$-modules to torsion towers of one-dimensional formal groups.
More precisely, we fix a finite unramified extension $K/E$ of $p$-adic local fields, and let $H$ be a one-dimensional formal $\cO_E$-module of finite height over $\cO_K$.
Let $L/K$ be any finite extension in a fixed algebraic closure $\overline{K}$.
Set
\[K_\infty=K(H[p^\infty](\overline{K})),\quad L_\infty=LK_\infty,\quad\Gamma_L=\Gal(L_\infty/L),\quad \Gamma=\Gamma_K.\]
We will construct a complete regular local ring $R_{H,L}$ equipped with commuting actions of $\varphi_E$ and $\Gamma_L$.
An \'{e}tale $(\varphi_E,\Gamma_L)$-module over $R_{H,L}$ is a finitely generated \'{e}tale $\varphi_E$-module with a commuting semilinear $\Gamma_L$-action satisfying the continuity condition of Definition \ref{def:phigammamod}.
\begin{introthm}[Theorem \ref{thm:phigamma_equiv}]\label{introthm:main}
The category of finitely generated $\cO_E$-modules with a continuous $\Gal_L$-action is equivalent to the category of \'{e}tale $(\varphi_E,\Gamma_L)$-modules over $R_{H,L}$.
\end{introthm}

Assume $E=L=K$, and let $h$ denote the relative $\cO_K$-height.
The following table summarizes several examples,
where $\dim R_{H,K}$ denotes the Krull dimension.
\vspace{0.2cm}
\begin{center}
\begingroup
\small
\renewcommand{\arraystretch}{1.15}
\renewcommand{\tabularxcolumn}[1]{m{#1}}

\begin{tabularx}{\linewidth}{@{}
  >{\raggedright\arraybackslash}m{0.46\linewidth}
  >{\centering\arraybackslash}X
  >{\centering\arraybackslash}m{0.15\linewidth}
@{}}
\toprule
Formal group $H$
  & $\Gamma$
  & $\dim R_{H,K}$ \\
\midrule

$\widehat{\mathbb G}_m$ with $K=\mathbb Q_p$
  & $\mathbb Z_p^\times$
  & $1$ \\

\addlinespace[0.6ex]
Lubin--Tate over $\mathcal O_K$
  & $\mathcal O_K^\times$
  & $1$ \\

\addlinespace[0.6ex]
Relative height $h>1$ over $\cO_K$, with absolute endomorphism ring $\mathcal O_K$
  & An open subgroup of $\mathrm{GL}_h(\mathcal O_K)$
  & $h^2-h+1$ \\

\addlinespace[0.6ex]
Potentially Lubin--Tate for an unramified extension $K'/K$
of degree $h>1$ (Section \ref{sec:pot_lubintate})
  & $\mathcal O_{K'}^\times
       \rtimes \operatorname{Gal}(K'/K)$
  & $h$ \\

\bottomrule
\end{tabularx}
\endgroup
\end{center}
\vspace{0.2cm}
The first two cases recover the classical cyclotomic equivalence \cite{Fontaine1990Representations} and Lubin--Tate equivalence \cite{Ren2009Galois} respectively.
When the relative height $h$ is greater than one, the base ring $R_{H,K}$ has Krull dimension greater than one.
The endomorphism $\varphi_K$ on $R_{H,K}$ is induced from the canonical Frobenius on the period ring $\AinfK$, but it does not reduce modulo $\pi_K$ to the $q_K$-power Frobenius if $h>1$, where $q_K=\#\kappa_K$.

\subsection{Overview of construction}
For the remainder of the introduction, we assume that $E=L=K$.
The equivalence is constructed from \textit{exponential periods} of formal groups.
Let $H$ be a one-dimensional formal $\cO_K$-module over $\cO_K$, let $h$ be the $\cO_K$-height, and let $T(H)$ be the integral Tate module.
We construct the \textit{exponential period map}, an injective $\Gal_K$-equivariant $\cO_K$-linear homomorphism
\[\tau_H:T(H)\to H(\AinfK).\]
Here, $H(\AinfK)$ consists of topologically nilpotent elements in the weak topology with the $\cO_K$-module structure induced by the formal group law of $H$.
Rigidity of the universal cover $\widetilde{H}$ identifies $\widetilde{H}(\cO_C)$ with $\widetilde{H}(\AinfK)$.
Projecting the lift of a compatible system of division points to its zeroth component gives its exponential period.

Applying a formal logarithm $\ell$ to exponential periods recovers the usual crystalline periods.
More precisely, let $[\ell]$ be the class in the de Rham realization $D_K$ of the Dieudonn\'{e} module.
Then
\[\ell(\tau_H(\alpha))=\langle [\ell],\alpha\rangle\in\BcrisK^+\]
for every $\alpha\in T(H)$, where $\langle \,-,-\,\rangle$ is the $p$-adic integration pairing \cite{Colmez1992Periodes}.

We first construct an abstract ring $\cA_H$ with commuting actions of $\varphi_K$ and $\Gamma$ which maps to $\AinfK$ using Frobenius iterates of exponential periods.
The absolute endomorphism ring of $H$ is defined as the union of endomorphism rings $\End_{\cO_{K'}}(H)$ for all finite extensions $K'$ of $K$.
In our setting, the absolute endomorphism ring $\Lambda_H$ is the ring of integers of a finite unramified extension $E_H/K$, and $T(H)$ is free of rank $n=h/[E_H:K]$ over $\Lambda_H$.
The choice of a $\Lambda_H$-basis $\alpha_1,\cdots,\alpha_n$ of $T(H)$ gives an isomorphism
\[\cA_H\cong\Lambda_H[[t_{k,i}:0\le k\le h-1,1\le i\le n]]\]
such that the exponential periods $\tau_i=\tau_H(\alpha_i)$ define a homomorphism
\[\xi_H:\cA_H\to\AinfK,\quad t_{k,i}\mapsto \varphi_K^k(\tau_i)\]
equivariant for $\varphi_K$ and $\Gal_K$.

The image of the composite $\cA_H\to\AinfK\to\OCb$ gives a subalgebra $o_H$ topologically generated by the reductions of exponential periods modulo $\pi_K$.
Let $\frp_H$ denote the kernel of $\cA_H\to \OCb$.
Localizing $\cA_H$ at the prime ideal $\frp_H$ and then $\frp_H$-adically completing it gives a complete local ring $R_H$ with residue field $k_H=\Frac(o_H)$, equipped with natural commuting actions of $\varphi_K$ and $\Gamma$.
Localizing at $\frp_H$ inverts elements whose reductions are nonzero in $k_H$; their images need not be units in $\AinfK$, but they are units in $W_{\cO_K}(\Cb)$.
Therefore, we obtain a commutative diagram
\begin{center}
    \begin{tikzcd}
    R_H\arrow[r]\arrow[d]& W_{\cO_K}(\Cb)\arrow[d]\\
    k_H\arrow[r]&\Cb
    \end{tikzcd}
\end{center}
consisting of equivariant maps for $\varphi_K$ and $\Gal_K$.
The next theorem relates the absolute Galois groups of $k_H$ and $K_\infty$.
Choose a separable closure $k_H^\sep\subset\Cb$, and let $w$ denote the valuation on $k_H^\sep$ induced from the valuation $v_{\Cb}$ on $\Cb$.
\begin{theorem}[Theorem \ref{thm:decomposition}]\label{introthm:decom}
The action of $\Gal_{K}$ on $\Cb$ induces a continuous injective homomorphism
\[\Gal_{K_\infty}\to \Gal_{k_H}\]
whose image is the decomposition subgroup $D_w$, i.e., the closed subgroup consisting of automorphisms which preserve $w$.
\end{theorem}
The proof of Theorem \ref{introthm:decom} involves the tilting equivalence for perfectoid fields \cite[Theorem 3.7]{Scholze2012Perfectoid}.
We first show that the completion $\widehat{K}_\infty$ is perfectoid; the tilting equivalence then gives $\Gal_{\widehat{K}_\infty}\cong\Gal_{\widehat{K}_\infty^\flat}$.
We show that $\widehat{K}_\infty^\flat$ is the completed perfection of $k_H$ with respect to the valuation on $\Cb$, and we use it to prove Theorem \ref{introthm:decom}.
Let $k_{H,K}\subset k_H^\sep$ be the henselization associated with $w$.
Its absolute Galois group is $D_w$, so Theorem \ref{introthm:decom} gives
\[\Gal_{K_\infty}\cong\Gal_{k_{H,K}}.\]
Finally, we define $R_{H,K}$ as the completion at the maximal ideal of the ind-(finite \'{e}tale) local $R_H$-algebra corresponding to $k_{H,K}/k_H$.
The commuting actions of $\varphi_K$ and $\Gamma$ extend to $R_{H,K}$.

\begin{example}
In the multiplicative case $H=\Gmhat$ and $K=\Qp$, choose compatible primitive $p$-power roots of unity $\zeta_{p^n}$ and put $\epsilon=(1,\zeta_p,\cdots)\in\cO_{\Cb}$.
The corresponding generator of $T(\Gmhat)$ has exponential period $\mu=[\epsilon]-1\in\Ainf$.
Then
\[R_H\cong \bigl(\Zp\llbracket \mu\rrbracket[\tfrac{1}{\mu}]\bigr)^\wedge,\quad k_H\cong \Fp(\!(\epsilon-1)\!)\]
where the completion is $p$-adic.
$k_H$ is the field of norms of the cyclotomic tower \cite{Wintenberger1983corps}.
In this case, $k_H$ is already henselian, so
\[k_{H,K}=k_H,\quad R_{H,K}=R_H\]
and our equivalence recovers Fontaine's theory \cite{Fontaine1990Representations}.
\end{example}

\subsection{Proof strategy of Theorem \ref{introthm:main}}\label{subsec:proofstrategy}
We first establish an equivalence of categories
\[\Rep_{\cO_K}(\Gal_{K_\infty})\simeq \textup{$\varphi_K$-Mod}_{R_{H,K}}.\]

In the cyclotomic and Lubin-Tate settings, the finite free version of this equivalence has a prismatic interpretation, using Laurent $F$-crystals \cite{Bhatt2022Prisms}, \cite[Theorem 1.3]{Marks2025Prismatic}, \cite[Theorem 1.1]{Wu2021Galois}.
When $h>1$, however, the endomorphism $\varphi_K$ on $R_{H,K}$ does not lift the $q_K$-power Frobenius modulo $\pi_K$ and hence does not arise from a ramified $\delta$-structure.
These prismatic constructions therefore do not apply to the pair $(R_{H,K},\varphi_K)$.
We leave it as an open question whether Theorem \ref{introthm:main} admits a geometric interpretation analogous to the prismatic interpretation of Fontaine's equivalence.

Our proof uses contraction modulo $\pi_K$: some iterate of the induced endomorphism on $R_{H,K}/(\pi_K)$ sends the maximal ideal into its square.
Together with flatness of $\varphi_K$ and its Frobenius action on the residue field, this contraction property yields the desired equivalence for \'{e}tale $\varphi_K$-modules.
We formulate these properties in the framework of $F$-dynamical systems.
\begin{definition}
An \textit{$F$-dynamical system over $\cO_K$} is a pair $(R,\varphi_K)$, where $R$ is a complete Noetherian local $\cO_K$-algebra and $\varphi_K$ is a local $\cO_K$-algebra endomorphism of $R$ satisfying the following conditions:
\begin{enumerate}
    \item The structure map $\cO_K\to R$ is faithfully flat.
    \item The induced endomorphism $\overline{\varphi}_K$ on $\overline{R}=R/(\pi_K)$ is contracting in the sense that $\overline{\varphi}_K^N(\frakm_{\overline{R}})\subset \frakm_{\overline{R}}^2$ for some $N\ge1$.
    \item The endomorphism on the residue field induced by $\varphi_K$ is the $q$-th power Frobenius, where $q=\#\kappa_K$.
\end{enumerate}
We say $(R,\varphi_K)$ is flat when $\varphi_K:R\to R$ is flat.
\end{definition}

Let $(R,\varphi_K)$ be an $F$-dynamical system over $\cO_K$, let $k$ be the residue field, and let $k^\sep$ be a separable closure of $k$.
Write $R^\sep=\widehat{R^\sh}$ for the completion at the maximal ideal of the strict henselization of $R$ corresponding to $k^\sep/k$.
The endomorphism $\varphi_K$ extends uniquely to an endomorphism $\varphi_K^\sep$ inducing the $q$-th power Frobenius on $k^\sep$, which commutes with a natural action of $\Gal_k$.
For $T\in\Rep_{\cO_K}(\Gal_k)$ and $M\in \textup{$\varphi_K$-Mod}_R$, we set
\[\DD(T)=(T\otimes_{\cO_K}R^\sep)^{\Gal_k},\quad \VV(M)=(M\otimes_RR^\sep)^{\varphi_K=\id}.\]
These functors give the following general equivalence.
\begin{introthm}[Theorem \ref{thm:FDS_cat_equiv}]\label{introthm:fds}
Let $(R,\varphi_K)$ be a flat $F$-dynamical system over $\cO_K$.
Then the functor 
\[\DD:\Rep_{\cO_K}(\Gal_{k})\simeq \textup{$\varphi_K$-Mod}_{R}\]
is an equivalence of categories with a quasi-inverse functor $\VV$.
\end{introthm}

We sketch the proof of Theorem \ref{introthm:fds}.
Put $Q=R/(\pi_K)$.
Contraction of the induced endomorphism $\varphi_K$ on $Q$ gives an equivalence
\[-\otimes_Qk:\textup{$\varphi_K$-Mod}_{Q}\simeq\textup{$\varphi_K$-Mod}_{k}.\]
The latter category is equivalent to the category of finite-dimensional continuous $\Fq$-representations of $\Gal_k$ \cite[Proposition 4.1.1]{Katz1973padic}.
The full statement follows by lifting through $R/(\pi_K^n)$ and using completeness, where flatness of $\varphi_K$ controls the torsion d\'{e}vissage.

For a one-dimensional formal $\cO_K$-module $H$ of finite height over $\cO_K$, we will prove that $(R_{H,K},\varphi_K)$ is a flat $F$-dynamical system over $\cO_K$.
Since the residue field $k_{H,K}$ has absolute Galois group isomorphic to $\Gal_{K_\infty}$, we deduce from Theorem \ref{introthm:fds} an equivalence of categories
\[\DD:\Rep_{\cO_K}(\Gal_{K_\infty})\simeq \textup{$\varphi_K$-Mod}_{R_{H,K}}.\]
Let $T\in \Rep_{\cO_K}(\Gal_K)$.
The diagonal action of $\Gal_K$ on $T\otimes_{\cO_K}(R_{H,K})^\sep$ induces a semilinear action of $\Gamma=\Gal_K/\Gal_{K_\infty}$ on $\DD(T\vert_{\Gal_{K_\infty}})$, still denoted $\DD(T)$.
The $\Gamma$-action satisfies the continuity condition of Definition \ref{def:phigammamod}, so $\DD(T)$ is an \'{e}tale $(\varphi_K,\Gamma)$-module.
Conversely, the additional $\Gamma$-action induces a continuous $\Gal_K$-action on the output of $\VV$, which yields the equivalence in Theorem \ref{introthm:main}:
\[\DD:\Rep_{\cO_K}(\Gal_K)\simeq \textup{$(\varphi_K,\Gamma)$-Mod}_{R_{H,K}}.\]

\subsection{Independence of exponential periods and the structure of $R_{H,K}$}
Still assuming $E=K=L$, we now describe the rings $o_H$, $R_H$, and $R_{H,K}$ more explicitly.
These structural results are established in Section \ref{sec:period_alg}; the construction and the proof of Theorem \ref{introthm:main} do not depend on them.
Let $\kappa_H$ denote the residue field of $\Lambda_H$, and fix a $\Lambda_H$-basis $\alpha_1,\cdots,\alpha_n$ of $T(H)$.
The ring $o_H$ constructed earlier is exactly the $\kappa_H$-subalgebra of $\OCb$ topologically generated by $\overline{\tau}_i$ with $1\le i\le n$, the reductions of the exponential periods of $\alpha_i$.
Those reduced periods satisfy the following formal independence property.
\begin{theorem}[Theorem \ref{thm:bH_injective}]\label{introthm:indep_char_p}
The homomorphism
\[\Fpbar\llbracket t_1,\cdots,t_n\rrbracket\to \OCb,\quad t_i\mapsto\overline{\tau}_i\]
is injective.
\end{theorem}
We prove Theorem \ref{introthm:indep_char_p} by applying arbitrarily small inertia transformations in each variable separately, using the fact that the image of $\rho:I_K\to\GL_n(E_H)$ is open \cite{Serre1967groupes,Sen1973Lie}, where $I_K$ denotes the inertia subgroup of $\Gal_K$.

As a consequence, we have
\[o_H\cong\kappa_H\llbracket t_1,\cdots,t_n\rrbracket,\quad k_H\cong\Frac(\kappa_H\llbracket t_1,\cdots,t_n\rrbracket).\]
When the absolute height $n$ is greater than one, $k_H$ is not a characteristic $p$ local field, and in particular, this deperfection differs from the classical field of norms construction \cite{Wintenberger1983corps}.
For $1\le k\le h-1$ and $1\le i\le n$, write
\[s_{k,i}=t_{k,i}-(t_{0,i})^{q^k}.\]
Then
\[\frp_H\coloneqq \ker(\cA_H\to o_H)=\left(\pi_K,\,s_{k,i}\mid 1\le k\le h-1,1\le i\le n\right)\]
and these generators form a regular sequence in $\cA_H$.
Therefore, the base ring $R_{H,K}$ is a complete regular local ring of dimension $n(h-1)+1$.
These parameters also explain the failure of $\varphi_K$ to lift Frobenius modulo $\pi_K$.
If $h>1$, then 
\[t_{1,i}-t_{0,i}^q=s_{1,i}\not\in\pi_K \cA_H.\]
On the other hand, 
\[\xi_H(s_{1,i})=\varphi_K(\tau_i)-\tau_i^q\in\pi_K\AinfK.\]
Therefore, $\xi_H$ does not reflect divisibility by $\pi_K$.
The same obstruction persists in $R_{H,K}$ because $\pi_K$ and the $s_{k,i}$ form a regular system of parameters there.

Before reduction, we also prove formal independence of the exponential periods together with their first $h-1$ Frobenius iterates.
Recall that the image of $\xi_H:\cA_H\to\AinfK$ is the $\Lambda_H$-subalgebra of $\AinfK$ topologically generated by exponential periods and their $\varphi_K$-iterates.
Let $\breve{K}$ denote the completion of the maximal unramified extension of $K$.
\begin{theorem}[Theorem \ref{thm:injectivity_ainf}]\label{introthm:injective_ainf}
The homomorphism 
\[\cO_{\breve{K}}\llbracket t_{k,i}:0\le k\le h-1,1\le i\le n\rrbracket\to\AinfK,\quad t_{k,i}\mapsto \varphi_K^k(\tau_i)\]
is injective.
\end{theorem}
The proof linearizes the inertia action $\rho:I_K\to\GL_n(E_H)$, which has an open image, and uses the nondegeneracy of the crystalline period pairing.
Restricting the coefficients to $\Lambda_H$ shows that $\xi_H$ is injective.
Hence, we obtain an equality
\[\dim(\mathrm{im}(\xi_H))=nh+1=\dim_{K}(\Gamma)+1\] 
relating the Krull dimension to the dimension as a $K$-analytic Lie group, where the last equality follows from $\dim_K(\Gamma)=n^2[E_H:K]$ by the open-image theorem.

Putting these results together, we obtain the formula
\[\dim\cA_H=nh+1,\quad\dim o_H=n,\quad\dim R_{H,K}=n(h-1)+1.\]
The relative height $h>1$ produces additional parameters for the base ring, while the absolute height $n>1$ produces a residue field that is not a one-variable local field.

Potentially Lubin-Tate groups with respect to finite unramified $L/K$ provide a concrete comparison between the multivariable construction and the one-variable theory.
The multivariable theory describes $\cO_K$-representations of $\Gal_K$, while the one-variable theory describes semilinear $\cO_L$-representations of $\Gal_K$.
In Section \ref{sec:pot_lubintate}, we compare the two theories by base change and iteration of Frobenius.

\subsectionnotoc{Outline}
In Section \ref{sec:exp_period}, we introduce exponential periods of formal groups of arbitrary dimension, and compare them with the $p$-adic integration pairing between the Tate and Dieudonn\'{e} modules.
In Section \ref{sec:onedim}, we analyze exponential periods in the one-dimensional case, establish a short exact sequence (Theorem \ref{thm:ses_htild_ainf}), and then construct the abstract model $\cA_H$.
In Section \ref{sec:perfectoid}, we show that the $p$-adic completion of the splitting field $K(H[p^\infty](\overline{K}))$ is perfectoid, and that its tilt is the completed perfection of the subfield constructed from reduced exponential periods.
In Section \ref{sec:fds}, we introduce the algebraic framework of $F$-dynamical systems and prove that the category of \'{e}tale $\varphi$-modules over a flat $F$-dynamical system is equivalent to the category of Galois representations of the residue field.
In Section \ref{sec:phigamma}, we construct the base ring $R_{H,L}$, study its properties, and finally prove the main theorem (Theorem \ref{introthm:main}).
In Section \ref{sec:period_alg}, we prove independence of reduced exponential periods and of the Frobenius iterates of the exponential periods, and explicitly describe $o_H$, $R_H$, and $R_{H,L}$.
In Section \ref{sec:pot_lubintate}, we study potentially Lubin-Tate formal groups, establish an equivalence of categories between one-variable $(\varphi,\Gamma)$-modules and semilinear Galois representations, and compare this theory with multivariable $(\varphi,\Gamma)$-modules.
Appendix \ref{app:witt} collects basic results on ramified Witt vectors, and Appendix \ref{app:onevariable} adapts Berger's argument \cite{Berger2014Lifting} to explain the obstruction to noncommutative coefficient-linear actions on a one-variable power series ring.

\sectionnotoc{Acknowledgement}
The author thanks his advisor Mark Kisin for continued encouragement and many helpful discussions.
The present work builds on Jean-Marc Fontaine's theory of $(\varphi,\Gamma)$-modules and owes a fundamental intellectual debt to his ideas.
On a more personal level, reading a letter from Fontaine \cite{Fontaineletter}, which Mark Kisin shared with the author in 2024, provided an important source of inspiration and motivation at an early stage of the project.
The author thanks Naoki Imai for his hospitality at the University of Tokyo.
The author thanks Sanath Devalapurkar, Oakley Edens, Naoki Imai, Yutaro Mikami, Dylan Pentland, Alexander Petrov, Gal Porat, Takeshi Tsuji, and Takumi Watanabe for helpful discussions.
The author specifically thanks Takeshi Tsuji for suggesting the argument of Lemma \ref{lem:contract_to_resfield}.
The main theorem (Theorem \ref{introthm:main}) was announced on August 11, 2026, at a conference at the University of Utah.
The author was supported by the Ezoe Memorial Foundation during the work.

\sectionnotoc{AI Disclosure}
By early August 2026, the author had completed a draft of this paper and subsequently used AI tools for proofreading and to obtain feedback on the exposition.
During the final revisions, the author also used GPT Astra in working out the precise estimates in Subsection \ref{subsec:indep_model}.

\sectionnotoc{Notation}
Let $K$ be a finite extension of $\Qp$, and let $\pi_K$ be a uniformizer of $K$.
Write $\kappa_K$ for the residue field of $K$.

Let $E$ be an intermediate field of $K/\Qp$, and let $\pi_E$ be a uniformizer of $E$.
Let $\overline{K}$ be an algebraic closure of $K$.
For an algebraic extension $L$ of $K$, we denote its $p$-adic completion by $\widehat{L}$.

For a field $L$, let $\Gal_L$ be the absolute Galois group of $L$.

Let $C=\widehat{\overline{K}}$, and let $\OC$ be its ring of integers.
For a perfectoid field $L$, we denote by $L^\flat$ its tilt.
We write $\AinfK=W_{\cO_K}(\OCb)$.
Let $\theta_K:\AinfK\to \cO_C$ be the canonical map.

We write $R\llbracket X_1,\cdots,X_d\rrbracket$ for the formal power series ring and $R\llbracket X_1,\cdots,X_d\rrbracket_0$ for the ideal of series with zero constant term.

Write $H$ for a formal group over $\cO_K$, $G$ for its special fiber, and $h_0$ for the ordinary height of $G$.

\newpage
\section{Exponential periods of formal groups}\label{sec:exp_period}
Let $H$ be a commutative formal group of dimension $d$ and of finite height over $\cO_K$.
In this section, we construct a natural homomorphism, called the exponential period map,
\[\tau_H:T(H)\to H(\AinfK)\]
which is $\Gal_K$-equivariant and injective.
It is related to the usual crystalline periods through a formal logarithm $\ell$ of $H$.
With this notation, the composite
\[T(H)\xrightarrow{\tau_H}(\ker\theta_K)^d\xrightarrow{\ell}\BcrisK^+\]
coincides with the $p$-adic integration pairing against the class $[\ell]$ in the de Rham realization $D_K$.

The construction uses the inclusion $T(H)\subset\widetilde{H}(\cO_C)$ and the rigidity isomorphism $\widetilde{H}(\AinfK)\isomto \widetilde{H}(\cO_C)$ of the universal cover.
Then $\tau_H$ is defined by the following diagram
\begin{center}
\begin{tikzcd}
  T(H)\arrow[rrd,dotted,"\tau_H"']\arrow[r,phantom,"\subset"]&\tilH(\OC)&\tilH(\AinfK)\arrow[l,"\sim"',"\textup{\textcolor{blue}{rigidity}}"]\arrow[d,"\pr"]\\
  &&H(\AinfK)
\end{tikzcd}
\end{center}
where $\pr$ denotes projection to the zeroth component.
We will show that $\pr$ is injective.
As an application of exponential periods, we give a short proof of the classical crystallinity statement of the rational Tate module $V(H)$ in the one-dimensional case.

\subsection{Universal covers and rigidity}
In this subsection, we review the definition of universal covers and prove the key rigidity property.
Let $K$ be a finite extension of $\Qp$.
It is often useful to view a formal group as a functor from the category of adic algebras over $\cO_K$ to $\Zp$-modules.
\begin{definition}
An \textit{adic $\cO_K$-algebra} is a pair $(A,I)$ of an $\cO_K$-algebra $A$ and a finitely generated ideal $I$ containing $\pi_K$ such that $A$ is separated and complete with respect to the $I$-adic topology.
\end{definition}
For an adic $\cO_K$-algebra $(A,I)$, we will view $A$ as a complete topological ring with respect to the $I$-adic topology.
Let $A^\cc$ denote the ideal of topologically nilpotent elements.

After choosing coordinates, a commutative formal group $H$ of dimension $d$ over $\cO_K$ can be viewed as a functor from the category of adic $\cO_K$-algebras to the category of $\Zp$-modules which sends $(A,I)$ to $H(A)\coloneqq (A^\cc)^d$ with a module structure given by the formal $\Zp$-module law of $H$.

\begin{definition}[\cite{Fargues2018Courbes,Scholze2013Moduli}]
Let $H$ be a formal group over $\cO_K$.
The \textit{universal cover} of $H$ is defined as the functor $\widetilde{H}$ from the category of adic $\cO_K$-algebras to the category of $\Qp$-vector spaces which sends $(A,I)$ to
\[\tilH(A)\coloneqq \varprojlim_{[p]_H}H(A)=\{(x_i)_{i\ge0}\in H(A)^{\NN}:x_i=[p]_H(x_{i+1})\}.\]
Here, multiplication by $p$ is invertible, with inverse $(x_0,x_1,\cdots)\mapsto (x_1,x_2,\cdots)$; therefore, $\tilH(A)$ is naturally a $\Qp$-vector space.
\end{definition}

The universal cover of a formal group satisfies a rigidity property with respect to certain infinitesimal thickenings.
We will show that reduction modulo any finitely generated ideal $J\subset I$ induces an isomorphism $\tilH(A)\to\tilH(A/J)$, where $A/J$ is endowed with the induced topology from $A$.
However, the topological ring $A/J$ is not necessarily separated and complete with the induced topology, and we define its formal-group points directly by evaluation of lifts.

Let $(A,I)$ be an adic $\cO_K$-algebra, and let $J$ be a subideal of $I$.
We consider the topology on $A/J$ induced from $A$; then we have $(A/J)^\cc=A^\cc/J$.
For $F(X)\in \cO_K\llbracket X_1,\cdots,X_d\rrbracket$ and an element $a\in ((A/J)^\cc)^d$, choose a lift $\widetilde{a}\in (A^\cc)^d$.
The image of $F(\widetilde{a})$ in $A/J$ is independent of the choice.
Therefore, we can similarly define $H(A/J)\coloneqq ((A/J)^\cc)^d$ with a module structure via the formal $\Zp$-module law of $H$, and define
\[\tilH(A/J)\coloneqq\varprojlim_{[p]_H}H(A/J).\]

We can now state the rigidity lemma for universal covers.
It is a slight variant of \cite[Proposition 4.5.2]{Fargues2018Courbes} and replaces the closedness assumption by finite generation.

\begin{lemma}\label{lem:rigidity}
Let $H$ be a formal group of finite height over $\cO_K$.
Let $(A,I)$ be an adic $\cO_K$-algebra and let $J\subset I$ be a finitely generated ideal.
Then the reduction 
\[\tilH(A)\to\tilH(A/J)\]
is an isomorphism, and the inverse sends $(x_i)_i$ to $(\lim_{n\to\infty}[p^n]_H(a_{n+i}))_i$ where $a_{n+i}\in (A^\cc)^d$ is a lift of $x_{n+i}$.
\end{lemma}

\begin{proof}
We first prove that for every ideal $J\subset I$ the reduction map
\begin{equation}\label{eq:rigidity}
    \tilH(A/JI)\to\tilH(A/J)
\end{equation}
is an isomorphism.

Fix coordinates $X=(X_1,\cdots,X_d)$ on $H$ and let $P=(P_1,\cdots,P_d)=[p]_H$.
Then $P\in (\cO_K\llbracket X_1,\cdots,X_d\rrbracket_0)^d$.
Since each component of $P(X)$ is a power series in $X_1^p,\cdots,X_d^p$ modulo $\pi_K$, the partial derivatives are all divisible by $\pi_K$.
Reducing modulo $\pi_K$, we obtain componentwise
\[P(X+\epsilon)-P(X)\in (\pi_K \epsilon_1,\cdots,\pi_K\epsilon_d,\epsilon_1^p,\cdots,\epsilon_d^p).\]

Let $x\in (A^\cc/J)^d$.
Pick two lifts $a,a+j\in (A^\cc)^d$; then we have $P(a+j)-P(a)\in (JI)^d$.
Therefore, the map $\iota:H(A/J)\to H(A/JI)$ sending $x$ to $[P(a)]$ for any lift $a$ is well-defined.
Let $r:H(A/JI)\to H(A/J)$ denote the reduction map.
Then the identities $r\iota=[p]_H$ and $\iota r=[p]_H$ hold, so the shifted map 
\[\tilH(A/J)\to \tilH(A/JI),\quad (x_0,x_1,\cdots)\mapsto (\iota(x_1),\iota(x_2),\cdots)\]
gives an inverse of the reduction map, showing the isomorphism (\ref{eq:rigidity}).

Set $J_0=J$ and $J_{n+1}=J_nI$ for each $n\ge0$.
By (\ref{eq:rigidity}), the reduction map $\tilH(A/J_n)\to \tilH(A/J)$ is an isomorphism with an inverse $(x_i)_i\mapsto ([p^n]_H(a_{n+i}))_i$ where $a_{n+i}\in ((A/J_n)^\cc)^d$ is a lift of $x_{n+i}$.

We now assume that $J$ is finitely generated.
Then $A$ is separated and complete for the filtration $\{J_n\}$; separatedness follows from $J_n\subset I^{n+1}$, while completeness follows by expressing successive differences in terms of finitely many generators of $J$ and summing their $I$-adically convergent coefficient series.
Hence,
\[\tilH(A)\cong\varprojlim_n\tilH(A/J_n).\]
These isomorphisms identify the inverse system $\tilH(A/J_n)$ with the constant system $\tilH(A/J)$, providing an isomorphism between their inverse limits.
The inverse map sends $(x_i)_i$ to $(\lim_{n\to\infty}[p^n]_H(a_{n+i}))_i$ where $a_{n+i}\in (A^\cc)^d$ is a lift of $x_{n+i}$.
\end{proof}

\subsection{The exponential period map}
Let $\AinfK=W_{\cO_K}(\OCb)$.
There is a surjective homomorphism $\theta_K:\AinfK\to\OC$ which fits into the following commutative diagram
\begin{center}
\begin{tikzcd}
  \AinfK\arrow[r,"\theta_K"]\arrow[d]&\OC\arrow[d]\\
  \OCb\arrow[r]&\OC/(\pi_K)
\end{tikzcd}
\end{center}
where the vertical maps are reductions modulo $\pi_K$.
The kernel of $\theta_K$ is a principal ideal \cite[Proposition 3.1.9]{Fargues2018Courbes}.
Let $\xi$ be any generator.
The ring $\AinfK$ is separated and complete for the $(\pi_K,\xi)$-adic topology, called the \textit{weak topology}, and hence is an adic $\cO_K$-algebra.

\begin{proposition}\label{prop:tilH_isom}
Let $H$ be a formal group over $\cO_K$ of finite height.
Then the commutative diagram
\begin{center}
\begin{tikzcd}
    \tilH(\AinfK)\arrow[r,"\theta_K"]\arrow[d]&\tilH(\OC)\arrow[d]\\
    \tilH(\OCb)\arrow[r]&\tilH(\OC/(\pi_K))
\end{tikzcd}
\end{center}
consists of $\Gal_K$-equivariant isomorphisms of $\Qp$-vector spaces.
\end{proposition}
\begin{proof}
Apply Lemma \ref{lem:rigidity} to $(\AinfK,I)$ with $I=(\pi_K,\xi)$ and the three finitely generated ideals $J_1=(\pi_K)$, $J_2=(\xi)$ and $J_3=(\pi_K,\xi)$.
The quotient identifications and commutativity give all four isomorphisms; the quotient by $\xi$ is $\cO_C$ with $\pi_K$-adic topology, the quotient by $\pi_K$ is $\OCb$ with its valuation topology, and the quotient by $(\pi_K,\xi)$ is $\cO_C/(\pi_K)$ with discrete topology.
Functoriality gives $\Gal_K$-equivariance.
\end{proof}

We first define the exponential period map on the universal cover $\widetilde{H}(\cO_C)$.
\begin{definition}
Let $H$ be a formal group over $\cO_K$ of finite height.
We define the \textit{exponential period map}
\begin{equation}\label{eq:exp_period}
    \widetilde{\tau}_H:\tilH(\OC)\to H(\AinfK)
\end{equation}
as the composite of the inverse of the isomorphism $\tilH(\AinfK)\isomto \tilH(\OC) $ in Proposition \ref{prop:tilH_isom} and the projection $\tilH(\AinfK)\to H(\AinfK)$ sending $(x_i)_{i\ge0}$ to $x_0$.
We denote its restriction to $T(H)$ by $\tau_H$.
\end{definition}

\begin{proposition}\label{prop:tilh_to_h_ainf}
Let $H$ be a formal group over $\cO_K$ of finite height.
Then the projection $\tilH(\AinfK)\to H(\AinfK)$ is injective.
\end{proposition}
\begin{proof}
Let $G=H\times_{\cO_K}\kappa_K$ be the special fiber of $H$.
Consider the following commutative diagram
\begin{equation}\label{cd:univ_Ainf_OCb}
\begin{tikzcd}
    \tilH(\AinfK)\arrow[r,"\pr"]\arrow[d]&H(\AinfK)\arrow[d]\\
    \tilG(\OCb)\arrow[r,"\pr"]&G(\OCb)
\end{tikzcd}
\end{equation}
The left vertical map is an isomorphism by Proposition \ref{prop:tilH_isom}.
By Lemma \ref{lem:tilh_to_h_OCb} below, the bottom horizontal map is also an isomorphism; commutativity therefore implies that the top horizontal map is injective.
\end{proof}
\begin{lemma}\label{lem:tilh_to_h_OCb}
Let $k$ be a perfect field of characteristic $p$, let $G$ be a formal group of finite height over $k$, and let $R$ be a perfect adic $k$-algebra.
Then the projection map $\pr:\tilG(R)\to G(R)$ is an isomorphism.
\end{lemma}
\begin{proof}
Consider the following commutative diagram:
\begin{equation}\label{cd:univ}
    \begin{tikzcd}
    G\arrow[d,equal]&G\arrow[l,"p"]\arrow[d,"V"]&G\arrow[d,"V^2"]\arrow[l,"p"]&\cdots\arrow[l]\\
    G&G^{(p^{-1})}\arrow[l,"F"]&G^{(p^{-2})}\arrow[l,"F"]&\cdots\arrow[l]
    \end{tikzcd}
\end{equation}
where $G^{(p^{-i})}=G\times_{k,F_k^{-i}}k$, $F_k$ is the Frobenius on $k$, and $F:G^{(p^{-(i+1)})}\to G^(p^{-i})$ is relative Frobenius.
The vertical map $V^i$ is the composite of the appropriately twisted Verschiebung maps.
More precisely, write $[p]_G(X)=f(X^p)$, and let $\sigma$ denote the $p$-power Frobenius on $k$; then we set
\[V^0=\id,\quad V^i=f^{\sigma^{-i}}\circ f^{\sigma^{-(i-1)}}\circ \cdots \circ f^{\sigma^{-1}}\,\,(i\ge1).\]

The vertical arrows define a homomorphism 
\[\cV:\widetilde{G}\to\varprojlim_{F}G^{(p^{-i})}.\] 
Then $\cV$ is an isomorphism (see the proof of \cite[Proposition 3.1.3 (iii)]{Scholze2013Moduli} as well as \cite[Proposition 4.6.4]{Fargues2018Courbes}).
Since $R$ is perfect, the projection map
\[\varprojlim_{F}G^{(p^{-i})}(R)\to G(R)\]
is an isomorphism.
\end{proof}

\begin{corollary}\label{cor:tau_equivariant}
The exponential period map 
\[\widetilde{\tau}_H:\tilH(\OC)\to H(\AinfK)\]
is a $\Gal_K$-equivariant injective $\Zp$-module homomorphism.
Its restriction to $T(H)$ lands inside $(\ker\theta_K)^d$.
\end{corollary}
\begin{proof}
By Propositions \ref{prop:tilH_isom} and \ref{prop:tilh_to_h_ainf}, $\widetilde{\tau}_H$ is a composition of an isomorphism followed by an injective map, both $\Gal_K$-equivariant $\Zp$-module homomorphisms.
The composite $\theta_K\circ \widetilde{\tau}_H:\widetilde{H}(\cO_C)\to H(\cO_C)$ is projection to the zeroth component, so it vanishes on $T(H)$.
\end{proof}
The exponential period map is functorial in $H$.
Consequently, if $H$ has an $\cO_E$-action, then $\widetilde{\tau}_H$ is $\cO_E$-linear for the formal $\cO_E$-module law.
\begin{remark}
In the unramified setting treated by Abrashkin, with $p>2$, $\tau_H$ is the map $j$ of \cite[Section 1.5.4]{Abrashkin1997Explicit}.
For $\alpha=(\alpha_i)\in T(H)$, choose lifts $\widetilde{\alpha}_i\in(\AinfK)^d$ of $\alpha_i$.
The map $j$ is then defined by 
\[j(\alpha)=\lim_{n\to\infty}[p^n]_H(\widetilde{\alpha}_n).\]
The universal cover description above gives the same map by Lemma \ref{lem:rigidity} and works over arbitrary $\cO_K$.
\end{remark}
\subsection{Comparison with crystalline periods}
The exponential period map $\tau_H$ recovers the classical crystalline periods after applying a formal logarithm.

Let $H$ be a formal group over $\cO_K$ of dimension $d$ and height $h_0<\infty$, and let $G$ be the special fiber of $H$.
Let $K_0$ denote the maximal unramified subextension of $K/\Qp$.
On the one hand, the contravariant rational Dieudonn\'{e} module $D$ of the special fiber $G$ is a $K_0$-vector space of dimension $h_0$ equipped with a bijective semilinear endomorphism $\phi_p$.
Set $D_K\coloneqq D\otimes_{K_0}K$.
This $h_0$-dimensional $K$-vector space carries the Hodge filtration
\[\Fil^i(D_K)=\begin{cases}D_K&i\le 0\\ \Omega_H&i=1\\ 0&i\ge2
\end{cases}\]
where $\Omega_H$ is the $d$-dimensional space of invariant differential forms.
The combined structure $(D,\phi_p:D\to D,\Fil^\bullet(D_K))$ is usually referred to as a filtered $\varphi$-module over $K$.

The assignment $H\mapsto (D,\phi_p,\Fil^\bullet D_K)$ defines a fully faithful contravariant functor from the isogeny category of finite-height formal groups over $\cO_K$ to the category of filtered $\varphi$-modules over $K$.
On the other hand, the rational Tate module defines a fully faithful covariant functor from the same isogeny category to the category of finite-dimensional continuous $\Qp$-representations of $\Gal_K$.
Fontaine's crystalline period ring $\Bcris$ relates these constructions through the functor $\DD_\cris$.

\begin{theorem}[{\cite[Theorem 6.2]{Fontaine1982certains}}]\label{thm:crystalline_comparison}
Let $H$ be a formal group of finite height over $\cO_K$, and let $V=V(H)$.
Then there is a natural Frobenius-equivariant isomorphism
\[D\isomto\Hom_{\Qp[\Gal_K]}(V,\Bcris).\]
After extension of scalars to $K$, the Hodge filtration identifies with the filtration induced by $\BdR$.
\end{theorem}

Explicit Dieudonn\'{e} theory in \cite{Fontaine1977Groupes} (see also \cite[Section 3]{Colmez1992Periodes}, \cite[Section V]{Katz1981Crystalline}) provides a way to understand $D$ and $D_K$ in terms of formal power series.
For $f(X)\in K\llbracket X_1,\cdots,X_d \rrbracket_0$, set $\delta f=f(H(X,Y))-f(X)-f(Y)$.
We say that $f(X)$ is a \textit{quasi-logarithm} if $\delta f$ and $df$ have bounded denominators\footnote{The condition on $df$ can be omitted as in \cite{Colmez1992Periodes} because boundedness of $\delta f$ implies boundedness of $df$.}.
The assignment $f\mapsto [f]$ identifies $D_K$ as follows:
\[D_K\cong \frac{\{f\in K\llbracket X \rrbracket_0:\textup{$\delta f$, $df$ bounded}\}}{\{f\in K\llbracket X \rrbracket_0: \textup{$f$ bounded}\}}.\]
A \textit{logarithm} of $H$ is a power series $f$ satisfying $\delta f=0$, and $f\mapsto df$ induces an isomorphism
\[\{f\in K\llbracket X \rrbracket_0:\delta f=0\}\cong \Omega_H.\]

When we restrict to quasi-logarithms $f(X)$ with coefficients in $K_0$, we obtain an isomorphism
\[D\cong \frac{\{f\in K_0\llbracket X \rrbracket_0:\textup{$\delta f$, $df$ bounded}\}}{\{f\in K_0\llbracket X \rrbracket_0: \textup{$f$ bounded}\}}\]
between $K_0$-vector spaces.
The isomorphism is Frobenius-equivariant with respect to $\phi_p$ defined by $\phi_p([f(X)])=[f^\sigma(X^p)]$, where $\sigma$ is the arithmetic Frobenius on $K_0$.

Let $\Acris$ denote the $p$-adic completion of the divided power envelope of $\Ainf$ with respect to $\theta:\Ainf\to O_C$, and let $\Bcris^+=\Acris[1/p]$.
Let $\BcrisK^+=\Bcris^+\otimes_{K_0}K$.

Let $H$ be a finite-height formal group over $\cO_K$, and let $V=V(H)$.
The \textit{$p$-adic integration pairing} \cite[Proposition 3.1]{Colmez1992Periodes} is a pairing
\[\langle-,-\rangle:D_K\times V\to\BcrisK^+\]
$K$-linear in the first variable and $\Qp$-linear in the second, with the following properties.
\begin{enumerate}
  \item It is equivariant with respect to the $\Gal_K$-action on $V$ and $\BcrisK^+$.
  \item It is compatible with filtrations, satisfying
  \[\langle \omega,\alpha\rangle\in \Fil^1(\BdR^+)\]
  for all $\omega\in \Omega_H$ and $\alpha\in V$.
  \item It restricts to a pairing
  \[\langle-,-\rangle:D\times V\to\Bcris^+\]
  which is Frobenius equivariant, i.e., 
  \[\langle \phi_p(\omega),\alpha\rangle=\varphi_p(\langle \omega,\alpha\rangle)\]
  holds for $\omega\in D$ and $\alpha\in V$.
\end{enumerate}
Here we briefly review the construction without proof.
View $V(H)$ as a subspace of $\widetilde{H}(\cO_C)$, let $\alpha=(\alpha_n)_n\in V$ with $\alpha_n=[p]_H(\alpha_{n+1})$, and let $f$ be a quasi-logarithm of $H$.
For each $n$, choose a lift $\widetilde{\alpha}_n\in (\AinfK)^d$ of $\alpha_n$; then $f(\widetilde{\alpha}_n)$ converges inside $\BcrisK^+$.
Moreover, the sequence $\{p^nf(\widetilde{\alpha}_n)\}_n$ converges to an element of $\BcrisK^+$ which depends only on $\alpha$ and $[f]\in D_K$.
For $\omega=[f]$, define
\[\langle\omega,\alpha\rangle\coloneqq\lim_{n\to\infty}p^nf(\widetilde{\alpha}_n).\]
Our sign convention is opposite to that of \cite{Colmez1992Periodes}.

The following relation between exponential periods and the $p$-adic integration pairing explains the terminology.
\begin{proposition}\label{prop:relation_to_padic_integration}
Let $H$ be a formal group over $\cO_K$ of finite height.
Let $\ell$ be a logarithm of $H$.
The composite
\[T(H)\xrightarrow{\tau_H}(\ker(\theta_K))^d\xrightarrow{\ell}\BcrisK^+\]
coincides with the restriction of the $p$-adic integration pairing
\[\langle[\ell],-\rangle:T(H)\to\BcrisK^+.\]
\end{proposition}
\begin{proof}
Let $\alpha\in T(H)$, and consider the element $(\tau_n)_n\in\tilH(\AinfK)$ corresponding to $\alpha$ under the natural isomorphism $\tilH(\OC)\cong\tilH(\AinfK)$ in Proposition \ref{prop:tilH_isom}.
By definition of exponential periods, we have $\tau_H(\alpha)=\tau_0$.

By the construction of the $p$-adic integration pairing, we have
\[\langle [\ell],\alpha\rangle=\lim_{n\to\infty} p^n\ell(\tau_n).\]
For each $n$, we have
\[p^n\ell(\tau_n)=\ell([p^n]_H(\tau_n))=\ell(\tau_0).\]
\end{proof}

Using only the construction and the equivariance of the integration pairing, along with properties of exponential periods, we now recover crystallinity in dimension one; the argument does not use the comparison isomorphism of Theorem \ref{thm:crystalline_comparison}.
\begin{proposition}\label{prop:onedim_crystalline}
Let $H$ be a one-dimensional formal group over $\cO_K$ of finite height.
Then $V(H)$ is a crystalline representation.
\end{proposition}
\begin{proof}
The $p$-adic integration pairing induces a $K_0$-linear homomorphism
\begin{equation}\label{eq:D_to_Dcrisdual}
  D\to \Hom_{\Qp[\Gal_K]}(V,\Bcris)=\DD_\cris(V^\vee)
\end{equation}
which is equivariant with respect to Frobenius.
We note that $D$ has slope $1/h_0$ and rank $h_0$, so it is simple as an isocrystal.
Consequently, if (\ref{eq:D_to_Dcrisdual}) is not injective, then it must be a zero map and so the pairing on $D_K$ must also be zero.
Choose the normalized logarithm $\ell$ of $H$ and a nonzero $\alpha\in T(H)$.
By Corollary \ref{cor:tau_equivariant}, $\tau=\tau_H(\alpha)$ is nonzero and lies in $\ker\theta_K$, so
\[\ell(\tau)=\tau u\]
for some $u\in 1+\Fil^1\BdR^+\subset(\BdR^+)^\times$.
Therefore, 
\[\langle [\ell],\alpha\rangle=\ell(\tau)\neq0\] 
by Proposition \ref{prop:relation_to_padic_integration}.
Therefore, the pairing is nonzero and the map (\ref{eq:D_to_Dcrisdual}) is injective.
It follows that
\[\dim_{K_0}(\DD_\cris(V^\vee))\ge \dim_{K_0}(D)=h_0.\]
On the other hand, every Galois representation $W$ satisfies
\[\dim_{K_0}(\DD_\cris(W))\le \dim_{\Qp}(W)\]
which implies the upper bound $\dim_{K_0}(\DD_\cris(V^\vee))\le h_0$.
Hence, we conclude that 
\[\dim_{K_0}(\DD_\cris(V^\vee))=h_0=\dim_{\Qp}(V^\vee)\]
so $V^\vee$ is crystalline.
Since the category of crystalline representations is closed under duality \cite[Theorem 5.2 (i)]{Fontaine1982certains}, we conclude that $V$ is crystalline.
\end{proof}
\subsection{Reduced periods and tilting}
We now express reduced exponential periods in tilt coordinates.
This compatibility will be used in Section \ref{sec:perfectoid} to identify the completion of the perfection of the period field with the tilt of the torsion tower.

Let $G$ be a formal group of finite height and dimension $d$ over $\Fpbar$.
Rigidity (Proposition \ref{prop:tilH_isom}) and Lemma \ref{lem:tilh_to_h_OCb} give the following isomorphisms
\begin{center}
\begin{tikzcd}
  \tilG(\OCb)\arrow[d,"\mathrm{red}"]\arrow[r,"\pr"]&G(\OCb)\\
  \tilG(\OC/(\pi_K))\arrow[ru,dotted]
\end{tikzcd}
\end{center}
We will give a formula for the resulting isomorphism 
\[\iota:\tilG(\OC/(\pi_K))\to G(\OCb).\]
Write $[p]_G(X)=f(X^p)$, and let $\sigma$ denote the $p$-power Frobenius on $\Fpbar$.
By defining the maps $V^i:G\to G^{(p^{-i})}$ by
\[V^0=\id,\quad V^i=f^{\sigma^{-i}}\circ f^{\sigma^{-(i-1)}}\circ \cdots \circ f^{\sigma^{-1}}\,\,(i\ge1),\]
there is a canonical isomorphism
\[\cV:\widetilde{G}(\OC/(\pi_K))\to\varprojlim_F G^{(p^{-i})}(\OC/(\pi_K)),\quad (x_i)_{i\ge0}\mapsto (V^i(x_i))_{i\ge0}\]
constructed in the proof of Lemma \ref{lem:tilh_to_h_OCb}.

\begin{proposition}\label{prop:explicit_isom_tilG_to_G}
Let $G$ be a formal group law over $\Fpbar$ of finite height.
There is a commutative diagram
\begin{center}
\begin{tikzcd}
    \tilG(\OCb)\arrow[d,"\mathrm{red}"]\arrow[r,"\pr"]&G(\OCb)\\
    \tilG(\OC/(\pi_K))\arrow[r,"\cV"']\arrow[ru,"\iota"]&\varprojlim_F G^{(p^{-i})}(\OC/(\pi_K))\arrow[u]
\end{tikzcd}
\end{center}
consisting of isomorphisms, where the right vertical map is induced by
\[\OCb\cong\varprojlim_{x\mapsto x^p}\cO_C/(\pi_K).\]
\end{proposition}
\begin{proof}
We need to show the following equality
\begin{equation}\label{eq:iota}
    \iota((x_i)_{i\ge0})=(V^i(x_i))_{i\ge0}.
\end{equation}

Write $X=(X_1,\cdots,X_d)$, and $[p]_G=f(X^p)$ with $f(X)\in\Fpbar\llbracket X \rrbracket_0^d$, and let $\sigma$ denote $p$-power Frobenius on both $\Fpbar$ and $\cO_C/(\pi_K)$.
Let $\alpha=(\alpha_i)_i\in \tilG(\OCb)$ so that $\alpha_i=[p]_G(\alpha_{i+1})$ holds for all $i$.
Write $\alpha_i=(\alpha_i^{(m)})$ under the tilt identification, where $\alpha_i^{(m)}\in(\cO_C/(\pi_K))^d$ satisfies $\alpha_i^{(m)}=\sigma \alpha_i^{(m+1)}$ componentwise for every $m$.
In other words, we represent an element $\alpha\in\tilG(\OCb)$ by a doubly indexed family $\alpha_i^{(m)}$:
\begin{center}
\begin{tikzpicture}
\matrix (M) [
  matrix of math nodes,
  row sep=0.1cm,
  column sep=0.1cm
] {
  \alpha_0^{(0)} & \alpha_0^{(1)} & \cdots & \alpha_0^{(m)} & \cdots \\
  \alpha_1^{(0)} & \alpha_1^{(1)} & \cdots & \alpha_1^{(m)} & \cdots \\
  \vdots         & \vdots         & \ddots & \vdots         & \cdots \\
  \alpha_m^{(0)} & \alpha_m^{(1)} & \cdots & \alpha_m^{(m)} & \cdots \\
  \vdots         & \vdots         & \vdots & \vdots         & \ddots \\
};

\node[
  fit=(M-1-1)(M-1-5),
  fill=cyan!30,
  rounded corners=4pt,
  inner xsep=4pt,
  inner ysep=2pt,
  opacity=.5
] {};

\node[
  fit=(M-1-1)(M-5-1),
  fill=YellowOrange!30,
  rounded corners=4pt,
  inner xsep=2pt,
  inner ysep=4pt,
  opacity=.5
] {};

\matrix (M2) [
  matrix of math nodes,
  row sep=0.1cm,
  column sep=0.1cm
] at (M.center) {
  \alpha_0^{(0)} & \alpha_0^{(1)} & \cdots & \alpha_0^{(m)} & \cdots \\
  \alpha_1^{(0)} & \alpha_1^{(1)} & \cdots & \alpha_1^{(m)} & \cdots \\
  \vdots         & \vdots         & \ddots & \vdots         & \cdots \\
  \alpha_m^{(0)} & \alpha_m^{(1)} & \cdots & \alpha_m^{(m)} & \cdots \\
  \vdots         & \vdots         & \vdots & \vdots         & \ddots \\
};
\end{tikzpicture}
\end{center}
The first column (\textcolor{YellowOrange}{orange strip}) is exactly $\mathrm{red}(\alpha)\in\tilG(\OC/(\pi_K))$.
The first row (\textcolor{cyan}{blue strip}) is $\pr(\alpha)=\alpha_0\in G(\OCb)$.
By definition,
\[\iota((\alpha_i^{(0)})_i)=(\alpha_0^{(m)})_m.\]
We now study the relations in the doubly indexed family.
The horizontal relation is $\alpha_i^{(m)}=\sigma\alpha_i^{(m+1)}$.
To analyze the vertical relation, let $(x^{(m)})\in (\OCb)^d$.
An element $a\in\Fpbar$ corresponds to $(\sigma^{-m}(a))_m\in\OCb$, so $a(x^{(m)})=(\sigma^{-m}(a)x^{(m)})_m$.
Consequently, for $f(X)\in\Fpbar\llbracket X \rrbracket$ we have
\[f((x^{(m)})_m)=(f^{\sigma^{-m}}(x^{(m)}))_m.\]
Since $(\alpha_i^{(m)})_m=[p]_G((\alpha_{i+1}^{(m)})_m)$, we have
\begin{align*}
    \alpha_i^{(m)}
    = ([p]_G)^{\sigma^{-m}}(\alpha_{i+1}^{(m)})
    = f^{\sigma^{-m}}\circ\sigma(\alpha_{i+1}^{(m)}).
\end{align*}
Hence, the vertical relation is
\begin{equation}\label{eq:vertical}
  \alpha_i^{(m)}
  =
  f^{\sigma^{-m}}
  \circ
  \sigma\bigl(\alpha_{i+1}^{(m)}\bigr)
\end{equation}
Iterating (\ref{eq:vertical}) $m$ times,
we obtain the following formula:
\begin{align*}\label{eq:formula_m0_from_mm}
    \alpha^{(m)}_0
    &=(f^{\sigma^{-m}}\circ\sigma)\circ (f^{\sigma^{-m}}\circ\sigma)\circ \cdots\circ (f^{\sigma^{-m}}\circ\sigma)(\alpha^{(m)}_m)\\
    &=(f^{\sigma^{-m}}\circ f^{\sigma^{-(m-1)}}\circ\cdots\circ f^{\sigma^{-1}})(\alpha^{(0)}_m)
\end{align*}
where we used $\sigma\circ f^{\sigma^{-j}}=f^{\sigma^{1-j}}\circ\sigma$ and $\sigma^m(\alpha^{(m)}_m)=\alpha^{(0)}_m$.
This proves the desired equality (\ref{eq:iota}).
\end{proof}

\newpage 
\section{Exponential periods of one-dimensional formal groups}\label{sec:onedim}
In this section, we first give a detailed analysis of exponential periods in the one-dimensional case.
For a one-dimensional formal $\cO_E$-module, we construct an $\cO_E$-linear operator $\cP$ on $H(\AinfK)$, whose restriction gives the short exact sequence of Corollary \ref{cor:ses_TH_ainf}
\begin{center}
\begin{tikzcd}
0\arrow[r]&T(H) \arrow[r,"\tau_H"]&\ker\theta_K\arrow[r,"\cP"]&\pi_K\cdot\AinfK\arrow[r]&0
\end{tikzcd}
\end{center}
where the module structures are induced by $H$.

Assuming that $K/E$ is finite unramified of degree $f$, we then construct the abstract period algebra that will serve as the starting point for the base ring in Theorem \ref{introthm:main}.
Let $E_H$ be the fraction field of the absolute endomorphism ring of $H$, and set $K'=KE_H$.
Writing $d$ for the degree of the Frobenius minimal polynomial introduced in Section \ref{subsec:descr_image} and $\alpha_1,\cdots,\alpha_n$ for a $\Lambda_H$-basis of $T(H)$, we will construct a ring
\[\cA_H\cong\cO_{K'}\llbracket t_{k,i}:0\le k\le df-1,1\le i\le n\rrbracket\]
with commuting actions of $\varphi_E$ and $\Gamma$, together with an equivariant homomorphism
\[\xi_H:\cA_H\to\AinfK,\quad t_{k,i}\mapsto \varphi_E^k(\tau_H(\alpha_i)).\]
The algebra $\cA_H$ is defined independently of any algebraic relations among the actual periods.
The map $\xi_H$ realizes its generators as periods; the injectivity questions are treated in Section \ref{sec:period_alg}.
\subsection{Endomorphisms and Tate modules of one-dimensional formal groups}
We recall basic facts about one-dimensional formal groups and their endomorphisms that will be used below, referring to \cite{Hazewinkel1978Formal} for further background.

We first recall the height and endomorphism theory over a field $k$ of characteristic $p$.
Let $G$ be a one-dimensional formal group over $k$.
Either $[p]_G(T)=0$, or there is an integer $h_0\ge1$ such that $[p]_G(T)=u(T^{p^h_0})$ for a series $u(T)\in k\llbracket T\rrbracket_0$ with $u'(0)\neq0$.
The height of $G$ is $\infty$ in the first case and $h_0$ in the second.
\begin{theorem}[{\cite[Corollary 20.2.14]{Hazewinkel1978Formal}}]
Let $H$ be a one-dimensional formal group of finite height $h_0$ over a separably closed field $k$ of characteristic $p$.
Then $\End_k(H)$ is isomorphic to the maximal order in the central division algebra $D_{1/h_0}$ of Hasse invariant $1/h_0$ and dimension $h_0^2$ over $\Qp$.
\end{theorem}

For the following facts, let $K$ be a complete discretely valued field of mixed characteristic $(0,p)$ with residue field $k$.
Let $H$ be a one-dimensional formal group over $\cO_K$.
The \textit{special fiber} $G$ is the formal group over $k$ obtained by reducing the coefficients of $H$ modulo the maximal ideal of $\cO_K$.
We write $\mathrm{ht}(H)=\mathrm{ht}(G)$ and call it the height of $H$.

Let $H_1$ and $H_2$ be one-dimensional formal groups over $\cO_K$. 
Then the derivative map
\[\Hom_{\cO_K}(H_1,H_2)\to \cO_K,\quad f\mapsto f'(0)\]
is an injective homomorphism of additive groups with closed image for the valuation topology \cite[Lemma 2.1.1]{Lubin1964Oneparameter}.
Therefore, we can identify $\End_{\cO_K}(H)$ with its image under $f\mapsto f'(0)$.
Since the image is closed and contains $\ZZ$, it contains $\Zp$.

\begin{definition}
Let $H$ be a one-dimensional formal group of finite height over $\cO_K$.
Its \textit{absolute endomorphism ring} is
\[\Lambda_H\coloneqq\bigcup_{K\subset K'\subset\overline{K},[K':K]<\infty}\End_{\cO_{K'}}(H)\]
\end{definition}

\begin{proposition}[{\cite[Proposition 5.1.1]{Cox1974Formal}}]\label{prop:absolute_end_finite}
Let $K/\Qp$ be finite, and let $H$ be a one-dimensional formal group of finite height $h_0$ over $\cO_K$.
Let $K_{h_0}/K$ be the unramified extension of degree $h_0$ inside $\overline{K}$.
Then $\Lambda_H\subset\cO_{K_{h_0}}$.
\end{proposition}

\begin{proof}
Let $G$ denote the special fiber of $H$.
Reduction induces an injective homomorphism $\Lambda_H\to\End_{k^\sep}(G)$ \cite[Lemma 2.3.1]{Lubin1964Oneparameter}.
Therefore, $E_H\coloneqq\Frac(\Lambda_H)$ embeds as a commutative subfield of $D_{1/h_0}$, and $[E_H:\Qp]$ divides $h_0$ (see \cite[Theorem 2.3.2]{Lubin1964Oneparameter}).
Waterhouse's unramifiedness result \cite[Corollary 3.5]{Waterhouse1972pdivisible} implies that $KE_H/K$ is unramified, hence Galois.
The action of $\Gal_K$ preserves the absolute endomorphism ring and acts on its derivatives in the usual way, so it preserves $E_H$.
Restriction gives a homomorphism
\[\Gal(KE_H/K)\hookrightarrow\Aut(E_H)\]
which is injective because $K$ and $E_H$ generate $KE_H$.
The fixed field theorem implies
\[[KE_H:K]=[E_H:E_H\cap K].\]
Combining it with $[E_H:\Qp]\mid h_0$ concludes the proof.
\end{proof}

Put $E_H=\Frac(\Lambda_H)$.
The degree $[E_H:\Qp]$ is finite and divides $\mathrm{ht}(H)$.
The \textit{absolute height} of $H$ is defined by
\[n\coloneqq \mathrm{ht}(H)/[E_H:\Qp].\]
The action of $\Lambda_H$ on $T(H)$ extends to an $E_H$-vector space structure on $V(H)=T(H)\otimes_{\Zp}\Qp$ of dimension $n$.
The $\Gal_K$-action is semilinear with respect to its action on $E_H$.

The field $E_H$ also determines the general linear group in which the inertia image is open.
Let $H$ be a one-dimensional formal group of finite height over $\cO_K$.
For $n\ge1$, let $H[p^n](\overline{K})$ denote the group of points $x\in\frakm_{\overline{K}}$ satisfying $[p^n]_H(x)=0$.
Recall that the Tate module
\[T(H)\coloneqq\varprojlim_{[p]_H}H[p^n](\overline{K})\] 
is a free $\Zp$-module of rank $h_0=\mathrm{ht}(H)$, equipped with a natural continuous action of $\Gal_K$.
Write \[\rho:\Gal_K\to\Aut_{\Zp}(T(H))\] for this representation.
By Proposition \ref{prop:absolute_end_finite}, the inertia subgroup $I_K\subset\Gal_K$ fixes $E_H$, and so its action on $V(H)$ is $E_H$-linear.
The following open-image theorem is due to Serre and Sen; \cite[Theorem 5]{Serre1967groupes} when $\Lambda_H=\Zp$ and \cite[Theorem 3 and the following remark]{Sen1973Lie} in general (see also \cite[Theorem 1.1]{Berger2020Rigidity} for a formulation using an open subgroup of the full Galois image).
\begin{theorem}\label{thm:openimage}
Let $K/\Qp$ be finite, and let $H$ be a one-dimensional formal group of finite height $h_0$ over $\cO_K$.
Put $E_H=\Frac(\Lambda_H)$ and $n=h_0/[E_H:\Qp]$.
The image of the $E_H$-linear representation
\[\rho\vert_{I_K}:I_K\to\Aut_{E_H}(V(H))\cong \GL_n(E_H)\]
is open for the $p$-adic topology.
\end{theorem}

Let $E$ be an intermediate field of the finite extension $K/\Qp$.
A \textit{formal $\cO_E$-module} over $\cO_K$ is a formal group $H$ equipped with a ring homomorphism 
\[\cO_E\to\End_{\cO_K}(H),\quad a\mapsto [a]_H\]
satisfying $[a]_H'(0)=a$.
Under our derivative identification, this is equivalent to $\cO_E\subset\End_{\cO_K}(H)$.
For a formal $\cO_E$-module $H$ of finite height, we define its \textit{relative height} by
\[\mathrm{ht}_{\cO_E}(H)=\mathrm{ht}(H)/[E:\Qp].\]

For the construction in Section \ref{subsec:model}, we will assume that $K/E$ is finite unramified and that $H$ is a one-dimensional formal $\cO_E$-module of finite height over $\cO_K$.
In this setting, the following integrality theorem identifies its absolute endomorphism ring with $\cO_{E_H}$.
\begin{proposition}[{\cite[Theorem 5.2.1]{Cox1974Formal}}]\label{prop:end_int_closed}
Let $K/E$ be a finite unramified extension of $p$-adic local fields.
If $H$ is a one-dimensional formal $\cO_E$-module of finite height $h_0$ over $\cO_K$, then $\Lambda_H=\cO_{E_H}$ is a discrete valuation ring.
Consequently, $T(H)$ is a finite free $\Lambda_H$-module of rank $n=h_0/[E_H:\Qp]$.
\end{proposition}
\begin{remark}
Set $E=K$, and fix $h\ge1$.
Then the isomorphism classes of formal $\cO_K$-modules $H$ of relative height $h$ over $\cO_K$ are in bijection with the monic Eisenstein polynomials over $\cO_K$ of degree $h$.
This was proved in \cite[Theorem 3.6.1]{Cox1974Formal} by extending the methods in \cite{Honda1970theory}.
The Eisenstein polynomial under the bijection is exactly the minimal polynomial of $q_K$-power Frobenius over $K$ in the endomorphism ring of the special fiber.
\end{remark}

\subsection{A description of exponential periods}\label{subsec:descr_image}
In this subsection, we characterize one-dimensional exponential periods of formal groups as the kernel of an operator defined by a Frobenius relation.

Fix a subfield $E\subset K$ containing $\Qp$, and keep the notation $q_K=\#\kappa_K$.
Let $H$ be a one-dimensional formal $\cO_E$-module over $\cO_K$ of finite height $h_0=\mathrm{ht}(H)$, and let $G$ denote its special fiber over $\kappa_K$.
Reduction embeds $\cO_E$ into $\End_{\kappa_K}(G)\subset\cO_{D_{1/h_0}}$, and the $q_K$-power Frobenius $X\mapsto X^{q_K}$ defines an element $\xi_{q_K}$ of this ring.
The reduced $\cO_E$-endomorphisms are defined over $\kappa_K$, so they commute with $\xi_{q_K}$, defining a commutative subalgebra $\cO_E[\xi_{q_K}]$ of $\cO_{D_{1/h_0}}$.
Let $P(X)=\sum_{i=0}^{d}c_iX^i$ be the monic minimal polynomial of $\xi_{q_K}$ over $E$.
Every $E$-conjugate of $\xi_{q_K}$ has positive valuation, so the nonleading coefficients of its minimal polynomial lie in $\frakm_E$.
Since $E(\xi_{q_K})$ is a commutative subfield of $D_{1/h_0}$, its degree over $\Qp$ divides $h_0$.
Therefore, the degree $d=[E(\xi_{q_K}):E]$ of $P$ divides the relative height $\mathrm{ht}_{\cO_E}(H)$.

Define the $\cO_E$-linear operator 
\[\cP:H(\AinfK)\to H(\AinfK),\quad \cP(x)=\sideset{}{_H}\sum_{i=0}^{d} [c_i]_H(\varphi_K^i(x))\]
where the sum is taken with respect to the formal group law of $H$.
Writing $\widetilde{\tau}_H$ for the exponential period map on $\widetilde{H}(\OC)$, our main goal is to prove the $\cO_E$-linear short exact sequence
\begin{center}
\begin{tikzcd}
0\arrow[r]&\tilH(\OC)\arrow[r,"\widetilde{\tau}_H"]&H(\AinfK)\arrow[r,"\cP"]&\pi_K\AinfK\arrow[r]&0
\end{tikzcd}
\end{center}
of Theorem \ref{thm:ses_htild_ainf}, where $\pi_K\AinfK$ is equipped with the formal group law of $H$.
Corollary \ref{cor:tau_equivariant} establishes injectivity of the period map $\widetilde{\tau}_H$ on the universal cover.
We next identify its image with $\ker\cP$.
\begin{proposition}\label{prop:tau_image}
Let $H$ be a one-dimensional formal $\cO_E$-module of finite height over $\cO_K$, with $\cP$ defined above.
Then
\[\mathrm{im}(\widetilde{\tau}_H)=\ker(\cP:H(\AinfK)\to H(\AinfK)).\]
\end{proposition}
\begin{proof}
Let $G$ denote the special fiber of $H$.
Reduction modulo $\pi_K$ gives a commutative diagram
\begin{center}
\begin{tikzcd}
    \tilH(\AinfK)\arrow[r,"\cP"]\arrow[d]&\tilH(\AinfK)\arrow[d]\\
    \tilG(\OCb)\arrow[r,"P(\xi_{q_K})"]&\tilG(\OCb)
\end{tikzcd}
\end{center}
where the bottom arrow is induced by the following endomorphism of $G$
\[P(\xi_{q_K})=\sideset{}{_G}\sum_{i=0}^{d}[c_i]_G\circ\xi_{q_K}^i.\]
The bottom arrow vanishes because $P(\xi_{q_K})=0$.
Rigidity (Proposition \ref{prop:tilH_isom}) identifies the vertical arrows, so the induced operator $\cP$ on $\tilH(\AinfK)$ vanishes.
It follows that $\mathrm{im}(\widetilde{\tau}_H)\subset\sA$ where $\sA=\ker(\cP:H(\AinfK)\to H(\AinfK))$.

The $\cO_E$-linearity of $\cP$ makes $\sA$ an $\cO_E$-submodule of $H(\AinfK)$.
Since $\OCb$ is perfect, $\varphi_K$ is an automorphism of $\AinfK$ and it commutes with $\cP$.
Therefore, $\sA$ is stable under $\varphi_K^{\ZZ}$.
Since $c_i/\pi_E\in\cO_E$ for $i<d$, we have
\[\cP(x)=\varphi_K^d(x)+_H\sideset{}{_H}\sum_{i=0}^{d-1}\left[\frac{c_i}{\pi_E}\right]_H(\varphi_K^{i}([\pi_E]H(x))).\]
Applying $\varphi_K^{-d}$ and taking the formal group inverse gives
\begin{equation}\label{eq:sA_equiv_condition}
    x\in\sA\Leftrightarrow x=\sideset{}{_H}\sum_{i=0}^{d-1}\left[-\frac{c_i}{\pi_E}\right]_H(\varphi_K^{i-d}([\pi_E]_H(x)))
\end{equation}
Define the $\cO_E$-linear operator 
\[\nu=\sideset{}{_H}\sum_{i=0}^{d-1}\left[-\frac{c_i}{\pi_E}\right]_H\varphi_K^{i-d}:\sA\to \sA.\]
Equation (\ref{eq:sA_equiv_condition}) gives the identity
\[\nu([\pi_E]_H(x))=x\]
for every $x\in\sA$.
Since $\nu$ is $\cO_E$-linear, we also have $[\pi_E]_H(\nu(x))=x$ for every $x\in\sA$.
Hence, $[\pi_E]_H$ is an automorphism of $\sA$ with inverse $\nu$.
We can identify the $[p]_H$ and $[\pi_E]_H$ inverse limits because there exist $e$ and $u\in\cO_E^\times$ such that $p=u\pi_E^{e}$.
In the latter coordinates, the map
\begin{equation}\label{eq:nu}
    \sA\to\tilH(\AinfK),\quad x\mapsto (x,\nu(x),\nu^2(x),\cdots)
\end{equation}
is clearly a right inverse to the zeroth projection $\tilH(\AinfK)\to\sA$.
Conversely, a compatible sequence $(y_m)_m\in \tilH(\AinfK)$ has $y_m\in \sA$ and $[\pi_E^m]_H(y_m)=y_0$, so we have $y_m=\nu^m(y_0)$.
The map (\ref{eq:nu}) is therefore the inverse of the zeroth projection $\tilH(\AinfK)\to\sA$, proving $\sA=\mathrm{im}(\widetilde{\tau}_H)$.
\end{proof}

By Corollary \ref{cor:tau_equivariant} and Proposition \ref{prop:tilH_isom}, the reduction modulo $\pi_K$ induces an isomorphism
\[\textup{im}(\widetilde{\tau}_H)\subset H(\AinfK)\to G(\OCb).\]
Proposition \ref{prop:tau_image} characterizes the inverse by the equation $\cP(x)=0$.
\begin{corollary}\label{cor:characterize_exp_period}
Let $H$ be a one-dimensional formal $\cO_E$-module of finite height over $\cO_K$, and let $P$ be the minimal polynomial of $\xi_{q_K}$ over $E$.
\begin{enumerate}
    \item For every $x\in\frakm_{\Cb}$, there is a unique lift $\{x\}_H\in\AinfK$ satisfying $\cP(\{x\}_H)=0$.
    \item Write $G$ for the special fiber of $H$.
    The assignment $x\mapsto \{x\}_H$ defines an $\cO_E$-linear isomorphism
    \[G(\OCb)\to \textup{im}(\widetilde{\tau}_H)\]
    inverse to the reduction modulo $\pi_K$.

\end{enumerate}
\end{corollary}

\begin{example}[Multiplicative case]\label{ex:multiplicative}
Consider the multiplicative formal group $H=\Gmhat$ over $\Zp$.
With $K=E=\Qp$, we have $[p]_{\Gmhat}(X)=(X+1)^p-1$, $\xi_p=[p]_G$, and $P(X)=X-p$.
Choose primitive roots of unity $\zeta_{p^n}$ compatibly; then $\alpha=(\zeta_{p^n}-1)_n$ belongs to $T(H)$.
Writing $\epsilon=(1,\zeta_p,\zeta_{p^2},\cdots)\in\OCb$, the reduced exponential period map sends $\alpha$ to $\overline{\tau}_H(\alpha)=\epsilon-1$.
The Teichm\"{u}ller lift gives an element $[\epsilon]-1\in\Ainf$ reducing to $\epsilon-1$ and satisfying
\[\varphi([\epsilon]-1)=[\epsilon]^p-1=[p]_{H}([\epsilon]-1).\]
Uniqueness in Corollary \ref{cor:characterize_exp_period} therefore gives
\[\tau_{H}(\alpha)=[\epsilon]-1.\]
Consequently, the exponential period of $\alpha$ is Fontaine's cyclotomic parameter $\mu=[\epsilon]-1$.
\end{example}

\begin{example}[Lubin-Tate case]\label{ex:tau_LT}
Take $E=K$, and let $H$ be a Lubin-Tate group associated to a uniformizer $\pi_K$; then $P(X)=X-\pi_K$.
After identifying the $p$-adic and $\pi_K$-adic Tate modules, write $\alpha=(\alpha_i)\in T(H)$ with $[\pi_K]_H(\alpha_{i+1})=\alpha_i$.
Its reduced period has the unique lift $\tau$ satisfying
\[\varphi_K(\tau)=[\pi_K]_H(\tau),\]
recovering the construction of \cite[Lemma 1.2]{Ren2009Galois}.
\end{example}

\begin{theorem}\label{thm:ses_htild_ainf}
Let $H$ be a one-dimensional formal $\cO_E$-module of finite height over $\cO_K$.
There is a $\Gal_K$-equivariant short exact sequence of $\cO_E$-modules
\begin{center}
\begin{tikzcd}
0\arrow[r]&\tilH(\OC)\arrow[r,"\widetilde{\tau}_H"]&H(\AinfK)\arrow[r,"\cP"]&\pi_K\cdot\AinfK\arrow[r]&0
\end{tikzcd}
\end{center}
where the last two terms carry the formal group law of $H$.
\end{theorem}

\begin{proof}
The reduction of $\cP$ is $P(\xi_{q_K})=0$, so its image lies in $\pi_K\AinfK$.
Corollary \ref{cor:tau_equivariant} and Proposition \ref{prop:tau_image} reduce the theorem to surjectivity of $\cP$.

We use Lemma \ref{lem:pth_power_Witt}, the key fact that taking a $p$-th power removes dependence on the last ramified Witt coordinates in every positive index, and refer the reader to Appendix \ref{app:witt} for basic facts and notation about ramified Witt vectors.
Fix $y=(y_n)_{n\ge0}\in \pi_K\cdot\AinfK$, written in the $\pi_K$-Witt coordinates, so $y_0=0$.
For every $x_0\in\frakm_{\Cb}$, we will construct a unique lift $x\in\AinfK^\cc$ satisfying $\cP(x)=y$.
The zeroth coordinate equation holds for the prescribed $x_0$ because $\cP$ reduces modulo $\pi_K$ to zero and $y_0=0$.
We construct the remaining Witt coordinates $x_n$ recursively for $n\ge1$.
The equation $\cP(x)=y$ is equivalent to
\begin{equation}\label{eq:cpx_equal_y}
    \varphi_K^d(x)=[\pi_E]_H\left(\sideset{}{_H}\sum_{i=0}^{d-1}\left[-\frac{c_i}{\pi_E}\right]_H(\varphi_K^i(x))\right)+_Hy.
\end{equation}
The $n$-th Witt coordinate of the left-hand side is $(x_n)^{q_K^d}$.
For Witt vectors over a $\kappa_K$-algebra, multiplication by $\pi_K$ satisfies $\pi_Kz=V_{\pi_K} \varphi_K(z)$.
Its zeroth coordinate is zero, and its $n$-th coordinate for $n\ge1$ is $z_{n-1}^{q_K}$.
The reduction of $[\pi_E]_H(X)$ has zero derivative, so we can write
\[[\pi_E]_H(X)=U(X^p)+\pi_K V(X)\]
with $U,V\in\cO_K\llbracket X\rrbracket_0$.
For $n\ge1$, Lemma \ref{lem:pth_power_Witt} shows that the $n$-th coordinate of $U(z^p)$ depends only on $z_0,\cdots,z_{n-1}$, and the Verschiebung identity gives the same assertion for $\pi_K V(z)$.
Consequently, the $n$-th coordinate of the right-hand side of (\ref{eq:cpx_equal_y}) depends only on $x_0,\cdots,x_{n-1}$ and $y_0,\cdots,y_n$. 
For each $n\ge1$, perfection of $\OCb$ determines a unique $x_n$ from the equation $x_n^{q_K^d}=b_n$ where $b_n$ depends only on the previously chosen coordinates and $y$.
The resulting Witt vector $x$ satisfies (\ref{eq:cpx_equal_y}) in every coordinate and lies in $\AinfK^\cc$ because its reduction $x_0$ is topologically nilpotent.
This proves existence and uniqueness.
\end{proof}
\begin{corollary}\label{cor:ses_TH_ainf}
Let $H$ be a one-dimensional formal $\cO_E$-module of finite height over $\cO_K$.
There is a $\Gal_K$-equivariant short exact sequence of $\cO_E$-modules
\begin{center}
\begin{tikzcd}
0\arrow[r]&T(H) \arrow[r,"\tau_H"]&H(\ker\theta_K)\arrow[r,"\cP"]&H(\pi_K\AinfK)\arrow[r]&0
\end{tikzcd}
\end{center}
where $H(I)$ denotes the ideal $I$ with the formal group law of $H$.
\end{corollary}
\begin{proof}
The compatibility $\theta_K\circ\widetilde{\tau}_H=\pr$ gives the following commutative diagram, which involves the exact sequence of Theorem \ref{thm:ses_htild_ainf}.
\begin{center}
\begin{tikzcd}
    0\arrow[r]&\tilH(\OC)\arrow[d,"\pr"]\arrow[r,"\widetilde{\tau}_H"]&\AinfK^\cc\arrow[d,"\theta_K"]\arrow[r,"\cP"]&\pi_K\cdot\AinfK\arrow[d]\arrow[r]&0\\
    0\arrow[r]&H(\cO_C)\arrow[r,"\id"]&H(\cO_C)\arrow[r]&0\arrow[r]&0
\end{tikzcd}
\end{center}
Since $C$ is algebraically closed, $[p]_H:H(\OC)\to H(\OC)$ is surjective, and successive choices of $p$-division points show that $\pr$ is surjective with kernel $T(H)$.
The snake lemma gives the desired short exact sequence.
\end{proof}
\begin{remark}\label{rmk:existing_exponential_ses}
For $K/\Qp$ unramified, Abrashkin proved a related short exact sequence for formal groups of arbitrary dimension using crystalline period rings \cite[Proposition, \S2.1]{Abrashkin1997Explicit}.
The argument above gives the one-dimensional sequence directly over a possibly ramified field $K$.
\end{remark}

\begin{remark}\label{rem:blochkato}
Corollary \ref{cor:ses_TH_ainf} defines a connecting map
\[\delta:H(\pi_K\AinfK)^{\Gal_K}\to H_{\mathrm{cont}}^1(\Gal_K,T(H)).\]
The fixed-ring identity $(\AinfK)^{\Gal_K}=\cO_K$, together with $\theta_K\vert_{\cO_K}=\id$, gives 
\[(\ker\theta_K)^{\Gal_K}=0.\]
Consequently, the connecting homomorphism is an injective $\cO_E$-linear homomorphism
\[\delta:H(\pi_K\cO_K)\to H_{\mathrm{cont}}^1(\Gal_K,T(H)).\]
\end{remark}

\subsection{Formal group ring}
We construct \textit{formal group rings}, which provide models for the subalgebras of period rings topologically generated by exponential periods.
We adapt \cite[Definition 2.4]{Calmes2013Invariants} by imposing the additional relations associated with the $\Lambda$-action.

Let $R$ be a commutative ring, let $H$ be a one-dimensional commutative formal $\Lambda$-module over $R$, and let $M$ be a free $\Lambda$-module of rank $n$.
Write $R[M]=R[x_m:m\in M]$ for the polynomial ring with one variable $x_m=[m]$ for each element of the underlying set of $M$.
Let $J=([m]:m\in M)$ be the augmentation ideal of $R[M]$, and set $R\llbracket M\rrbracket=\varprojlim_r R[M]/J^r$.
The following is a formal $\Lambda$-module analogue of \cite[Corollary 2.13]{Calmes2013Invariants}.
\begin{lemma}\label{lem:formal_group_ring}
Let $M$ be a finite free $\Lambda$-module of rank $n$.
Let $I_H\subset R\llbracket M\rrbracket$ be the closed ideal generated by
\[[m_1+m_2]-H([m_1],[m_2]),\quad  [\lambda m]-[\lambda]_H([m])\]
for $m_1,m_2,m\in M$ and $\lambda\in\Lambda$.
Every $\Lambda$-basis $e_1,\cdots,e_n$ of $M$ determines an isomorphism
\[R\llbracket M \rrbracket /I_H\cong R\llbracket t_1,\cdots,t_n \rrbracket\]
which sends $[e_i]\mapsto t_i$.
\end{lemma}
\begin{proof}
Every element $m\in M$ has a unique expression $m=c_1e_1+\cdots+c_ne_n$ with $c_i\in\Lambda$.
Since the series $[c_1]_H(t_1)+_H\cdots+_H[c_n]_H(t_n)$ has zero constant term, the assignment $[m]\mapsto [c_1]_H(t_1)+_H\cdots+_H [c_n]_H(t_n)$ defines a continuous $R$-algebra homomorphism
\[R\llbracket M \rrbracket \to R\llbracket t_1,\cdots,t_n \rrbracket .\]
The formal $\Lambda$-module identities and continuity show that this homomorphism factors through $R\llbracket M\rrbracket/I_H$.
The continuous homomorphism in the opposite direction
\[R\llbracket t_1,\cdots,t_n \rrbracket \to R\llbracket M \rrbracket /I_H,\quad t_i\mapsto [e_i]\]
is its inverse.
\end{proof}
\begin{definition}
Let $H$ be a formal $\Lambda$-module over $R$, and let $M$ be a finite free $\Lambda$-module.
We define the \textit{formal group ring} by
\[R\llbracket M \rrbracket _{H,\Lambda}\coloneqq R\llbracket M \rrbracket /I_H\]
endowed with the quotient topology.
\end{definition}
\begin{remark}\label{rem:formalgroupring_functor}
A $\Lambda$-linear map $f:M\to N$ induces the continuous $R$-algebra homomorphism $R\llbracket M\rrbracket_{H,\Lambda}\to R\llbracket N\rrbracket_{H,\Lambda}$ sending $[m]\mapsto [f(m)]$.
This construction defines a functor to complete topological $R$-algebras.
By functoriality, every $\Lambda$-linear endomorphism $f$ of $M$ induces a continuous $R$-algebra endomorphism of $R \llbracket M \rrbracket_{H,\Lambda}$ by sending $[m]\mapsto[f(m)]$.
\end{remark}
Formal group rings satisfy the universal property that for every adic $R$-algebra $A$ and a finite free $\Lambda$-module $M$ there is a canonical isomorphism
\[\Hom_{R,\cts}(R\llbracket M\rrbracket_{H,\Lambda},A)\simeq\Hom_{\Lambda}(M,H(A)),\]
converting a $\Lambda$-linear map into a continuous homomorphism of $R$-algebras.
In fact, any $\Lambda$-linear map
\[\iota:M\to H(A)\]
induces a continuous $R$-algebra homomorphism
\[\widetilde{\iota}:R\llbracket M \rrbracket_{H,\Lambda} \to A,\quad [m]\mapsto \iota(m).\]

Let $H$ be a one-dimensional formal group of finite height over $\cO_{K}$, and set $K'=KE_H$.
The absolute endomorphisms give $T(H)$ a natural $\Lambda_H$-module structure.
For $\gamma\in\Gal_K$ and $\lambda\in\Lambda_H$, $([\lambda]_H)^\gamma$ is an endomorphism of $H$, whose derivative is $\sigma(\lambda)$ with $\sigma=\gamma\vert_{K'}\in\Gal(K'/K)$, so it coincides with $[\sigma(\lambda)]_H$.
Hence, the action of $\gamma$ on $T(H)$ is $\sigma$-semilinear.
The following lemma extends such semilinear actions to formal group rings.
In particular, it gives a natural $\Gal_K$-action on $\cO_{K'}\llbracket T(H)\rrbracket_{H,\Lambda_H}$ whenever $T(H)$ is finite free over $\Lambda_H$.
\begin{lemma}\label{lem:semilinear_action}
Let $M$ be a finite free $\Lambda_H$-module, and let $\iota:M\to M$ be a $\sigma$-semilinear endomorphism for some $\sigma\in\Gal(K'/K)$.
Then $\iota$ induces a continuous $\cO_K$-algebra endomorphism $\widetilde{\iota}$ of $\cO_{K'}\llbracket M \rrbracket _{H,\Lambda_H}$ whose restriction to $\cO_{K'}$ is $\sigma$.
\end{lemma}
\begin{proof}
On the completed polynomial ring $\cO_{K'}\llbracket M \rrbracket$, define
\[\widetilde{\iota}([m])=[\iota(m)],\quad\widetilde{\iota}(a)=\sigma(a)\]
for $m\in M$ and $a\in \cO_{K'}$.
It remains to check that $\widetilde{\iota}(I_H)\subset I_H$.
For $\lambda\in\Lambda_H$ and $m\in M$, semilinearity gives 
\begin{align*}
    \widetilde{\iota}([\lambda m]-[\lambda]_H([m]))
    &= [\sigma(\lambda)\iota(m)]-[\sigma(\lambda)]_H([\iota(m)])\in I_H.
\end{align*}
Since $\sigma$ fixes the coefficients of $H$, we have
\begin{align*}
    \widetilde{\iota}(H([m_1],[m_2])-[m_1+m_2])
    =H([\iota(m_1)],[\iota(m_2)])-[\iota(m_1)+\iota(m_2)]\in I_H
\end{align*}
for $m_1,m_2\in M$.
Continuity now gives $\widetilde{\iota}(I_H)\subset I_H$, so $\widetilde{\iota}$ descends to the formal group ring.
\end{proof}

\subsection{The abstract period algebra}\label{subsec:model}
Assume that $K/E$ is finite unramified, and let $f$ denote the degree $[K:E]$.
Let $H$ be a one-dimensional formal $\cO_E$-module of finite height over $\cO_K$.
Put $K'=KE_H$, and let $\sigma=\sigma_E\in \Gal(K'/E)$ be the lift of the $q_E$-power Frobenius on the residue field.
Proposition \ref{prop:end_int_closed} implies $\Lambda_H=\cO_{E_H}$, so $T(H)$ is finite free over $\Lambda_H$.
Let $P(X)\in\cO_E[X]$ be the minimal polynomial of $\xi_{q_K}$ over $E$, and set $d=\deg P$.

For $k\ge0$, write $H_k=H^{\sigma^k}$ for the twisted formal group law.
Define the $\Lambda_H$-module structure on $H_k$ by $[\lambda]_{H_k}=([\lambda]_H)^{\sigma^k}$.
Since $K/E$ is unramified, there is an identification $\AinfK\cong W_{\cO_E}(\OCb)$; let $\varphi_E$ be the Witt Frobenius lifting $x\mapsto x^{q_E}$.
It satisfies $\varphi_E^f=\varphi_K$ with $f=[K:E]$.
Then the $k$-th Frobenius iterates of exponential periods define a map
\[j_k:T(H)\to H_k(\AinfK),\quad \alpha\mapsto\varphi_E^k(\tau_H(\alpha)).\]
With the twisted $\Lambda_H$-action on $H_k$, the map $j_k$ is $\Lambda_H$-linear and $\Gal_K$-equivariant.
We define
\[\cA_k\coloneqq \cO_{K'}\llbracket T(H)\rrbracket_{H_k,\Lambda_H}\]
where we use the twisted formal $\Lambda_H$-module law on $H_k$ given by $([\lambda]_H)^{\sigma^k}$.
The formal group ring $\cA_k$ is endowed with the topology defined by the ideal generated by $\pi_K$ and the augmentation ideal.
By the universal property of formal group rings, $j_k$ induces a continuous $\cO_{K'}$-algebra homomorphism
\[\xi_k:\cO_{K'}\llbracket T(H)\rrbracket_{H_k,\Lambda_H}\to\AinfK,\quad [\alpha]\mapsto \varphi_E^k(\tau_H(\alpha))\]
which is $\Gal_K$-equivariant.

We now define an abstract period ring of $H$ by
\[\cA_H\coloneqq\widehat{\bigotimes}_{0\le k\le df-1}\cA_k\]
with the completed tensor product taken over $\cO_{K'}$.
The continuous homomorphisms $\xi_k$ together induce a continuous homomorphism
\[\xi_H=\widehat{\bigotimes}_{0\le k\le df-1}\xi_k:\cA_H\to\AinfK.\]

We now proceed to define an endomorphism $\varphi_E$ on $\cA_H$.
For each $k\ge0$, acting via $\sigma=\sigma_E$ on coefficients and identifying their generators gives an isomorphism
\[\varphi_E:\cA_k\to \cA_{k+1}.\]
It fits into the following commutative diagram
\begin{center}
\begin{tikzcd}
    \cA_k\arrow[r,"\xi_k"]\arrow[d,"\varphi_E"]&\AinfK\arrow[d,"\varphi_E"]\\
    \cA_{k+1}\arrow[r,"\xi_{k+1}"]&\AinfK
\end{tikzcd}
\end{center}

\begin{lemma}\label{lem:form_of_pt}
Let $K/E$ be finite, and let $H$ be a one-dimensional formal $\cO_E$-module of finite height over $\cO_K$.
Put $K'=KE_H$ and $r=[K':K]$.
The minimal polynomial $P(T)$ of $\xi_{q_K}$ over $E$ is of the form $Q(T^r)$ with $Q(T)\in\cO_E[T]$.
\end{lemma}
In the proof below, we do not require that $K/E$ is unramified.
In fact, $K'=KE_H$ is an unramified extension of $K$ and the arithmetic Frobenius $\sigma_K$ acts on $K'$.
\begin{proof}
Let $\xi=\xi_{q_K}$ and $\sigma=\sigma_K\vert_{E_H}$.
Then $\sigma$ has order $r=[KE_H:K]$.
Inside $D=\End_{\Fpbar}(G)\otimes\Qp$, we have $\xi a=\sigma_K(a)\xi$ for every $a\in E_H$.
Write
\[P_i(T)=\sum_{j\equiv i\pmod{r}}c_jT^j,\quad p_i=P_i(\xi).\]
By choosing a primitive generator $a$ of $E_H/E$, the elements $a,\sigma(a),\cdots,\sigma^{r-1}(a)$ are distinct.
Right multiplication by $a$ on $D$ is $E_H$-linear, and satisfies $p_ia=\sigma^i(a)p_i$.
The eigenspaces for these distinct eigenvalues form a direct sum, so the equality $\sum_i p_i=P(\xi)=0$ implies that $p_i=0$ for every $i$.
By minimality of $P$, we conclude $P=P_0$, showing the assertion.
\end{proof}
For $0\le i\le df-1$, we denote the image of $[\alpha]\in\cA_i$ inside $\cA_H$ by $[\alpha^{(i)}]$.

\begin{lemma}
There is a well-defined $\cO_{K'}$-algebra homomorphism
\[\nu:\cA_{df}\to\cA_H,\quad[\alpha]\mapsto\sideset{}{_H}\sum_{\ell=0}^{d-1}[-c_\ell]_H([\alpha^{(f\ell)}]).\]
\end{lemma}
\begin{proof}
Recall that the $\Lambda_H$-module structure on $\cA_k$ is defined using the $\sigma^k$-twisted formal law.
Put $r=[K':K]$.
By Lemma \ref{lem:form_of_pt}, the condition $c_\ell\neq0$ implies $r\mid \ell$, so $f\ell$ is divisible by $fr=[K':E]$, meaning that the group law on any relevant target is untwisted.
Since $r\mid d$, the $\sigma^{fd}$-twisted law on the source is also untwisted.
The assignment in the statement is therefore $\Lambda_H$-linear for the formal group laws on source and target, and the universal property gives $\nu$.
\end{proof}
By Proposition \ref{prop:tau_image}, the exponential period $\tau=\tau_H(\alpha)$ satisfies the equation
\[\varphi_E^{fd}(\tau)=\sideset{}{_H}\sum_{\ell=0}^{d-1}[-c_\ell]_H(\varphi_E^{f\ell}(\tau)).\]
Therefore, the following diagram commutes.
\begin{center}
\begin{tikzcd}
    \cA_{df}\arrow[rd,"\xi_{df}"']\arrow[r,"\nu"]&\cA_H\arrow[d,"\xi_H"]\\
    &\AinfK
\end{tikzcd}
\end{center}
The continuous homomorphisms $\varphi_E:\cA_k\to\cA_{k+1}$, followed by either the canonical inclusion $\cA_{k+1}\to\cA_H$ when $0\le k\le df-2$ or the homomorphism $\nu:\cA_{df}\to\cA_H$ when $k=df-1$, together induce a continuous endomorphism $\varphi_E:\cA_H\to\cA_H$.
By construction, it fits into the following commutative diagram
\begin{center}
    \begin{tikzcd}
    \cA_H\arrow[r,"\varphi_E"]\arrow[d,"\xi_H"]&\cA_H\arrow[d,"\xi_H"]\\
    \AinfK\arrow[r,"\varphi_E"]&\AinfK
    \end{tikzcd}
\end{center}
Moreover, the actions of $\varphi_E$ and $\Gal_K$ on $\cA_H$ commute with each other.
On the generators of $\cA_k$ with $k\le df-2$, both composites send $[\alpha]\mapsto [\gamma\alpha]$ in $\cA_{k+1}$.
At the last factor, commutation follows from the Galois equivariance of $\nu$.

We can work in explicit coordinates of $\cA_H$ by choosing a basis of $T(H)$.
Fix a $\Lambda_H$-basis $\alpha_1,\cdots,\alpha_n$ of $T(H)$.
Then there is an isomorphism
\[\cA_H\cong \cO_{K'}\llbracket t_{k,i}:0\le k\le df-1,1\le i\le n\rrbracket,\quad [\alpha_i^{(k)}]\mapsto t_{k,i}.\]
Every element $\gamma$ of $\Gal_{K_\infty}$ acts trivially on $T(H)$, and semilinearity forces it to fix $\Lambda_H$, and hence $K'=KE_H$.
It follows that $\gamma$ fixes the coefficients and generators of $\cA_H$, so the Galois action on $\cA_H$ factors through $\Gamma=\Gal(K_\infty/K)$.
In the coordinates $t_{k,i}$ fixed above, $\varphi_E$ restricts to $\sigma$ on $\cO_{K'}$ and acts on each variable by
\[\varphi_E(t_{k,i})=t_{k+1,i}\,\,(0\le k\le df-2),\quad \varphi_E(t_{df-1,i})=\sideset{}{_H}\sum_{\ell=0}^{d-1}[-c_\ell]_H(t_{f\ell,i}).\]
Fix $\gamma\in\Gamma$ and $1\le i\le n$, write $\gamma\alpha_i=c_1\alpha_1+\cdots+c_n\alpha_n$, and set 
\[F_{\gamma,i}(X_1,\cdots,X_n)=[c_1]_H(X_1)+_H\cdots+_H[c_n]_H(X_n).\]
With this notation, $\gamma$ acts on the variable $t_{k,i}$ by
\[\gamma(t_{k,i})=(F_{\gamma,i})^{\sigma^k}(t_{k,1},\cdots,t_{k,n}).\]
The continuous $\cO_{K'}$-algebra homomorphism $\xi_H$ identifies with
\[\cA_{H}\to\AinfK,\quad t_{k,i}\mapsto \varphi_E^k(\tau_H(\alpha_i)).\]

Let $\overline{\xi}_H:\cA_H\to\OCb$ be the reduction of $\xi_H$ modulo $\pi_K$, and put $o_H=\mathrm{im}(\overline{\xi}_H)$ and $\frp_H=\ker(\overline{\xi}_H)$.
Since $\OCb$ is a domain, $\frp_H$ is a prime ideal.
Set
\[R_H=((\cA_{H})_{\frp_H})^\wedge=\varprojlim_r (\cA_{H})_{\frp_H}/\frp_H^r(\cA_{H})_{\frp_H}.\]
Let $k_H=\Frac(o_H)\subset\Cb$.
The ring $R_H$ is a complete local ring with residue field $k_H$.
Since $\xi_H$ sends $\frp_H$ into $(\pi_K)$ and every element outside $\frp_H$ to a unit of $W_{\cO_K}(\Cb)$, it extends uniquely to a local homomorphism $R_H\to W_{\cO_K}(\Cb)$.
Moreover, $\varphi_E^{-1}(\frp_H)=\frp_H$ and $\Gal_K$ preserves $\frp_H$, so these actions extend to $R_H$.
The resulting commutative diagram is equivariant for $\varphi_E$ and $\Gal_K$.
\begin{center}
    \begin{tikzcd}
        R_H\arrow[r]\arrow[d]&W_{\cO_K}(\Cb)\arrow[d]\\
        k_H\arrow[r]&\Cb
    \end{tikzcd}
\end{center}
The residue field $k_H$ of $R_H$ is defined directly from the reduced exponential periods.
In the next section, we relate the absolute Galois groups of $k_H$ and $K_\infty$.

\newpage
\section{Formal groups and perfectoid fields}\label{sec:perfectoid}
In this section, we first prove that the completed torsion field $\widehat{K}_\infty$ is perfectoid and that its tilt is the completed perfection of the reduced period field $k_H$.
We then identify $\Gal_{K_\infty}$ with the valuation decomposition subgroup of $\Gal_{k_H}$; passing to the corresponding henselization gives the required isomorphism between absolute Galois groups.

\subsection{Perfectoidness of the torsion tower}
Let $K/E$ be a finite unramified extension of $p$-adic local fields.
Write $\kappa_K=\Fq$, and choose a uniformizer $\pi$ of $E$, which is also a uniformizer of $K$.

\begin{definition}[{\cite[Definition 3.1]{Scholze2012Perfectoid}}]
A \textit{perfectoid field} is a complete nonarchimedean field $F$ of residue characteristic $p$, equipped with a nondiscrete rank one valuation, such that the Frobenius is surjective on $\cO_F/(p)$.
\end{definition}

\begin{proposition}\label{prop:perfectoid_formal_group}
Let $K/E$ be finite unramified, and let $H$ be a one-dimensional formal $\cO_E$-module of finite height over $\cO_K$.
Set $K_\infty=K(H[p^\infty](\overline{K}))$, and let $\widehat{K}_\infty$ be its $p$-adic completion.
Then $\widehat{K}_\infty$ is a perfectoid field.
\end{proposition}
\begin{lemma}\label{lem:perfectoid_varpi}
Let $K/E$ be finite unramified, and let $H$ be a one-dimensional formal $\cO_E$-module of finite height over $\cO_K$.
Let $\alpha=(\alpha_i)_{i\ge0}$ be a nonzero element of $T(H)$.
Set $K_{\alpha}=\bigcup_{n=1}^\infty K(\alpha_n)$ and let $\widehat{K_\alpha}$ denote its completion.
The extension $K_\alpha/K$ is infinite and totally ramified, and its completion $\widehat{K_\alpha}$ is a perfectoid field.
\end{lemma}

\begin{proof}
Write $p=u\pi^e$ with $u\in\cO_E^\times$.
We can identify the inverse limits
\[\varprojlim_{[\pi]_H}\frakm_C\cong \varprojlim_{[p]_H}\frakm_C=\widetilde{H}(\OC),\quad (\beta_i)_{i\ge0}\mapsto (\alpha_i=[u^{-i}]_H(\beta_{ei}))_{i\ge0}.\] 
Hence, we may prove the assertion using a sequence $\beta=(\beta_i)_{i\ge0}$ satisfying $\beta_0=0$, $\beta_1\neq0$, and $[\pi]_H(\beta_{i+1})=\beta_i$.
Let $G$ be the special fiber of $H$ over $\kappa_K$, and  set $Q=q_E^{\mathrm{ht}_{\cO_E}(H)}$, the degree of the first nonzero term of $[\pi]_G$.
We first show that $K(\beta_1)/K$ and every extension $K(\beta_{i+1})/K(\beta_i)$ are totally ramified, with $\beta_i$ a uniformizer of $K(\beta_i)$.
\begin{enumerate}
    \item By Weierstrass preparation, $[\pi]_H(X)/X$ is the product of a unit series and an Eisenstein polynomial of degree $Q-1$ over $\cO_K$.
    Hence, $K(\beta_1)/K$ is totally ramified and $\beta_1$ is a uniformizer.
    \item For $i\ge1$, the point $\beta_{i+1}$ is a root of $[\pi]_H(X)-\beta_i$.
    Since $\beta_i$ is a uniformizer of $K(\beta_i)$, Weierstrass preparation factors this series as a unit times an Eisenstein polynomial of degree $Q$ over $K(\beta_i)$.
    Therefore, $K(\beta_{i+1})/K(\beta_i)$ is totally ramified of degree $Q$, and $\beta_{i+1}$ is a uniformizer of $K(\beta_{i+1})$.
\end{enumerate}
The $\pi$-division tower therefore has unbounded ramification index, and its union is $K_{\alpha}$.
We conclude that $K_\alpha/K$ is infinite and totally ramified.

The ring of integers of $K_{\alpha}$ is $A_\alpha=\bigcup_i \cO_K[\beta_i]$ because each $\beta_i$ is a uniformizer and generates the totally ramified extension $K(\beta_i)/K$.
Since $[\pi]_G$ has first nonzero term of degree $Q$, it has the form $[\pi]_G(X)=g(X^Q)$ with $g\in\kappa_K\llbracket X\rrbracket$.
Let $\sigma$ act on coefficients by $Q$-power Frobenius; then 
\[\overline{\beta}_i= (g^{\sigma^{-1}}([\overline{\beta}_{i+1}]))^Q\]
in $A_\alpha/(\pi)$.
Therefore, each $\overline{\beta}_i$ is a $p$-th power in $A_\alpha/(\pi_K)$, so Frobenius is surjective on $\cO_{\widehat{K}_\alpha}/(\pi)$.
Since its valuation is nondiscrete, choose a pseudouniformizer $\varpi$ such that $0<v(\varpi)<v(\pi)/p$.
Surjectivity modulo $\pi$ implies surjectivity of $\Phi:\cO_{\widehat{K}_\alpha}/(\varpi)\to \cO_{\widehat{K}_\alpha}/(\varpi^p)$.
Since $\varpi^p\mid p$, \cite[Remark 3.2]{Scholze2026Etale} gives surjectivity of Frobenius modulo $p$, showing that $\widehat{K}_\alpha$ is perfectoid.
\end{proof}
\begin{proof}[Proof of Proposition \ref{prop:perfectoid_formal_group}]
Choose a nonzero element $\alpha\in T(H)$.
By Lemma \ref{lem:perfectoid_varpi}, $\widehat{K}_\alpha$ is perfectoid.
The inclusion $\widehat{K}_\alpha\subset \widehat{K}_\infty$ shows that the valuation on $\widehat{K}_\infty$ is nondiscrete.
Each finite composite $\widehat{K}_\alpha L$, with $L/K_\alpha$ finite inside $K_\infty$, is perfectoid by \cite[Theorem 3.7 (i)]{Scholze2012Perfectoid}.
Their union is dense inside $\widehat{K}_\infty$, so Frobenius is surjective on $\cO_{\widehat{K}_\infty}/(p)$.
Therefore, $\widehat{K}_\infty$ is perfectoid.
\end{proof}

Tilting gives an equivalence between perfectoid algebras over $\widehat{K}_\infty$ and perfectoid algebras over $\widehat{K}_\infty^\flat$ \cite[Theorem 5.2]{Scholze2012Perfectoid}.
In particular, tilting identifies their finite separable extensions and hence their absolute Galois groups.

\begin{remark}
Set $\widetilde{\AA}_H=W_{\cO_E}(\widehat{K}_\infty^\flat)$.
The Witt ring carries its weak topology induced by the valuation topology on $\widehat{K}_\infty^\flat$.
The tilting equivalence, the equivalence for \'{e}tale Frobenius modules over ramified Witt vectors of a perfect field $\widehat{K}_\infty^\flat$, and Galois descent give an equivalence of categories
\[\Rep_{\cO_E}(\Gal_K)\simeq\textup{$(\varphi_E,\Gamma)$-Mod}_{\widetilde{\AA}_H}\]
where the right-hand side consists of finitely generated \'{e}tale $\varphi_E$-modules with a commuting continuous $\Gamma$-action.
Therefore, perfectoidness already supplies an equivalence over a ring with perfect residue field.
The construction below produces an analogous equivalence over a base ring with imperfect residue field, using the exponential periods of $H$.
\end{remark}

\subsection{The period field $k_H$ is a deperfection}
Reducing the exponential period map modulo $\pi_K$ gives an injective $\Lambda_H$-linear homomorphism
\[\overline{\tau}_{H}:T(H)\to \OCb\]
where the target is equipped with the formal $\Lambda_H$-module law of the special fiber $G$.
Since the reduced exponential periods are topologically nilpotent, the universal property of the formal group rings gives a continuous homomorphism
\[\overline{\xi}:\kappa_{K'}\llbracket T(H) \rrbracket _{H,\Lambda_H}\to \OCb.\]
The image of $\overline{\xi}$ is the ring $o_H$ defined in Section \ref{subsec:model}.
In fact, the reductions of the additional generators $\varphi_E^k(\tau_H(\alpha))$ are the powers $\overline{\tau}_H(\alpha)^{q_E^k}$.
Recall the notation $k_H=\Frac(o_H)$.

\begin{lemma}\label{lem:kh_stability}
\begin{enumerate}
    \item The subfield $k_H\subset\Cb$ depends only on the isomorphism class of $H$ over $\cO_K$, and in particular is independent of the chosen coordinate.
    \item The subfield $k_H$ is stable under the action of $\Gal_K$ on $\Cb$.
\end{enumerate}
\end{lemma}
\begin{proof}
(i) 
A morphism $f:H\to H'$ over $\cO_K$ induces a morphism between their universal covers.
If $f$ is an isomorphism, let $\alpha\mapsto\alpha'$ under the induced map $T(H)\to T(H')$; then functoriality implies
\[\overline{\tau}_{H'}(\alpha')=\overline{f}(\overline{\tau}_{H}(\alpha)).\]
Since $\overline{f}$ has coefficients in $\kappa_K$, the inclusion $k_{H'}\subset k_H$ follows.
Applying the same argument to $f^{-1}$ gives the reverse inclusion, so $k_H=k_{H'}$.

(ii)
Galois equivariance of the period map gives
\[\gamma\cdot\overline{\tau}_{H}(\alpha)=\overline{\tau}_{H}(\gamma\cdot\alpha)\]
for all $\gamma\in\Gal_K$ and $\alpha\in T(H)$.
Continuity and the equivariance on generators show that $\mathrm{im}(\overline{\xi})$ is $\Gal_K$-stable, and therefore so is its fraction field.
\end{proof}

The reduction modulo $\pi_K$ of the exponential period map $\widetilde{\tau}_H$ coincides with the composite
\[\widetilde{H}(\cO_C)\simeq\widetilde{G}(\OCb)\xrightarrow{\pr}G(\OCb)\]
by rigidity.
For $\alpha\in V(H)$, we denote by $\alpha^\flat$ its image.
\begin{proposition}\label{prop:deperf_infty}
Assume that $K/E$ is unramified.
Let $H$ be a one-dimensional formal $\cO_E$-module of finite height over $\cO_K$.
Then 
\begin{enumerate}
    \item The perfection $k_H^\perf$ is dense in $\widehat{K}_\infty^\flat$ for the valuation induced from $\Cb$.
    \item The field $\widehat{K}_\infty^\flat$ is the closure in $\Cb$ of the compositum of the fields $\kappa_{K}(\!(\alpha^\flat)\!)$ where $\alpha$ ranges over the nonzero elements of $V(H)$.
\end{enumerate}
\end{proposition}

\begin{remark}\label{rem:OE_mod_assum_is_necessary}
The unramified hypothesis in Proposition \ref{prop:deperf_infty} cannot be omitted in general.
Consider $H=\Gmhat$ with $E=\Qp$.
For $K=\Qp$, the torsion tower is $\QQ_{p,\infty}=\Qp(\zeta_{p^\infty})$.
The calculation in Example \ref{ex:multiplicative} gives $k_H=\Fp(\!(\epsilon-1)\!)$, and Proposition \ref{prop:deperf_infty} (i) identifies its completed perfection with $\widehat{\QQ}_{p,\infty}^\flat$.
Let $L/\Qp$ be finite and totally ramified such that $L\not\subset\QQ_{p,\infty}$.
The extension $L_\infty=L(\zeta_{p^\infty})$ is a proper finite extension of $\QQ_{p,\infty}$.
Completion and tilting preserve its degree, so $\widehat{L}_\infty^\flat$ strictly contains $\widehat{\QQ}_{p,\infty}^\flat$.
Since $L$ is totally ramified, its residue field is still $\Fp$, and the reduced multiplicative period remains $\epsilon-1$.
Hence, the period field over $L$ is again $k_H'=\Fp(\!(\epsilon-1)\!)$.
Consequently,
\[((k_H')^\perf)^\wedge=\widehat{\QQ_{p,\infty}}^\flat\subsetneqq\widehat{L}_\infty^\flat\]
so the density assertion fails over $L$.
\end{remark}

We prove Proposition \ref{prop:deperf_infty} by first establishing the corresponding density statement for each division tower $K_\alpha$ with $\alpha\in T(H)\setminus\{0\}$.
Write $\alpha^\flat=\overline{\tau}_H(\alpha)$, and let $o_\alpha$ be the image of the evaluation map $\kappa_{K}\llbracket T\rrbracket\to\OCb$ sending $T$ to $\alpha^\flat$.
Set $k_\alpha=\Frac(o_\alpha)=\kappa_{K}(\!(\alpha^\flat)\!)$.
\begin{lemma}\label{lem:deperf_alpha}
Let $K/E$ be finite unramified, let $H$ be a one-dimensional formal $\cO_E$-module of finite height over $\cO_K$, and let $\alpha\in T(H)$ be nonzero.
Then 
\begin{enumerate}
    \item The perfection $(k_\alpha)^\perf$ is dense in $\widehat{K_\alpha}^\flat$.
    \item  For $u_n=p^{-n}\alpha\in V(H)$, let $u_n^\flat\in\OCb$ denote its image under the universal cover identification.
    Then 
    \[(k_\alpha)^\perf=\bigcup_{n}\kappa_{K}(\!(u_n^\flat)\!).\]
\end{enumerate}
\end{lemma}
\begin{proof}
Evaluation at the nonzero topologically nilpotent element $\alpha^\flat$ identifies $o_\alpha$ with $\kappa_{K}\llbracket\alpha^\flat\rrbracket$.
If $h_0$ is the height of $G$, then $[p]_G(T)=f(T^{p^{h_0}})$ for some $f\in \kappa_{K}\llbracket T \rrbracket_0$ with $f'(0)\neq0$.
The relation
\[u_n^\flat=f((u_{n+1}^\flat)^{p^{h_0}})\]
gives
\[\kappa_{K}\llbracket u_n^\flat\rrbracket=\kappa_{K}\llbracket (u_{n+1}^\flat)^{p^{h_0}}\rrbracket\]
so
\[(o_\alpha)^\perf=\bigcup_{n=0}^\infty\kappa_{K}\llbracket u_n^\flat\rrbracket.\]
Taking fraction fields proves (ii).

We now prove (i).
Let $A_\alpha=\cO_{K_{\alpha}}$ and $\widehat{A}_\alpha=\cO_{\widehat{K_\alpha}}$.
Since completion does not change the mod $\pi_K$ quotient, we have
\[\widehat{A}_\alpha^\flat\coloneqq \varprojlim_{x\mapsto x^p}A_\alpha/(\pi).\]
The inclusion $o_\alpha^\perf\subset\OCb$ identifies it as a subring of $\widehat{A}_\alpha^\flat$.
It suffices to prove that this subring is dense.
The inverse limit topology is defined by the discrete quotients $A_\alpha/(\pi)$, so it is enough to show that each projection
\[\pr_i:o_\alpha^\perf\to A_\alpha/(\pi)\]
is surjective.
Since $\pr_i=\pr_0\circ\Frob_p^{-i}$ and Frobenius is an automorphism of $o_\alpha^\perf$, it suffices to prove surjectivity of $\pr_0$.
Since $A_\alpha/(\pi_K)$ is generated over $\kappa_{K}$ by the reductions of the nonzero division coordinates $\alpha_m$, it is enough to lift each $\overline{\alpha}_m$ through $\pr_0$.
Proposition \ref{prop:explicit_isom_tilG_to_G} identifies the zeroth tilt projection with the reduced zeroth universal cover coordinate, so
\[\pr_0(u_m^\flat)=\overline{\alpha}_m.\]
Hence, $o_\alpha^\perf$ is dense inside $\widehat{A}_\alpha^\flat$, proving (i).
\end{proof}

\begin{proof}[Proof of Proposition \ref{prop:deperf_infty}]
Let $\ell$ be the closure of $k_H^\perf$ in $\Cb$.
Every reduced exponential period lies in $\widehat{K}_\infty^\flat$ by Proposition \ref{prop:explicit_isom_tilG_to_G}.
Therefore, the power series defining $k_H$, its perfection $k_H^\perf$, and the closure $\ell$ all lie in $\widehat{K}_\infty^\flat$; alternatively, $\Gal_{K_\infty}$ acts trivially on $\ell$, so the identity $(\Cb)^{\Gal_{K_\infty}}=\widehat{K}_\infty^\flat$ implies the same inclusion.
Fix a nonzero element $\alpha\in T(H)$.
The field $\ell$ is perfectoid and contains $\widehat{K}_\alpha^\flat$, which is also perfectoid by Lemma \ref{lem:deperf_alpha}.
The tilting equivalence applies to the inclusion $\ell\subset\Cb$ over $\widehat{K}_\alpha$, giving its untilt inside $C$
\[\ell^\#=\mathrm{Im}(\theta_{\Cb}:W(\ell^\circ)\to \OC)\left[\frac{1}{p}\right].\]
This formula defines a complete subfield $\ell^\#\subset C$ independent of $\alpha$, and $\ell^\#$ contains $\widehat{K}_\alpha$ for every such $\alpha$.
The fields $K_\alpha$ generate $K_\infty$, so completeness therefore gives $\widehat{K}_\infty\subset\ell^\#$, and tilting yields $\widehat{K}_\infty^\flat\subset\ell$.
Therefore, $\ell=\widehat{K}_\infty^\flat$, proving (i).

Let $m$ be the closure of the compositum of the fields $\kappa_{K}(\!(\alpha^\flat)\!)$ where $\alpha$ ranges over the nonzero elements of $V(H)$.
By Lemma \ref{lem:deperf_alpha}, $m$ contains every $\widehat{K}_\alpha^\flat$, so the same untilting argument gives $m=\widehat{K}_\infty^\flat$.
\end{proof}
\begin{corollary}\label{cor:deperf_cb}
Let $k_H^\alg$ denote the algebraic closure of $k_H$ inside $\Cb$.
Then $k_H^\alg$ is dense in $\Cb$.
\end{corollary}
\begin{proof}
Let $\ell$ be the closure of $k_H^\alg$ in $\Cb$.
The completion of an algebraically closed rank-one valued field is algebraically closed, so $\ell$ is algebraically closed.
Proposition \ref{prop:deperf_infty} gives $\widehat{K}_\infty^\flat\subset\ell\subset \Cb$.
Let $\ell^\sharp\subset C$ be the untilt corresponding to $\ell$ under tilting over $\widehat{K}_\infty$; then $\widehat{K}_\infty\subset\ell^\#\subset C$.
By \cite[Proposition 3.8]{Scholze2012Perfectoid}, $\ell^\#$ is algebraically closed.
Since $\ell^\#$ contains $K$, it contains $\overline{K}$, and completeness gives $\ell^\#=C$.
Hence, $\ell=\Cb$.
\end{proof}

\subsection{Comparison of the Galois groups of $K_\infty$ and $k_H$}
By Lemma \ref{lem:kh_stability}, $k_H$ is stable under $\Gal_K$.
Since $\Gal_{K_\infty}$ fixes every reduced exponential period and the coefficient field $\kappa_{K'}$, its action on $k_H$ is trivial, so the $\Gal_K$-action factors through $\Gamma=\Gal(K_\infty/K)$.
By fixing $k_H^\sep\subset\Cb$, restriction of the Galois action defines a homomorphism $\Gal_{K_\infty}\to\Gal_{k_H}$.
We identify the image of this homomorphism as a decomposition subgroup; we refer the reader to \cite[Section 5.2]{Engler2005Valueda} for basic facts about valued fields and decomposition groups.

\begin{theorem}\label{thm:decomposition}
Let $K/E$ be a finite unramified extension of $p$-adic local fields, and let $H$ be a one-dimensional formal $\cO_E$-module of finite height over $\cO_K$.
Let $w$ be the restriction of $v_{\Cb}$ to $k_H^\sep$.
Restriction induces an isomorphism of profinite groups
\[\Gal_{K_\infty}\xrightarrow{\sim} D_w\subset\Gal_{k_H}\]
\end{theorem}
\begin{proof}
Write $w^\alg$ for the restriction of $v_{\Cb}$ to the algebraic closure $k_H^\alg\subset\Cb$.
Uniqueness of extensions of a valuation along a purely inseparable extension identifies $D_{w^\alg}$ with $D_w$ under $\Gal(k_H^\alg/k_H^\perf)\cong \Gal(k_H^\sep/k_H)$.
Therefore, it suffices to identify the image of $\Gal_{K_\infty}\to\Gal(k_H^\alg/k_H^\perf)$ with $D_{w^\alg}$.

Since $K_\infty$ is henselian as an algebraic extension of the local field $K$, we have
\[\Gal_{K_\infty}\cong\Aut_\cts(C/\widehat{K}_\infty).\]
The tilting equivalence identifies $\Aut_\cts(C/\widehat{K}_\infty)$ with $\Aut_\cts(\Cb/\widehat{K}_\infty^\flat)$.
Restriction to $k_H^\alg$ induces an injective group homomorphism
\begin{equation*}
    \Aut_\cts(\Cb/\widehat{K}_\infty^\flat)\hookrightarrow \Gal(k_H^\alg/k_H^\perf)
\end{equation*}
Proposition \ref{prop:deperf_infty} and Corollary \ref{cor:deperf_cb} identify $\widehat{K}_\infty^\flat$ and $\Cb$ as the completions of $k_H^\perf$ and $k_H^\alg$ respectively, so restriction gives an isomorphism
\begin{equation*}
    \Aut_\cts(\Cb/\widehat{K}_\infty^\flat)\isomto D_{w^\alg}\subset \Gal(k_H^\alg/k_H^\perf)
\end{equation*}
where the inverse sends an element of $D_{w^\alg}$ to a unique extension to an isometry of $\Cb$ fixing $\widehat{K}_\infty^\flat$.
The resulting restriction map from $\Gal_{K_\infty}$ is continuous.
Since its source is compact and its target is Hausdorff, it is an isomorphism of topological groups.
\end{proof}

\begin{corollary}\label{cor:fullimage_equiv_hens}
Assume the hypotheses of Theorem \ref{thm:decomposition}.
The following two conditions are equivalent.
\begin{enumerate}
    \item The injective homomorphism $\Gal_{K_\infty}\to\Gal_{k_H}$ is an isomorphism of profinite groups.
    \item The valuation on $k_H$ induced by $\Cb$ is henselian.
\end{enumerate}
\end{corollary}
After determining the structure of $k_H$ in Section \ref{sec:period_alg}, we will show that these conditions hold if and only if the absolute height of $H$ is one.
\begin{proof}
By Theorem \ref{thm:decomposition}, the condition (i) is equivalent to $D_w=\Gal_{k_H}$.
The fixed field of $D_w$ is the chosen henselization of $k_H$, which equals $k_H$ exactly when the valuation is henselian.
\end{proof}

We now construct a henselian deperfection $k_{H,K}\subset \widehat{K}_\infty^\flat$ whose absolute Galois group is $\Gal_{K_\infty}$.
\begin{definition}
We define
\[k_{H,K}=(k_H^\sep)^{D_w}\]
where $w$ is the restriction of $v_{\Cb}$, so $k_{H,K}$ is the henselization of $k_H$ associated with $w$.
\end{definition}
The inclusion $k_H\subset\widehat{K}_\infty^\flat$ of valued fields, with the latter henselian, induces the inclusion $k_{H,K}\subset \widehat{K}_\infty^\flat$ by the universal property of henselization.
\begin{corollary}
Assume the hypotheses of Theorem \ref{thm:decomposition}, and fix a separable closure $k_H^\sep$ in $\Cb$.
Then we have $k_{H,K}=k_H^\sep\cap \widehat{K}_\infty^\flat$.
\end{corollary}
\begin{proof}
By Theorem \ref{thm:decomposition}, restriction of the action of $\Gal_{K_\infty}$ to the subfield $k_H^\sep\subset \Cb$ identifies with $D_w$.
Therefore,
\[k_{H,K}=(\Cb)^{\Gal_{K_\infty}}\cap k_H^\sep.\]
Ax-Sen-Tate's theorem implies $(\Cb)^{\Gal_{K_\infty}}=\widehat{K}_\infty^\flat$, so the assertion follows.
\end{proof}

Since $\widehat{K}_\infty^\flat$ is complete and hence henselian, the chosen henselization  $k_{H,K}$ embeds in $\widehat{K}_\infty^\flat$.

\begin{proposition}\label{prop:deperfection}
The field $k_{H,K}$ has the following properties:
\begin{enumerate}
    \item The $\Gal_K$-action on $\Cb$ preserves $k_{H,K}$, and induces an action of $\Gamma$.
    \item $k_{H,K}$ is a subfield of $\widehat{K}_\infty^\flat$, and $k_{H,K}^\perf$ is dense in $\widehat{K}_\infty^\flat$.
    \item Restriction gives a canonical isomorphism of profinite groups $\Gal_{k_{H,K}}\cong\Gal_{K_\infty}$.
\end{enumerate}
\end{proposition}
\begin{proof}
(i)
Let $\gamma\in\Gal_K$, viewed through its induced action on $k_H^\sep$.
The fixed field definition gives
\[\gamma(k_{H,K})=(k_H^\sep)^{\gamma D_w\gamma^{-1}}.\]
Conjugation gives $\gamma D_w\gamma^{-1}=D_{w\circ\gamma^{-1}}$.
The $\Gal_K$-action preserves $v_{\Cb}$, so $w\circ\gamma^{-1}=w$.
Therefore, we have
\[\gamma(k_{H,K})=(k_H^\sep)^{D_w}=k_{H,K}.\]
By Theorem \ref{thm:decomposition}, the $\Gal_{K_\infty}$-action induces an element of $D_w$, which fixes $k_{H,K}$ by definition, so the action factors through $\Gamma$.

(ii)
We already checked the inclusion $k_{H,K}\subset \widehat{K}_\infty^\flat$.
The inclusions 
\[k_H^\perf\subset k_{H,K}^\perf\subset\widehat{K}_\infty^\flat\]
and Proposition \ref{prop:deperf_infty} prove the density assertion.

(iii)
The decomposition subgroup $D_w$ is closed \cite[Lemma 5.2.1]{Engler2005Valueda}, so infinite Galois theory identifies $\Gal_{k_{H,K}}$ with $D_w$, which is further identified with $\Gal_{K_\infty}$ by Theorem \ref{thm:decomposition}.
\end{proof}

Let $L/K$ be a finite extension.
Let $L_\infty=K_\infty L$.
Its $p$-adic completion $\widehat{L}_\infty$ is a finite extension of the perfectoid field $\widehat{K}_\infty$, so is perfectoid.
Using the identification $\Gal_{K_\infty}\cong D_w$, we view $\Gal_{L_\infty}$ as a subgroup of $D_w$ and define 
\[k_{H,L}=(k_H^\sep)^{\Gal_{L_\infty}},\]
which is a finite separable extension of $k_{H,K}$ inside $\Cb$.
\begin{proposition}\label{prop:lHL}
Let $L/K$ be finite.
The field $k_{H,L}$ has the following properties.
\begin{enumerate}
    \item The $\Gal_L$-action preserves $k_{H,L}$ and factors through $\Gamma_L=\Gal(L_\infty/L)$.
    \item The field $k_{H,L}$ is a subfield of $\widehat{L}_\infty^\flat$, and its perfection is dense.
    \item Restriction gives an isomorphism of profinite groups $\Gal_{k_{H,L}}\cong\Gal_{L_\infty}$.
\end{enumerate}
\end{proposition}
\begin{proof}
With the chosen embeddings, completion of perfection identifies finite separable extensions $\ell/k_{H,K}$ inside $k_H^\sep$ with finite extensions of $\widehat{K}_\infty^\flat$ inside $\Cb$, sending $\ell$ to $\widehat{\ell^\perf}$.
Since the $\Gal_K$-action preserves the valuation, it commutes with perfection and completion, so this correspondence is $\Gal_K$-equivariant.
The field $k_{H,L}$ corresponds to $\widehat{L}_\infty^\flat$, so $\widehat{k_{H,L}^\perf}=\widehat{L}_\infty^\flat$, proving (ii).
Equivariance implies that $\Gal_L$ preserves $k_{H,L}$, and the fixed field definition shows that $\Gal_{L_\infty}$ acts trivially on it.
Hence, the action factors through $\Gamma_L$, proving (i).
The correspondence between finite extensions of $K_\infty$ and finite extensions of $\widehat{K}_\infty^\flat$ sends $L_\infty$ first to its completion and then to $\widehat{L}_\infty^\flat$.
Under the isomorphism $\Gal_{k_{H,K}}\cong \Gal_{K_\infty}$, the open subgroup fixing $k_{H,L}$ is $\Gal_{L_\infty}$, which proves (iii).
\end{proof}

\newpage

\section{$F$-dynamical systems and \'{e}tale $\varphi$-modules}\label{sec:fds}
The endomorphism $\varphi_E$ on $R_H$ constructed in Section \ref{subsec:model} need not lift the $q_E$-power Frobenius modulo $\pi_E$.
In this section, we introduce $F$-dynamical systems, replacing this condition with the weaker contracting condition modulo $\pi_E$ together with a Frobenius action on the residue field.
For a flat $F$-dynamical system, we prove an equivalence between continuous $\cO_E$-representations of the absolute Galois group of a residue field and \'{e}tale $\varphi$-modules over a flat $F$-dynamical system.
We first treat contracting systems in characteristic $p$, then lift the equivalence over $\cO_E$.
Finally, we verify the required hypotheses for the pair $(R_H,\varphi_E)$.
\subsection{\'{E}tale $\varphi$-modules over fields of characteristic $p$}
Let $R$ be a ring and let $\varphi$ be a ring endomorphism of $R$.
A \textit{$\varphi$-module} over $R$ is a finitely generated $R$-module $M$ equipped with an additive endomorphism $\varphi_M$ satisfying $\varphi_M(rm)=\varphi(r)\varphi_M(m)$.
Equivalently, it is equipped with an $R$-linear map
\[\varphi_M^L:\varphi^*M\coloneqq M\otimes_{R,\varphi}R\to M,\quad m\otimes r\mapsto r\varphi_M(m)\]
We call $M$ \textit{\'{e}tale} if its linearization $\varphi_M^L$ is an isomorphism.
A morphism of $\varphi$-modules is an $R$-linear map $f:M\to N$ satisfying $f\circ\varphi_M=\varphi_N\circ f$.
We write $\phimod_R$ for the category of \'{e}tale $\varphi$-modules over $R$.

Let $k$ be a field containing $\Fq$, fix a separable closure $k^\sep$, and write $\Gal_k=\Gal(k^\sep/k)$.
Throughout this subsection, $\Rep_{\Fq}(\Gal_k)$ denotes the category of finite-dimensional $\Fq$-vector spaces with continuous $\Gal_k$-action.
For $V\in \Rep_{\Fq}(\Gal_k)$, define
\[\DD_k(V)\coloneqq(V\otimes_{\Fq}k^\sep)^{\Gal_k}\]
where $\Gal_k$ acts diagonally, and equip $\DD_k(V)$ with the semilinear endomorphism induced by $\id_V\otimes\Frob_q$.
Galois descent gives the canonical isomorphism
\begin{equation}\label{isom:functor_D}
    \DD_k(V)\otimes_kk^\sep\cong V\otimes_{\Fq}k^\sep.
\end{equation}
Since the linearization of $\id_V\otimes\Frob_q$ is invertible, (\ref{isom:functor_D}) shows by faithfully flat descent that $\DD_k(V)$ is \'{e}tale.
Therefore, we obtain a functor
\[\DD_k:\Rep_{\Fq}(\Gal_k)\to\textup{$\varphi$-Mod}_{k}\]
where the endomorphism of $k$ is $\Frob_q$.

The original representation $V$ can be recovered from $\DD_k(V)$.
Taking $\varphi$-fixed vectors in (\ref{isom:functor_D}), and using $(k^\sep)^{\Frob_q=\id}=\Fq$, gives
\begin{align*}
    (\DD_k(V)\otimes k^\sep)^{\varphi=\id}
    \cong (V\otimes_{\Fq}k^\sep)^{\varphi=\id}
    \cong V.
\end{align*}

For $M\in \textup{$\varphi$-Mod}_{k}$, define
\[\VV_k(M)\coloneqq(M\otimes_kk^\sep)^{\varphi=\id}\]
with the $\Gal_k$-action induced by its action on $k^\sep$.
\begin{theorem}[{\cite[Proposition 4.1.1]{Katz1973padic}}]\label{thm:equiv_phimod_field}
Let $k$ be a field containing $\Fq$.
The functors
\[\DD_k:\Rep_{\Fq}(\Gal_k)\to\textup{$\varphi$-Mod}_{k},\quad \VV_k:\textup{$\varphi$-Mod}_{k}\to\Rep_{\Fq}(\Gal_k)\]
are mutually quasi-inverse equivalences.
\end{theorem}
The standard proof uses the following statement (see the proof of \cite[Proposition 1.2.6]{Fontaine1990Representations}).
\begin{lemma}\label{lem:functor_V_field}
Let $k$ be a field containing $\Fq$.
For every $M\in\textup{$\varphi$-Mod}_{k}$, multiplication induces an isomorphism
\[\VV_k(M)\otimes_{\Fq}k^\sep\cong M\otimes_kk^\sep.\]
\end{lemma}

\subsection{Contracting systems in characteristic $p$}
\begin{definition}
Let $I$ be an ideal of a ring $R$.
An endomorphism $\phi:R\to R$ preserving $I$ is called \textit{$I$-adically contracting} if $\phi^N(I)\subset I^2$ for some $N\ge1$.
For a local ring $(R,\frakm_R)$, we call $\phi$ \textit{contracting} if it is $\frakm_R$-adically contracting.
\end{definition}

Contraction forces every \'{e}tale $\varphi$-module to be free.

\begin{lemma}\label{lem:structure_thm_contracting}
Let $Q$ be a complete Noetherian local ring equipped with a contracting endomorphism $\varphi$.
Every \'{e}tale $\varphi$-module over $Q$ is finite free as a $Q$-module.
\end{lemma}
The proof uses Fitting ideals (compare \cite[Proposition 1.2.3]{Emerton2004introduction} and \cite[Corollary 4.9]{Du2025Multivariable}).
\begin{proof}
Fix $M\in\textup{$\varphi$-Mod}_{Q}$.
Let $\Fit_i(M)$ denote the $i$-th Fitting ideal of $M$, and write $\Fit_{-1}(M)=0$.
\'{E}taleness and compatibility of Fitting ideals with base change give
\[\Fit_i(M)=\varphi(\Fit_i(M))Q.\]
If $J=\Fit_i(M)$ is a proper ideal, then $J\subset\frakm_Q$, and for every $r\ge1$
\[J=\varphi^r(J)Q\subset\varphi^r(\frakm_Q)Q.\]
Choose $N$ with $\varphi^N(\frakm_Q)\subset\frakm_Q^2$.
Then $J\subset\frakm_Q^{2^s}$ for every $s\ge0$, so $J=0$ by $\frakm_Q$-adic separatedness.
Therefore, every Fitting ideal of $M$ is either 0 or $Q$.
Let $r$ be the least integer with $\Fit_r(M)=Q$; then $\Fit_{r-1}(M)=0$, so \cite[\href{https://stacks.math.columbia.edu/tag/07ZD}{Tag 07ZD}]{stacks-project} shows that $M$ is locally free of rank $r$, hence free because $Q$ is local.
\end{proof}

\begin{definition}
A \textit{contracting system over $\Fq$} is a complete Noetherian local $\Fq$-algebra $Q$ equipped with a contracting $\Fq$-algebra endomorphism $\varphi$ inducing $\Frob_q$ on its residue field.
\end{definition}
\begin{remark}
Contracting endomorphisms also occur in the study of local algebraic dynamics; see \cite{Avramov2006Homology} and \cite{Majidi-Zolbanin2013Entropy}.
Let $Q$ be a Noetherian local ring and let $\varphi$ be a contracting endomorphism of finite length.
Under this finite-length hypothesis, \cite[Theorem 13.3]{Avramov2006Homology} gives the equivalence between regularity of $Q$ and flatness of $\varphi$, extending Kunz's criterion for Frobenius to contracting endomorphisms.
\end{remark}
To extend $\DD_k(-)=(-\otimes_{\Fq}k^\sep)^{\Gal_k}$ to a contracting system $(Q,\varphi)$, we replace $k^\sep$ with the completed strict henselization of $Q$, with its compatible Galois and Frobenius actions.
\begin{lemma}\label{lem:henselian_equiv_cat}
Let $(R,\frakm_R)$ be a complete local ring with residue field $k$.
Reduction modulo $\frakm_R$ induces an equivalence between finite \'{e}tale $R$-algebras and finite \'{e}tale $k$-algebras.
\end{lemma}
\begin{proof}
The ring $R$ is henselian by \cite[\href{https://stacks.math.columbia.edu/tag/04GM}{Tag 04GM}]{stacks-project}, so the assertion follows from the finite \'{e}tale lifting equivalence \cite[\href{https://stacks.math.columbia.edu/tag/04GK}{Tag 04GK}]{stacks-project}.
\end{proof}

Fix a contracting system $(Q,\varphi)$ with residue field $k$.
Fix $k^\sep$, let $Q^\sh$ be the corresponding strict henselization of $Q$, and let $Q^\sep=\widehat{Q^\sh}$ with completion taken for $\frakm_QQ^\sh$.
This is a complete Noetherian local ring, faithfully flat over $Q$, with the maximal ideal $\frakm_QQ^\sep$ and residue field $k^\sep$.
The endomorphism $\varphi$ extends uniquely to a continuous endomorphism $\varphi^\sep$ of $Q^\sep$ inducing $\Frob_q$ on $k^\sep$.
It commutes with the natural continuous $\Gal_k$-action and is contracting, so $(Q^\sep,\varphi^\sep)$ is again a contracting system over $\Fq$.

For $V\in \Rep_{\Fq}(\Gal_k)$, we define
\[\DD_Q(V)\coloneqq (V\otimes_{\Fq}Q^\sep)^{\Gal_k}\]
using the diagonal Galois action.
Since $\varphi^\sep$ commutes with $\Gal_k$, the map $\id_V\otimes\varphi^\sep$ induces a $\varphi$-semilinear endomorphism of $\DD_Q(V)$.

\begin{proposition}\label{prop:functor_D_contracting}
Let $(Q,\varphi)$ be a contracting system over $\Fq$ with residue field $k$.
The assignment $V\mapsto \DD_Q(V)$ defines a functor
\[\DD_Q:\Rep_{\Fq}(\Gal_k)\to \textup{$\varphi$-Mod}_{Q}.\]
The canonical map
\[\DD_Q(V)\otimes_QQ^\sep\to V\otimes_{\Fq}Q^\sep\]
is an isomorphism compatible with Frobenius and $\Gal_k$.
\end{proposition}

\begin{lemma}\label{lem:galois_invariants}
Let $(R,\frakm)$ be a complete Noetherian local ring with residue field $k$.
Let $S=\widehat{R^\sh}$ with natural $\Gal_k$-action.
For any finitely generated $R$-module $N$, there is a canonical isomorphism
\[N\isomto (N\otimes_RS)^{\Gal_k},\quad n\mapsto n\otimes1.\]
\end{lemma}
\begin{proof}
For $a\ge1$, set $R_a=R/\frakm^a R$ and $S_a=S/\frakm_S^a\cong R^\sh/\frakm^a$.
Finite \'{e}tale descent followed by passage to the filtered union gives
\[N_a\cong (N_a\otimes_{R_a}S_a)^{\Gal_k}\]
for every finite $R_a$-module $N_a$.
Since $N\otimes_RS$ is complete, taking inverse limits with $N_a=N/\frakm^aN$ proves the assertion.
\end{proof}
\begin{proof}[Proof of Proposition \ref{prop:functor_D_contracting}]
Since $V$ is finite-dimensional and its action is continuous, there is a finite Galois extension $k'/k$ such that $\Gal_{k'}$ acts trivially on $V$.
Let $Q'\subset Q^\sep$ be the finite \'{e}tale local $Q$-algebra with residue field $k'$.
Lemma \ref{lem:galois_invariants} gives a canonical isomorphism $(Q^\sep)^{\Gal_{k'}}=Q'$.
Since $\Gal_{k'}$ acts trivially on $V$, we obtain
\begin{align*}
    \DD_Q(V)
    =\left(\left(V\otimes_{\Fq}Q^\sep\right)^{\Gal_{k'}}\right)^{\Gal(k'/k)}
    =\left(V\otimes_{\Fq}Q'\right)^{\Gal(k'/k)}
\end{align*}
Finite \'{e}tale Galois descent shows that $\DD_Q(V)$ is finite projective, hence free, over $Q$, and gives the canonical isomorphism
\[\DD_Q(V)\otimes _QQ'\xrightarrow{\sim} V\otimes_{\Fq}Q'.\]
After faithfully flat base change, the Frobenius linearization is that of $\id_V\otimes\varphi_{Q'}$ and is invertible.
Therefore, $\DD_Q(V)$ is an \'{e}tale $\varphi$-module over $Q$, and further base change to $Q^\sep$ gives the asserted comparison.
\end{proof}

For $M\in\textup{$\varphi$-Mod}_Q$, set
\[\VV_Q(M)\coloneqq\left(M\otimes_QQ^\sep\right)^{\varphi=\id}.\]
The action of $\Gal_k$ on the second factor preserves $\VV_Q(M)$ because it commutes with Frobenius.

\begin{proposition}\label{prop:functor_V_contracting}
Let $(Q,\varphi)$ be a contracting system over $\Fq$ with residue field $k$.
The assignment $M\mapsto \VV_Q(M)$ defines a functor
\[\VV_Q:\textup{$\varphi$-Mod}_{Q}\to\Rep_{\Fq}(\Gal_k).\]
For every $M\in \textup{$\varphi$-Mod}_{Q}$, the canonical map
\[\VV_Q(M)\otimes_{\Fq}Q^\sep\to M\otimes_QQ^\sep\]
is an isomorphism.
\end{proposition}

The following lemma is a key observation; contraction removes the maximal-ideal contribution to both fixed vectors and coinvariants, so both are detected on the residue field.

\begin{lemma}\label{lem:contract_to_resfield}
Let $(Q,\varphi)$ be a contracting system over $\Fq$ with residue field $k$.
For every $\varphi$-module $M$ over $Q$ reduction induces isomorphisms of $\Fq$-vector spaces
\[M^{\varphi=\id}\to (M\otimes_Qk)^{\varphi=\id},\quad \coker(\varphi_M-\id)\to \coker(\varphi_{M\otimes_Qk}-\id).\]
\end{lemma}
\begin{proof}
We apply $\varphi-\id$ to the short exact sequence of $\Fq$-vector spaces
\begin{center}
    \begin{tikzcd}
    0\arrow[r]&\frakm_Q\cdot M\arrow[r]&M\arrow[r]&M\otimes_Q k\arrow[r]&0
    \end{tikzcd}
\end{center}
to obtain the following commutative diagram with exact rows.
\begin{center}
    \begin{tikzcd}
    0\arrow[r]&\frakm_Q\cdot M\arrow[r]\arrow[d,"\varphi-\id"]&M\arrow[r]\arrow[d,"\varphi-\id"]&M\otimes_Q k\arrow[r]\arrow[d,"\varphi-\id"]&0\\
    0\arrow[r]&\frakm_Q\cdot M\arrow[r]&M\arrow[r]&M\otimes_Q k\arrow[r]&0
    \end{tikzcd}
\end{center}
Choose $N\ge1$ with $\varphi^N(\frakm_Q)\subset\frakm_Q^2$.
Then $\varphi_M^{Ns}(\frakm_QM)\subset\frakm_Q^{2^s}M$ for every $s\ge0$.
As a finite module over the complete Noetherian ring $Q$, the module $\frakm_Q \cdot M$ is $\frakm_Q$-adically complete.
The convergent series $-\sum_{j=0}^\infty\varphi_M^j$ therefore defines an additive endomorphism of $\frakm_QM$, and the geometric series identity shows that it is inverse to $\varphi_M-\id$.
Therefore, 
\[\varphi-\id:\frakm_Q\cdot M\to\frakm_Q\cdot M\]
is bijective.
The snake lemma now gives both asserted isomorphisms.
\end{proof}
\begin{proof}[Proof of Proposition \ref{prop:functor_V_contracting}]
Since $(Q^\sep,\varphi^\sep)$ is a contracting system, the base change $M\otimes_QQ^\sep$ is an \'{e}tale $\varphi$-module over it.
Lemma \ref{lem:contract_to_resfield} gives a natural $\Gal_k$-equivariant isomorphism
\begin{align*}
    \VV_Q(M)\xrightarrow{\sim} (M\otimes_Qk^\sep)^{\varphi=\id}
    =\VV_k(M\otimes_Qk)
\end{align*}
so $\VV_Q(M)$ is a finite-dimensional $\Fq$-representation by Theorem \ref{thm:equiv_phimod_field}.
Therefore, $\VV_Q$ is naturally isomorphic to $\VV_k\circ(-\otimes_Qk)$, as expressed by the following commutative diagram.
\begin{center}
\begin{tikzcd}
    \textup{$\varphi$-Mod}_{Q}\arrow[rd,"\VV_Q"]\arrow[d,"-\otimes_Qk"]\\
    \textup{$\varphi$-Mod}_{k}\arrow[r,"\VV_k"]&\Rep_{\Fq}(\Gal_k)
\end{tikzcd}
\end{center}
Consider the canonical map
\[\VV_Q(M)\otimes_{\Fq}Q^\sep\to M\otimes_QQ^\sep.\]
Its source is finite free because $\VV_Q(M)$ is finite-dimensional, and its target is finite free by Lemma \ref{lem:structure_thm_contracting}.
Therefore, it suffices to show that its reduction
\[\VV_Q(M)\otimes_{\Fq}k^\sep\to (M\otimes_Qk)\otimes_kk^\sep\]
is an isomorphism.
Under the preceding identification $\VV_Q(M)\cong \VV_k(M\otimes_Qk)$, this is exactly the isomorphism of Lemma \ref{lem:functor_V_field}, which proves the assertion.
\end{proof}

Lemma \ref{lem:contract_to_resfield} also gives the following analogue of the Artin-Schreier exact sequence.

\begin{lemma}\label{lem:SES_contracting}
Let $(Q,\varphi)$ be a contracting system over $\Fq$.
Then there is a $\Gal_k$-equivariant short exact sequence of $\Fq$-vector spaces
\begin{center}
    \begin{tikzcd}
    0\arrow[r]&\Fq\arrow[r]&Q^\sep\arrow[r,"\varphi^\sep-\id"]&Q^\sep\arrow[r]&0
    \end{tikzcd}
\end{center}
\end{lemma}
\begin{proof}
Apply Lemma \ref{lem:contract_to_resfield} to $Q^\sep$ as a $\varphi$-module over itself.
Its fixed vectors reduce isomorphically to $\Fq$, and its coinvariants vanish because $X^q-X-\alpha$ is separable and has a root in $k^\sep$ for every $\alpha\in k^\sep$.
Therefore,
\[\ker(\varphi^\sep-\id)=\Fq,\quad \coker(\varphi^\sep-\id)=0.\]
\end{proof}

\begin{theorem}\label{equivalence_contracting}
Let $(Q,\varphi)$ be a contracting system over $\Fq$ with residue field $k$.
The functors $\DD_Q$ and $\VV_Q$ are mutually quasi-inverse equivalences between $\Rep_{\Fq}(\Gal_k)$ and $\textup{$\varphi$-Mod}_{Q}$.
There is a commutative diagram
\begin{center}
\begin{tikzcd}
    \Rep_{\Fq}(\Gal_k)\arrow[r,"\DD_Q"]\arrow[rd,"\DD_k"']&\textup{$\varphi$-Mod}_{Q}\arrow[d,"-\otimes_Qk"]\\
    &\textup{$\varphi$-Mod}_{k}
\end{tikzcd}
\end{center}
consisting of equivalences of categories.
\end{theorem}
\begin{proof}
Let $V\in \Rep_{\Fq}(\Gal_k)$.
Proposition \ref{prop:functor_D_contracting} gives the canonical Frobenius and $\Gal_k$-equivariant isomorphism
\[\DD_Q(V)\otimes_QQ^\sep\cong V\otimes_{\Fq}Q^\sep.\]
Taking Frobenius fixed vectors gives
\[\VV_Q(\DD_Q(V))\cong (V\otimes_{\Fq}Q^\sep)^{\varphi=\id}.\]
By Lemma \ref{lem:SES_contracting}, the map $v\mapsto v\otimes1$ induces a natural $\Gal_k$-equivariant isomorphism
\[V\to (V\otimes_{\Fq}Q^\sep)^{\varphi=\id}\]
so $\VV_Q\circ\DD_Q\simeq\id$.

Now fix $M\in\textup{$\varphi$-Mod}_{Q}$.
Proposition \ref{prop:functor_V_contracting} gives the canonical Frobenius and $\Gal_k$-equivariant isomorphism
\[\VV_Q(M)\otimes_{\Fq}Q^\sep\to M\otimes_QQ^\sep.\]
Taking Galois invariants yields
\[\DD_Q(\VV_Q(M))\cong (M\otimes_QQ^\sep)^{\Gal_k}.\]
By Lemma \ref{lem:galois_invariants}, the canonical map
\[M\to (M\otimes_QQ^\sep)^{\Gal_k}\]
is an isomorphism.
It is natural in $M$ and Frobenius equivariant, so $\DD_Q\circ\VV_Q\simeq\id$.
\end{proof}

\subsection{Definition of $F$-dynamical systems}
Fix a finite extension $E/\Qp$, write $\Fq$ for its residue field, and choose a uniformizer $\pi_E$.

\begin{definition}
An \textit{$F$-dynamical system over $\cO_E$} is a pair $(R,\varphi)$ consisting of a complete Noetherian local $\cO_E$-algebra $R$ and a local $\cO_E$-algebra endomorphism $\varphi$ satisfying the following conditions.
\begin{enumerate}
    \item The structure map $\cO_E\to R$ is faithfully flat.
    Equivalently, $\pi_E$ is a nonzerodivisor and belongs to $\frakm_R$.
    \item The induced endomorphism $\overline{\varphi}$ of $R/(\pi_E)$ is contracting, and the induced endomorphism of $k=R/\frakm_R$ is the $q$-power Frobenius.
\end{enumerate}
We call $(R,\varphi)$ flat if $\varphi:R\to R$ is flat.
\end{definition}
\begin{example}\label{ex:fds}
Let $k$ be a perfect field containing $\Fq$.
Then $(W_{\cO_E}(k),\varphi_E)$ is a flat $F$-dynamical system over $\cO_E$.
In fact, $W_{\cO_E}(k)$ is a complete discrete valuation ring, faithfully flat over $\cO_E$, its quotient by $\pi_E$ is $k$, and $\varphi_E$ lifts $x\mapsto x^q$.
\end{example}

\begin{remark}
The terminology is inspired by the ``local algebraic dynamical system'' of \cite{Majidi-Zolbanin2013Entropy}, which consists of a Noetherian local ring together with a local endomorphism of finite length.
\end{remark}

\begin{lemma}\label{lem:fds_extension_fields}
Let $(R,\varphi)$ be an $F$-dynamical system over $\cO_E$ with residue field $k$.
For an algebraic separable extension $k'/k$, let $R'$ be  the corresponding ind-finite \'{e}tale local $R$-algebra, and put $R_{k'}=\widehat{R'}$ where completion is taken at $\frakm_RR'$.
Then the following are true.
\begin{enumerate}
    \item There is a unique endomorphism $\varphi_{k'}:R_{k'}\to R_{k'}$, extending $\varphi$ and inducing $x\mapsto x^q$ on $k'$.
    The pair $(R_{k'},\varphi_{k'})$ is an $F$-dynamical system over $\cO_E$.
    \item If $\varphi$ is flat (resp. finite), then $\varphi_{k'}$ is flat (resp. finite).
\end{enumerate}
\end{lemma}
\begin{proof}
(i)
Write $R'=\varinjlim_\ell R_\ell$, where $\ell$ runs through the finite subextensions of $k'/k$ and $R_\ell$ is the corresponding finite \'{e}tale local $R$-algebra.
At each finite stage, the equivalence of Lemma \ref{lem:henselian_equiv_cat} uniquely lifts $\Frob_q:\ell\to\ell$ to a $\varphi$-semilinear map $R_\ell\to R_\ell$.
These maps are compatible and define $\varphi':R'\to R'$.
Since $\varphi'$ is local, it extends uniquely to a local endomorphism $\varphi_{k'}$ of $R_{k'}$.
The maps $R\to R'$ and $R'\to R_{k'}$ are faithfully flat (the latter because $R$ is Noetherian), and so is their composite.
It remains to verify the contraction of $R_{k'}/(\pi_E)$.
Choose $N$ with $\varphi^N(\frakm_R)\subset\frakm_R^2+\pi_ER$.
Since $\frakm_{R_{k'}}=\frakm_R R_{k'}$, we obtain $\varphi_{k'}^N(\frakm_{R_{k'}})\subset \frakm_{R_{k'}}^2+\pi_ER_{k'}$.

(ii)
It suffices to prove the assertion for $\varphi':R'\to R'$, because flatness is preserved on completing a local homomorphism of Noetherian local rings, and a finite local homomorphism remains finite after completion.
Factor $\varphi'$ as 
\[R'\xrightarrow{a\mapsto a\otimes1} R'\otimes_{R,\varphi}R\xrightarrow{\theta}R'\]
where $\theta(a\otimes r)=\varphi'(a)r$.
The first map is the base change of $\varphi$, so it is flat (resp. finite).
It remains to show that $\theta$ is an isomorphism.
The map $\theta$ is the filtered colimit of the corresponding maps for finite subextensions $\ell/k$ of $k'/k$, so it suffices to treat $k'/k$ finite.
Both rings are finite \'{e}tale over $R$ through the second tensor factor and the given structure map respectively.
By Lemma \ref{lem:henselian_equiv_cat}, it suffices to show that the induced map 
\[k'\otimes_{k,\Frob_q}k\to k',\quad a\otimes b\mapsto a^qb\] 
is an isomorphism.
This is an isomorphism because relative Frobenius is an isomorphism on \'{e}tale algebras in characteristic $p$.
\end{proof}

\begin{proposition}\label{prop:key_ses}
Let $(R,\varphi)$ be an $F$-dynamical system over $\cO_E$ with residue field $k$, and choose a separable closure $k^\sep$ of $k$.
Let $(\BR,\varphi^\sep)$ be the $F$-dynamical system obtained from $k^\sep/k$ by Lemma \ref{lem:fds_extension_fields}.
There is a canonical $\Gal_k$-equivariant short exact sequence of $\cO_E$-modules
\begin{center}
\begin{tikzcd}
0\arrow[r]&\cO_E\arrow[r]&\BR\arrow[r,"\varphi^\sep-\id"]&\BR\arrow[r]&0.
\end{tikzcd}
\end{center}
\end{proposition}
\begin{proof}
Lemma \ref{lem:fds_extension_fields} allows us to apply the defining properties of an $F$-dynamical system to $\BR$.
Replacing $(R,\varphi)$ by $(\BR,\varphi^\sep)$, we may assume that $k$ is separably closed.
The structure map $\cO_E\to R$ is injective and $R$ is $\pi_E$-torsion free because it is faithfully flat over $\cO_E$.
Put $\overline{R}=R/(\pi_E)$, which is a contracting system over $\Fq$.
Consider the following commutative diagram
\begin{center}
    \begin{tikzcd}
    0\arrow[r]&\ker(\varphi-\id)\arrow[r]\arrow[d]&R\arrow[r,"\varphi-\id"]\arrow[d]&R\arrow[d]\\
    0\arrow[r]&\Fq\arrow[r]&\overline{R}\arrow[r,"\overline{\varphi}-\id"]&\overline{R}\arrow[r]&0
    \end{tikzcd}
\end{center}
By Lemma \ref{lem:SES_contracting}, $\overline{\varphi}-\id:\overline{R}\to\overline{R}$ is surjective with kernel $\Fq$.
If $\varphi(x)=x$, then its image $\overline{x}\in\overline{R}$ belongs to $\ker({\varphi}-\id)=\Fq$.
Choose a lift $a_0\in\cO_E$ of $\overline{x}$; then $x=a_0+\pi_E x_1$ for a unique element $x_1\in R$.
Since $\varphi$ fixes $\cO_E$ and $\pi_E$ is a nonzerodivisor, $\varphi(x_1)=x_1$.
Iterating the process gives $a_i\in\cO_E$ with $x-\sum_{i=0}^{N-1}\pi_E^ia_i\in\pi_E^NR$.
Completeness of $\cO_E$ and separatedness of $R$ yield $x=\sum_{i=0}^\infty\pi_E^ia_i\in\cO_E$.
Therefore,
\[\ker(\varphi-\id:R\to R)=\cO_E.\]
To prove surjectivity, fix $y\in R$.
Surjectivity of $\overline{\varphi}-\id$ gives $x_0,y_1\in R$ such that $y=(\varphi-\id)(x_0)+\pi_E y_1$.
Inductively choose $x_i,y_{i+1}\in R$ with $y_i=(\varphi-\id)(x_i)+\pi_Ey_{i+1}$, starting from $y_0=y$.
The series $x=\sum_{i=0}^\infty\pi_E^ix_i$ converges in $R$, and continuity and $\cO_E$-linearity give $(\varphi-\id)(x)=y$.
This proves surjectivity and completes the proof.
\end{proof}

\subsection{From representations to $\varphi$-modules}
Write $\Rep_{\cO_E}(\Gal_k)$ for the category of finitely generated $\cO_E$-modules with a continuous action of $\Gal_k$, where the modules carry the $\pi_E$-adic topology.
We first construct a functor $\DD_R:\Rep_{\cO_E}(\Gal_k)\to\phimod_R$ without assuming flatness of $\varphi$.

Let $(R,\varphi)$ be an $F$-dynamical system over $\cO_E$ with residue field $k$.
For $V\in \Rep_{\cO_E}(\Gal_k)$, define
\[\DD_R(V)\coloneqq(V\otimes_{\cO_E}R^\sep)^{\Gal_k}\]
using the diagonal Galois action and the semilinear endomorphism induced by $\id\otimes\varphi^\sep$.

\begin{proposition}\label{prop:functor_D_fds}
Let $(R,\varphi)$ be an $F$-dynamical system over $\cO_E$ with residue field $k$.
The preceding construction defines a functor
\[\DD_R:\Rep_{\cO_E}(\Gal_k)\to\textup{$\varphi$-Mod}_{R}.\]
For every $\cO_E$-representation $V$, the multiplication map
\[\DD_R(V)\otimes_RR^\sep\cong V\otimes_{\cO_E}R^\sep\]
is an isomorphism compatible with $\varphi$ and $\Gal_k$.
\end{proposition}

\begin{proof}
If $\pi_E^nV=0$ for some $n$, then $V$ is finite, so continuity implies that its $\Gal_k$-action factors through $\Gal(k'/k)$ for some finite Galois extension $k'$ of $k$.
Let $R'/R$ be the finite \'{e}tale algebra corresponding to $k'/k$.
The identities $(R^\sep/\pi_E^m R^\sep)^{\Gal_{k'}}=R'/\pi_E^mR'$ imply
\[(V\otimes_{\cO_E}R^\sep)^{\Gal_{k'}}=V\otimes_{\cO_E}R'.\]
Finite Galois descent therefore gives
\[\DD_R(V)\otimes_RR'\cong V\otimes_{\cO_E}R'.\]
By base change to $R^\sep$, we obtain the isomorphism 
\begin{equation}\label{isom:V_torsion}
\DD_R(V)\otimes_RR^\sep\to V\otimes_{\cO_E}R^\sep
\end{equation}
for every $\pi_E$-power torsion representation.
Since $R\to R^\sep$ is faithfully flat and $R^\sep$ is flat over $\cO_E$, (\ref{isom:V_torsion}) shows that $\DD_R$ is exact on torsion representations.
Its values are \'{e}tale by the linearization calculation.

Now let $V$ be an arbitrary finitely generated $\cO_E$-representation of $\Gal_k$.
The complete Noetherian local ring $R^\sep$ is also $\pi_E$-adically complete, so
\[V\otimes_{\cO_E}R^\sep\cong \varprojlim_n \bigl((V/\pi_E^nV)\otimes_{\cO_E}R^\sep\bigr).\]
Taking $\Gal_k$-invariants commutes with inverse limits, so
\begin{align*}
    \DD_R(V)
    \cong\varprojlim_n \DD_R(V/\pi_E^nV)
\end{align*}

Fix $n\ge1$, let $m\ge n$, and write $D_m=\DD_R(V/\pi_E^mV)$.
The natural maps give compatible right exact sequences
\begin{center}
\begin{tikzcd}
    D_{m+1}\arrow[r,"\times\pi_E^n"]\arrow[d]&D_{m+1}\arrow[r]\arrow[d]&D_n\arrow[d,equal]\arrow[r]&0\\
    D_m\arrow[r,"\times\pi_E^n"]&D_m\arrow[r]&D_n\arrow[r]&0
\end{tikzcd}
\end{center}
Torsion case exactness makes these rows exact and transition maps $D_{m+1}\to D_m$ surjective.
Set $K_m\coloneqq\ker(D_m\to D_n)$.
The transition maps $K_{m+1}\to K_m$ are surjective, so $\{K_m\}$ satisfies the Mittag-Leffler condition.
Taking inverse limits gives an exact sequence
\[0\to \varprojlim_m K_m\to \varprojlim_m D_m\to D_n\to 0.\]
Set $L_m\coloneqq\ker(D_m\to K_m)$ where the map is multiplication by $\pi_E^n$.
If we write $W_m=(V/\pi_E^mV)[\pi_E^n]$, then exactness of $\DD_R$ on torsion representations implies 
\[L_m=\DD_R(W_m),\quad\mathrm{im}(L_{m'}\to L_m)=\DD_R(\mathrm{im}(W_{m'}\to W_m)).\]
For fixed $m$, $W_m$ is a finite set so the images inside $W_m$ stabilize, so the images inside $L_m$ also stabilize, showing the Mittag-Leffler condition for $L_m$.
Therefore, we obtain another short exact sequence 
\[0\to \varprojlim_m L_m\to \varprojlim_m D_m\to \varprojlim_m K_m\to 0.\]
Therefore,
\begin{equation}\label{isom:mod_and_V}
    \DD_R(V)/(\pi_E^n)\cong\DD_R(V/\pi_E^nV)
\end{equation}

Put $\overline{R}=R/(\pi_E)$.
Compatibility of strict henselization and completion with quotients gives a canonical $\Gal_k$ and $\varphi$ equivariant isomorphism $(R^\sep)/(\pi_E)\cong(\overline{R})^\sep$.
Therefore, 
\begin{align*}
    \DD_R(V/\pi_EV)
    \cong \left((V/\pi_EV)\otimes_{\Fq}(\overline{R})^\sep\right)^{\Gal_k}
    =\DD_{\overline{R}}(V/\pi_EV).
\end{align*}
Combining this with (\ref{isom:mod_and_V}) yields
\begin{equation}\label{isom:D_and_Dbar}
    \DD_R(V)/(\pi_E)\cong\DD_{\overline{R}}(V/(\pi_E)).
\end{equation}
By Proposition \ref{prop:functor_D_contracting} and (\ref{isom:D_and_Dbar}), $\DD_R(V)/(\pi_E)$ is finite over $R/(\pi_E)$.
Equation (\ref{isom:mod_and_V}) and the inverse limit description show that $\DD_R(V)$ is $\pi_E$-adically complete, so topological Nakayama implies that it is finitely generated over $R$.
The map
\begin{equation}\label{isom:key_D}
    \DD_R(V)\otimes_RR^\sep\to V\otimes_{\cO_E}R^\sep
\end{equation}
is a homomorphism of finitely generated $R^\sep$-modules.
By (\ref{isom:mod_and_V}), its reduction modulo $\pi_E^m$ is the torsion comparison (\ref{isom:V_torsion}) for $V/\pi_E^mV$, so it is an isomorphism for every $m$.
Both modules are $\pi_E$-adically complete, so (\ref{isom:key_D}) is an isomorphism.
It remains to verify \'{e}taleness.
Consider the Frobenius linearization
\[\varphi^L_{\DD_R(V)}:\DD_R(V)\otimes_{R,\varphi}R\to \DD_R(V).\]
It suffices to show that this map becomes an isomorphism after the faithfully flat base change $R\to R^\sep$.
Using (\ref{isom:key_D}) and compatibility of the endomorphisms, the base change is identified with
\[(V\otimes_{\cO_E}R^\sep)\otimes_{R^\sep,\varphi^\sep}R^\sep\to V\otimes_{\cO_E}R^\sep,\]
which sends $(v\otimes r)\otimes r'\mapsto v\otimes\varphi^\sep(r)r'$.
Its inverse is $v\otimes s\mapsto (v\otimes1)\otimes s$, which is well-defined because $\varphi^\sep$ fixes $\cO_E$.
Therefore, $\DD_R(V)$ is \'{e}tale.
\end{proof}
The construction of $\DD_R$ commutes with reduction modulo $\pi_E$.
\begin{corollary}
Let $(R,\varphi)$ be an $F$-dynamical system over $\cO_E$.
Write $\overline{R}=R/(\pi_E)$.
The following diagram, where vertical arrows are reductions modulo $\pi_E$, commutes up to the canonical natural isomorphism (\ref{isom:D_and_Dbar})
\begin{center}
\begin{tikzcd}
    \Rep_{\cO_E}(\Gal_k)\arrow[r,"\DD_R"]\arrow[d]&\textup{$\varphi$-Mod}_{R}\arrow[d]\\
    \Rep_{\Fq}(\Gal_k)\arrow[r,"\DD_{\overline{R}}"]&\textup{$\varphi$-Mod}_{\overline{R}}
\end{tikzcd}
\end{center}
\end{corollary}
Taking Frobenius invariants after extension to $R^\sep$ recovers the original representation.
\begin{proposition}\label{prop:recover_rep_from_D_fds}
Let $(R,\varphi)$ be an $F$-dynamical system over $\cO_E$ with residue field $k$.
For $V\in\Rep_{\cO_E}(\Gal_k)$, there is a canonical $\Gal_k$-equivariant isomorphism
\[V\isomto (\DD_R(V)\otimes_RR^\sep)^{\varphi=\id}.\]
\end{proposition}
\begin{proof}
Proposition \ref{prop:functor_D_fds} gives
\[(\DD_R(V)\otimes_RR^\sep)^{\varphi=\id}\isomto (V\otimes_{\cO_E}R^\sep)^{\varphi=\id}.\]
By Proposition \ref{prop:key_ses}, there is a short exact sequence
\begin{center}
    \begin{tikzcd}
    0\arrow[r]&\cO_E\arrow[r]&R^\sep\arrow[r,"\varphi-\id"]&R^\sep\arrow[r]&0
    \end{tikzcd}
\end{center}
Since $R^\sep$ is flat over $\cO_E$, tensoring this sequence with $V$ remains exact and gives
\begin{center}
    \begin{tikzcd}
    0\arrow[r]&V\arrow[r,"v\mapsto v\otimes 1"]&V\otimes_{\cO_E}R^\sep\arrow[r,"\id\otimes\varphi-\id"]&V\otimes_{\cO_E}R^\sep\arrow[r]&0
    \end{tikzcd}
\end{center}
The first arrow therefore identifies $V$ with the required kernel.
All maps are natural and $\Gal_k$-equivariant.
\end{proof}

\subsection{From $\varphi$-modules to representations}
Let $(R,\varphi)$ be an $F$-dynamical system over $\cO_E$ with residue field $k$.
\begin{definition}
For $M\in\phimod_R$, we define
\[\VV_R(M)\coloneqq (M\otimes_R\BR)^{\varphi=\id}.\]
The action of ${\Gal_k}$ on $\BR$ commutes with $\varphi$, and therefore induces an $\cO_E$-linear action on $\VV_R(M)$.
\end{definition}

\begin{lemma}[{\cite[Proposition 1.1.5]{Fontaine1990Representations}}]\label{lem:flat_implies_abelian}
Let $R$ be a Noetherian ring with a flat endomorphism $\varphi$.
Then $\phimod_R$ is an abelian category.
\end{lemma}

\begin{proposition}\label{prop:key_isom_V}
Let $(R,\varphi)$ be a flat $F$-dynamical system over $\cO_E$.
Fix $M\in\phimod_R$.
Then $\VV_R(M)$ is finitely generated over $\cO_E$, and the canonical map
\[\VV_R(M)\otimes_{\cO_E}\BR\to M\otimes_R\BR\]
is an isomorphism compatible with $\varphi$ and ${\Gal_k}$.
Moreover, for every $n\ge1$, the natural map
\[\VV_R(M)/\pi_E^n\VV_R(M)\to \VV_R(M/\pi_E^nM)\]
is an isomorphism.
\end{proposition}

\begin{proof}
We first prove the canonical isomorphism for modules annihilated by $\pi_E$, using Proposition \ref{prop:functor_V_contracting} for the contracting system $R/\pi_ER$.
Flatness of $\varphi$ then allows induction on the power of $\pi_E$ annihilating the module.
Finally, we pass to inverse limits, proving finite generation and compatibility with reduction along the way.

First suppose that $\pi_EM=0$.
Put $\overline{R}=R/\pi_ER$.
Then $M$ is an \'{e}tale $\varphi$-module over $\overline{R}$.
Since $\overline{R}^\sep$ and $R^\sep/\pi_ER^\sep$ are naturally identified, the canonical map
\[\VV_R(M)\to \VV_{\overline{R}}(M)\]
is an isomorphism.
Proposition \ref{prop:functor_V_contracting} gives a canonical isomorphism
\[\VV_{\overline{R}}(M)\otimes_{\Fq}\overline{R}^\sep\cong M\otimes_{\overline{R}}\overline{R}^\sep,\]
which implies the assertion.

Next consider modules annihilated by a power of $\pi_E$.
Assume the assertion holds for modules annihilated by $\pi_E^N$, and let $M$ satisfy $\pi_E^{N+1}M=0$.
Set 
\[M'=M[\pi_E]=\ker(\pi_E:M\to M),\quad M''=M/M'.\]
Then $\pi_E^NM''=0$.
By Lemma \ref{lem:flat_implies_abelian} and \textcolor{Cerulean}{flatness of $\varphi$}, both $M'$ and $M''$ are \'{e}tale $\varphi$-modules.
Flatness of $R^\sep$ over $R$ gives the following commutative diagram with exact rows.
\begin{center}
    \begin{tikzcd}
    0\arrow[r]&M'\otimes_R\BR\arrow[d,"\varphi-\id"]\arrow[r]&M\otimes_R\BR\arrow[d,"\varphi-\id"]\arrow[r]&M''\otimes_R\BR\arrow[d,"\varphi-\id"]\arrow[r]&0\\
    0\arrow[r]&M'\otimes_R\BR\arrow[r]&M\otimes_R\BR\arrow[r]&M''\otimes_R\BR\arrow[r]&0
    \end{tikzcd}
\end{center}
The base case identifies $M'\otimes_RR^\sep$ with $\VV_R(M')\otimes_{\Fq}(R^\sep/\pi_ER^\sep)$.
Since $\varphi-\id$ is surjective on $R^\sep/\pi_ER^\sep$ by Lemma \ref{lem:SES_contracting}, the left vertical map is surjective.
The snake lemma gives
\[0\to \VV_R(M')\to\VV_R(M)\to\VV_R(M'')\to0.\]
Tensoring this sequence with the flat $\cO_E$-module $R^\sep$ gives the upper row of the following commutative diagram.
\begin{center}
    \begin{tikzcd}
        0\arrow[r]&\VV_R(M')\otimes_{\cO_E}\BR\arrow[d]\arrow[r]&\VV_R(M)\otimes_{\cO_E}\BR\arrow[d]\arrow[r]&\VV_R(M'')\otimes_{\cO_E}\BR\arrow[d]\arrow[r]&0\\
        0\arrow[r]&M'\otimes_R\BR\arrow[r]&M\otimes_R\BR\arrow[r]&M''\otimes_R\BR\arrow[r]&0
    \end{tikzcd}
\end{center}
The outer vertical maps are isomorphisms by the base case and the induction hypothesis, so the middle vertical map is also an isomorphism.
This proves the canonical isomorphism for every $M$ annihilated by a power of $\pi_E$.
By faithful flatness of $\cO_E\to R^\sep$, the canonical isomorphism implies that $\VV_R(M)$ is finitely generated over $\cO_E$.
It also implies that $\VV_R$ is exact on modules annihilated by powers of $\pi_E$ because $R^\sep$ is flat over $R$.

We now pass to an arbitrary $M$ by $\pi_E$-adic completion.
Taking $\varphi$-invariants commutes with inverse limits, so
\[\varprojlim_n\VV_R(M/\pi_E^nM)\cong \left(\varprojlim_n((M/\pi_E^nM)\otimes_R\BR)\right)^{\varphi=\id}\]
The finite $R^\sep$-module $M\otimes_RR^\sep$ is $\pi_E$-adically complete, so the inverse limit inside the parentheses is $M\otimes_RR^\sep$.
Therefore, 
\begin{equation}\label{eq:limit_V}
    \VV_R(M)\cong\varprojlim_n\VV_R(M/\pi_E^nM).
\end{equation}
Fix $n\ge1$.
We claim that the canonical map
\begin{equation}\label{eq:mod_pin_V}
    \VV_R(M)/\pi_E^n\VV_R(M)\cong\VV_R(M/\pi_E^nM)
\end{equation}
is an isomorphism.
Put $A_m=\VV_R(M/\pi_E^mM)$.
The transition maps $A_{m+1}\to A_m$ are surjective.
Indeed, $\VV_R$ is exact on modules annihilated by powers of $\pi_E$, and $M/\pi_E^{m+1}\to M/\pi_E^m$ is surjective.
For $m>n$, the same exactness of $\VV_R$ gives
\[A_m\xrightarrow{\times\pi_E^n}A_m\to A_n\to0.\]
Each $A_m$ is finite, so the inverse systems of kernels in these sequences satisfy the Mittag-Leffler condition.
Passing to inverse limits, as in Proposition \ref{prop:functor_D_fds}, gives
\[(\varprojlim_m A_m)/\pi_E^n(\varprojlim_m A_m)\isomto A_n\]
which proves (\ref{eq:mod_pin_V}) by (\ref{eq:limit_V}).

Equations (\ref{eq:mod_pin_V}) and (\ref{eq:limit_V}) show that $\VV_R(M)$ is $\pi_E$-adically complete.
Its reduction modulo $\pi_E$ is finite by the characteristic $p$ case, so topological Nakayama implies that it is finitely generated over $\cO_E$.
Both sides of the canonical map
\[\VV_R(M)\otimes_{\cO_E}\BR\to M\otimes_R\BR\]
are finite $R^\sep$-modules and hence $\pi_E$-adically complete, so it suffices to prove the mod $\pi_E^n$ reduction is an isomorphism for every $n\ge1$.
By (\ref{eq:mod_pin_V}), the reduction is the canonical map for $M/\pi_E^nM$, which is an isomorphism by the torsion case.
\end{proof}

We obtain the following structure theorem for underlying modules of \'{e}tale $\varphi$-modules.
\begin{corollary}\label{cor:etphimod_structure}
Let $(R,\varphi)$ be a flat $F$-dynamical system over $\cO_E$.
For every $M\in\phimod_R$, there is an isomorphism of $R$-modules
\[M\cong R^{\oplus r}\oplus \bigoplus_{i=1}^n R/\pi_E^{e_i}R.\]
\end{corollary}

\begin{proof}
Since $\VV_R(M)$ is finite over the DVR $\cO_E$, there is an isomorphism $\VV_R(M)\cong \cO_E^{\oplus r}\oplus \bigoplus_{i=1}^n \cO_E/(\pi_E^{e_i})$ for some $r$ and $1\le e_i<\infty$.
Proposition \ref{prop:key_isom_V} therefore gives an isomorphism 
\[M\otimes_R\BR\cong (\BR)^{\oplus r}\oplus\bigoplus_{i=1}^n \BR/\pi_E^{e_i}\BR.\]
\begin{lemma}[{\cite[Proposition 2.5.8]{Grothendieck1967Elements}}]\label{lem:ff_isom}
Let $R\to R'$ be a faithfully flat homomorphism of Noetherian local rings.
Let $M$ and $N$ be finitely generated $R$-modules such that $M\otimes_{R}R'\cong N\otimes_{R}R'$ as $R'$-modules.
Then $M$ and $N$ are isomorphic as $R$-modules.
\end{lemma}
Apply Lemma \ref{lem:ff_isom} to $M$ and $R^{\oplus r}\oplus \bigoplus_{i=1}^n R/(\pi_E^{e_i})$ to obtain the required isomorphism.
\end{proof}

\begin{proposition}\label{prop:functor_V}
Let $(R,\varphi)$ be a flat $F$-dynamical system over $\cO_E$.
The assignment $M\mapsto \VV_R(M)$ defines an exact functor
\[\VV_R:\phimod_R\to\Rep_{\cO_E}(\Gal_k).\]
\end{proposition}
\begin{proof}
By Proposition \ref{prop:key_isom_V}, $\VV_R(M)$ is finitely generated over $\cO_E$.
Exactness may be checked after the faithfully flat base change $\cO_E\to \BR$, where Proposition \ref{prop:key_isom_V} identifies $\VV_R(-)\otimes_{\cO_E}R^\sep$ with the exact functor $-\otimes_RR^\sep$.

It remains to show that, for every $n\ge1$, the action of ${\Gal_k}$ on $\VV_R(M)/\pi_E^n\VV_R(M)$ has open stabilizers.
From (\ref{eq:mod_pin_V}), we may replace $M$ by $M/\pi_E^nM$ and therefore assume that $\pi_E^nM=0$.
Then $V\coloneqq \VV_R(M)$ is a finite set, embedded in the finite $R^\sep$-module $N\coloneqq M\otimes_RR^\sep$.
Let $\frakm$ be the maximal ideal of $R^\sep$.
Since $N$ is $\frakm$-adically complete and $V$ is a finite set, there exists $c\ge1$ such that $V\cap\frakm^cN=0$.
Therefore, reduction induces a $\Gal_k$-equivariant injection $V\to N/\frakm^cN$.
Write $R^\sh$ for the ind-(finite \'{e}tale) local $R$-algebra whose completion is $R^\sep$.
Since
\[R^\sep/\frakm^c\cong R^\sh/\frakm_R^c R^\sh\]
every element of $N/\frakm^cN$ is represented using finitely many coefficients from a finite \'{e}tale stage and therefore has an open stabilizer in $\Gal_k$.
The equivariant injection $V\to N/\frakm^cN$ therefore shows that every element of $V$ has an open stabilizer, proving continuity.
\end{proof}
\subsection{The equivalence and its compatibility with base change}
\begin{theorem}\label{thm:FDS_cat_equiv}
Let $(R,\varphi)$ be a flat $F$-dynamical system over $\cO_E$ with residue field $k$.
Then
\[\DD_R:\Rep_{\cO_E}(\Gal_k)\to\textup{$\varphi$-Mod}_R\]
is an equivalence with a quasi-inverse $\VV_R$.
\end{theorem}

\begin{proof}
Proposition \ref{prop:recover_rep_from_D_fds} gives a natural isomorphism $\VV\circ\DD\simeq\id$.

Now fix $M\in\phimod_R$.
Taking $\Gal_k$-invariants in the canonical isomorphism of Proposition \ref{prop:key_isom_V} gives
\[\DD(\VV(M))\cong(M\otimes_R\BR)^{\Gal_k}.\]
By Lemma \ref{lem:galois_invariants}, the canonical map
\[M\to (M\otimes_R\BR)^{\Gal_k},\quad m\mapsto m\otimes1\]
is an isomorphism, so the assertion follows.
\end{proof}

We next record compatibility with change of the base ring.
Let $(R,\varphi_R)$ and $(S,\varphi_S)$ be flat $F$-dynamical systems over $\cO_E$.
Let $f:R\to S$ be a local $\cO_E$-algebra homomorphism satisfying $f\circ\varphi_R=\varphi_S\circ f$.
Base change along $f$ defines a functor
\[f^*:\textup{$\varphi$-Mod}_R\to \textup{$\varphi$-Mod}_S,\quad(M,\varphi_M)\mapsto (M\otimes_RS,m\otimes s\mapsto \varphi_M(m)\otimes\varphi_S(s))\]
\begin{proposition}\label{prop:fds_compatibiliy}
Let $f:(R,\varphi_R)\to (S,\varphi_S)$ be a local $\varphi$-equivariant homomorphism of flat $F$-dynamical systems over $\cO_E$ inducing $k\hookrightarrow \ell$ on residue fields.
Assume that $\ell/k$ is algebraic separable, and choose an embedding $\ell\hookrightarrow k^\sep$ over $k$.
Then there is a natural isomorphism
\[\VV_S\circ f^*\simeq \Res\circ\VV_R\]
where $\Res$ denotes the restriction of representations via $\Gal_\ell\subset\Gal_k$, so the following diagram commutes up to natural isomorphism
\begin{center}
\begin{tikzcd}
    \textup{$\varphi$-Mod}_R\arrow[r,"f^*"]\arrow[d,"\VV_R"]&\textup{$\varphi$-Mod}_S\arrow[d,"\VV_S"]\\
    \Rep_{\cO_E}(\Gal_k)\arrow[r,"\Res"]&\Rep_{\cO_E}(\Gal_\ell)
\end{tikzcd}
\end{center}
\end{proposition}
\begin{proof}
With the chosen embedding $\ell\subset k^\sep$, we may use $k^\sep$ as a separable closure of both $k$ and $\ell$.
Functoriality of finite \'{e}tale lifts gives a unique local map $R^\sh\to S^\sh$ over $f$ inducing the identity on the common residue field $k^\sep$.
After completion, this gives a local map $R^\sep\to S^\sep$, and uniqueness of the lift shows that it commutes with $\varphi$ and is $\Gal_{\ell}$-equivariant.

Fix $M\in \textup{$\varphi$-Mod}_R$.
Proposition \ref{prop:key_isom_V} gives a $\varphi$ and $\Gal_k$ equivariant isomorphism
\[(M\otimes_RR^\sep)^{\varphi=\id}\otimes_{\cO_E}R^\sep\cong M\otimes_RR^\sep.\]
Extending scalars to $S^\sep$ gives a $\varphi$ and $\Gal_\ell$ equivariant isomorphism
\[(M\otimes_RR^\sep)^{\varphi=\id}\otimes_{\cO_E}S^\sep\cong M\otimes_RS^\sep.\]
Taking $\varphi$-invariants yields a natural $\Gal_\ell$-equivariant isomorphism 
\[\Res(\VV_R(M))\isomto\VV_S(M\otimes_RS)\]
as desired.
\end{proof}

\subsection{$F$-dynamical systems associated to formal groups}
Let $K/E$ be finite unramified, put $q=q_E$ and $f=[K:E]$.
Let $H$ be a one-dimensional formal $\cO_E$-module of finite height over $\cO_K$.
Retain the notation $K'$, $P$, $d$, $R_H$ and the coordinates of Section \ref{subsec:model}, and write $G$ for the special fiber of $H$.
\begin{proposition}\label{prop:bh_fds}
Under the above hypotheses, $(R_H,\varphi\coloneqq \varphi_E)$ is a flat $F$-dynamical system over $\cO_E$.
\end{proposition}

Applying Theorem \ref{thm:FDS_cat_equiv} to $(R_H,\varphi)$ gives the following equivalence.
\begin{corollary}\label{cor:phimod_bh}
The functor $\DD_{R_H}$ gives an equivalence
\[\DD_{R_H}:\Rep_{\cO_E}(\Gal_{k_H})\simeq\textup{$\varphi$-Mod}_{R_H}\]
with a quasi-inverse $\VV_{R_H}$.
\end{corollary}

We first record a criterion for constructing an $F$-dynamical system by localization and completion.
\begin{lemma}\label{lem:construction_fds}
Let $S$ be a Noetherian $\cO_E$-algebra that is faithfully flat over $\cO_E$.
Let $f:S\to S$ be an $\cO_E$-algebra endomorphism, and let $\frp\subset S$ be a prime ideal containing $\pi_E$, subject to the following conditions.
\begin{enumerate}
    \item The endomorphism $f$ induces $\Frob_q$ on $S/\frp$.
    \item With $\overline{S}=S/(\pi_E)$ and $\overline{\frp}=\frp/(\pi_E)$, $\overline{f}^N(\overline{\frp})\subset \overline{\frp}^2$ for some $N\ge1$.
\end{enumerate}
Let $R=(S_\frp)^\wedge$ where the completion is taken with respect to $\frp S_\frp$.
The endomorphism induced by $f$, denoted $\varphi$, makes $(R,\varphi)$ an $F$-dynamical system over $\cO_E$.
\end{lemma}
\begin{proof}
The ring $R$ is Noetherian and complete local.
Localization and completion preserve $\pi_E$-torsion freeness, and $\pi_E$ lies in its maximal ideal, so $\cO_E\to R$ is faithfully flat.
Since $f(s)$ reduces to $\overline{s}^q$ in the domain $S/\frp$, $f(s)\in\frp$ if and only if $s\in\frp$, so $f^{-1}(\frp)=\frp$.
Therefore, $f$ induces a local endomorphism of $S_\frp$, and a continuous endomorphism $\varphi$ on $R$, which induces $\Frob_q$ on the residue field $\Frac(S/\frp)$.

The identification $R/(\pi_E)\cong (S_\frp/(\pi_E))^\wedge$ reduces the contraction condition to the induced endomorphism of $S_\frp/(\pi_E)$.
Since $S_\frp/(\pi_E)\cong \overline{S}_{\overline{\frp}}$, the containment $\overline{f}^N(\overline{\frp})\subset\overline{\frp}^2$ localizes and then extends by continuity to the completion, proving the required contraction.
\end{proof}

\begin{proof}[Proof of Proposition \ref{prop:bh_fds}]
The abstract power series description of $\cA_{H}$ shows that it is regular.
Therefore, $R_H=((\cA_H)_{\frp_H})^\wedge$ is also regular.
Retain the notation $\overline{\xi}_H:\cA_{H}\to \OCb$ for the reduced period homomorphism, and let $\frp_H=\ker(\overline{\xi}_H)$.
We first apply Lemma \ref{lem:construction_fds} to prove that $(R_H,\varphi)$ is an $F$-dynamical system, and then prove flatness.
The power series model makes $\cA_H$ faithfully flat over $\cO_E$, and equivariance of $\overline{\xi}_H$ gives Frobenius on $\cA_H/\frp_H$.
It remains to prove contraction at $\frp_H$ modulo $\pi_E$.
Put 
\[\kappa'=\kappa_{K'},\quad A_0=\cA_H/(\pi_E).\]
The induced endomorphism $\overline{\varphi}$ acts on $\kappa'$ by $\Frob_q$.
Write $\overline{\frp}_H=\frp_H/(\pi_E)\subset A_0$.
The induced homomorphism, still denoted $\overline{\xi}_H:A_0\to \OCb$, has kernel $\overline{\frp}_H$.
We must find $N\ge1$ such that $\varphi^N(\overline{\frp}_H)\subset\overline{\frp}_H^2$.
In fact, we will show $\overline{\varphi}^{fd}(\overline{\frp}_H)\subset \overline{\frp}_H^p$.

Let us write $P(X)=X^d+\sum_{\ell=0}^{d-1}c_\ell X^\ell$ with $c_\ell\in\frakm_E$.
Choose formal coordinates $t_{k,i}$ of $\cA_H$.
The defining Frobenius relation of Section \ref{subsec:model} gives
\begin{equation*}
    \overline{\varphi}^{fd}(t_{0,i})
    = \sideset{}{_G}\sum_{\ell=0}^{d-1}[-c_\ell]_G(t_{f\ell,i}).
\end{equation*}
Since $c_i\in\frakm_E$, the reduced endomorphism $[-c_i]_G$ has derivative zero and hence lies inside $\kappa'\llbracket T^p\rrbracket$.
Since $\kappa'$ is perfect, each term $[-c_\ell]_G(t_{f\ell,i})$ is a $p$-th power in $A_0$, and their formal group sum is again a $p$-th power in $A_0$.
For $0\le k\le fd-1$, we have $t_{k,i}=\varphi^k(t_{0,i})$, so
\[\varphi^{fd}(t_{k,i})=\varphi^k(\varphi^{fd}(t_{0,i}))\in A_0^p.\]
Since $A_0^p=\kappa'\llbracket t_{k,i}^p\rrbracket$ is closed, continuity and the preceding computation give $\overline{\varphi}^{fd}(A_0)\subset A_0^p$.
Let $a\in\overline{\frp}_H$.
Choose $b\in A_0$ with $\varphi^{fd}(a)=b^p$.
Equivariance gives
\[\overline{\xi}_H(b^p)=\Frob_q^{fd}(\overline{\xi}_H(a))=0.\]
Since $\OCb$ is reduced, $b\in\overline{\frp}_H$.
Therefore, 
\[\overline{\varphi}^{fd}(\overline{\frp}_H)\subset \overline{\frp}_H^p\]
proving contraction and therefore the $F$-dynamical system axioms.

We now prove that $\varphi:R_H\to R_H$ is flat.
Since $\cA_H$ is regular local, finiteness of its local endomorphism $\varphi$ will imply flatness by \cite[Theorem 23.1]{Matsumura1987Commutative}.
Since $\varphi^{-1}(\frp_H)=\frp_H$, flatness on $\cA_H$ passes to the local endomorphism of $(\cA_H)_{\frp_H}$ and then to its completion $R_H$.
Therefore, it suffices to prove finiteness on $\cA_{H}$.
For a local homomorphism of Noetherian complete local rings $A\to B$, if $B/\frakm_AB$ is finite-dimensional over $\kappa_A$, complete Nakayama implies that $B$ is finite over $A$ \cite[\href{https://stacks.math.columbia.edu/tag/031D}{Tag 031D}]{stacks-project}.
Choosing a $\Lambda_H$-basis $\alpha_1,\cdots,\alpha_n$ of $T(H)$ gives an isomorphism
\[\cA_{H}\cong\cO_{K'}\llbracket t_{k,i}:0\le k\le fd-1,1\le i\le n\rrbracket\]
and
\[\varphi(t_{k,i})=t_{k+1,i}\,\textup{ ($0\le k<fd-1$)},\quad \varphi(t_{fd-1,i})=\sideset{}{_H}\sum_{\ell=0}^{d-1} [-c_\ell]_H(t_{f\ell,i}).\]
Since $K'/K$ is unramified, the maximal ideal of $\cA_H$ is $\frakm_H=(\pi_E,t_{k,i})$.
Modulo $\varphi(\frakm_H)\cA_H$, all $t_{k,i}$ with $k\ge1$ vanish, and the remaining relations give an isomorphism
\[\frac{\cA_{H}}{\varphi(\frakm_{H})\cA_{H}}\cong\frac{\kappa'\llbracket t_{0,1},\cdots,t_{0,n}\rrbracket}{\left([-c_0]_G(t_{0,1}),\cdots, [-c_0]_G(t_{0,n})\right)}.\]
Since $c_0\neq0$ and $G$ has finite height, the endomorphism $[-c_0]_G$ is nonzero.
Write $[-c_0]_G(T)=T^Nu(T)$ with $u(0)\neq0$.
Then the quotient above is isomorphic to 
\[\frac{\kappa'\llbracket t_{0,1},\cdots,t_{0,n}\rrbracket}{(t_{0,1}^{N},\cdots,t_{0,n}^{N})}\cong \frac{\kappa'[t_{0,1},\cdots,t_{0,n}]}{(t_{0,1}^{N},\cdots,t_{0,n}^{N})}.\]
This quotient has $\kappa'$-dimension $N^n$, so complete Nakayama proves that $\varphi$ on $\cA_H$ is finite.
Regularity gives flatness, which passes to $R_H$ and completes the proof.
\end{proof}

The equivalence above concerns representations of $\Gal_{k_H}$.
In the next section, we pass to the coefficient ring whose residue field is $k_{H,L}$, use the Galois group comparison of Section \ref{sec:perfectoid}, and then incorporate the action of $\Gamma_L$.

\newpage
\section{$(\varphi,\Gamma)$-modules associated to formal groups}\label{sec:phigamma}
Combining the perfectoid theory studied in Section \ref{sec:perfectoid} with the framework of $F$-dynamical systems in Section \ref{sec:fds}, we first establish an equivalence
\[\Rep_{\cO_E}(\Gal_{L_\infty})\simeq \phimod_{R_{H,L}}.\]
We then incorporate the action of $\Gamma_L=\Gal(L_\infty/L)$ which fits into the exact sequence
\[1\to\Gal_{L_\infty}\to\Gal_L\to\Gamma_L\to1.\]
The main point is to specify a continuity condition that makes the resulting category equivalent to continuous $\cO_E$-representations of $\Gal_L$.

\subsection{The rings $R_{H,L}$ and representations of $\Gal_{L_\infty}$}

By Proposition \ref{prop:bh_fds}, the pair $(R_H,\varphi_E)$ is a flat $F$-dynamical system over $\cO_E$, where $R_H=((\cA_H)_{\frp_H})^\wedge$ has residue field $k_H$.
Fix a finite extension $L/K$, and set $L_\infty=K_\infty L$.
Proposition \ref{prop:lHL} provides a finite separable extension $k_{H,L}/k_{H,K}$ and a natural isomorphism of profinite groups
\[\Gal_{L_\infty}\isomto \Gal_{k_{H,L}}.\]
The extension $k_{H,L}/k_H$ is algebraic separable.
Let $R_{H,L}'$ denote the ind-(finite \'{e}tale) local $R_H$-algebra with residue field $k_{H,L}$, and put $R_{H,L}=(R_{H,L}')^\wedge$ where the completion is taken with respect to the maximal ideal.
By Lemma \ref{lem:fds_extension_fields}, $\varphi_E$ extends uniquely to $R_{H,L}$ with $q_E$-power action on its residue field, and $(R_{H,L},\varphi_E)$ is again a flat $F$-dynamical system over $\cO_E$. 
By construction, $R_{H,L}$ is a finite \'{e}tale local algebra over $R_{H,K}$ corresponding to the finite separable extension $k_{H,L}/k_{H,K}$.
With the fixed embedding $k_H^\sep\subset\Cb$, the completed strict henselization of $R_{H,L}$ identifies with $R_H^\sep$.
Lemma \ref{lem:galois_invariants} implies
\[(R_H^\sep)^{\Gal_{L_\infty}}=R_{H,L}.\]
Applying Theorem \ref{thm:FDS_cat_equiv} and the preceding identification of absolute Galois groups gives the following equivalence.
\begin{corollary}\label{cor:linfty_equiv}
There is an equivalence
\[\DD_{H,L}:\Rep_{\cO_E}(\Gal_{L_\infty})\simeq\textup{$\varphi_E$-Mod}_{R_{H,L}},\quad V\mapsto (V\otimes_{\cO_E}R_{H}^\sep)^{\Gal_{L_\infty}}\]
where $\Gal_{L_\infty}$ acts on $R_{H}^\sep$ through its identification with $\Gal_{k_{H,L}}$.
The quasi-inverse is 
\[\VV_{H,L}(M)=(M\otimes_{R_{H,L}}R_{H}^\sep)^{\varphi=\id}.\]
\end{corollary}

\begin{example}[Lubin-Tate case]\label{ex:lubintate}
Take $E=K$, and let $H$ be a Lubin-Tate formal group over $\cO_K$ associated with a uniformizer $\pi_K$.
The minimal polynomial is $P(X)=X-\pi_K$, so $d=1$ and $\cA_H\cong\cO_K\llbracket T\rrbracket$.
If $\alpha$ is an $\cO_K$-basis of $T(H)$ and $t=\overline{\tau}_H(\alpha)$, evaluation identifies $o_H$ with $\Fq\llbracket t\rrbracket$ and $k_H$ with $\Fq(\!(t)\!)$.
The induced valuation is a positive multiple of the $t$-adic valuation, so $k_H$ is complete and hence $k_{H,K}=k_H$.
Therefore, $\frp_H=(\pi_K)$ and
\[R_{H,K}\cong\left(\cO_{K}\llbracket T\rrbracket\left[\tfrac{1}{T}\right]\right)_{(\pi_K)}^\wedge,\quad \varphi_K(T)=[\pi_K]_H(T)\]
The map $R_{H,K}\to W_{\cO_E}(C^\flat)$ sending $T\mapsto \tau_H(\alpha)$ is injective because its reduction modulo $\pi_K$ embeds $\Fq(\!(t)\!)$ into $\Cb$, the target is $\pi_K$-torsion-free, and the source is $\pi_K$-adically separated.
By Example \ref{ex:tau_LT}, its image is the usual Lubin-Tate base ring \cite[Section 1.3]{Ren2009Galois}.

For $K=E=\Qp$ and $\Gmhat$, this gives Fontaine's cyclotomic coefficient ring with $T\mapsto[\epsilon]-1$ and $\varphi(T)=(T+1)^p-1$ \cite[Section 3.2.2]{Fontaine1990Representations}.
\end{example}

\subsection{The main equivalence}
Fix $k_H^\sep\subset\Cb$.
For every $g\in\Gal_K$, the automorphism of $R_H$ from Section \ref{subsec:model} and the compatible action on $k_H^\sep$ induce a unique semilinear automorphism of the strict henselization, and hence of $R_H^\sep$.
Uniqueness gives the group action and commutation with $\varphi_E$.
The subgroup $\Gal_L$ preserves $R_{H,L}\subset R_H^\sep$.
Since $\Gal_{L_\infty}$ acts trivially on $R_H$ and $k_{H,L}$, it acts trivially on $R_{H,L}$.
The action therefore factors through $\Gamma_L=\Gal(L\infty/L)$.
To extend Corollary \ref{cor:linfty_equiv} to Galois representations of $L$, we need to specify a continuity condition on the $\Gamma$-action.

\begin{lemma}\label{lem:rhsep_wcb}
For the chosen embedding $k_H^\sep\subset\Cb$, the period homomorphism extends uniquely to a local $\cO_E$-algebra homomorphism
\[R_H^\sep\to W_{\cO_E}(\Cb)\]
inducing that embedding on residue fields.
This homomorphism is equivariant for $\Gal_K$ and $\varphi_E$.
\end{lemma}

\begin{proof}
Since $K/E$ is unramified, there is a canonical isomorphism $ W_{\cO_E}(\Cb)\cong W_{\cO_K}(\Cb)$.
The period homomorphism therefore gives a local $\cO_K$-algebra map
\[R_H\to W_{\cO_E}(\Cb)\]
which is equivariant for $\Gal_K$ and $\varphi_E$.
Put $R=R_H$, and let $k'$ range over finite subextensions of $k_H^\sep/k_H$.
Let $R_{k'}$ be the finite \'{e}tale local $R$-algebra with residue field $k'$.
It is complete, and Lemma \ref{lem:fds_extension_fields} supplies its compatible endomorphism $\varphi_E$.
Since $W_{\cO_E}(\Cb)$ is henselian, the chosen embedding $k'\hookrightarrow\Cb$ lifts uniquely to an $R$-algebra homomorphism $R_{k'}\to W_{\cO_E}(\Cb)$.
These lifts are compatible with inclusions of finite subextensions.
The compatible maps give a homomorphism from the strict henselization of $R$ to $W_{\cO_E}(\Cb)$.
Its maximal ideal maps into the ideal $(\pi_E)$, so it extends uniquely to its completion $R_H^\sep$.
Uniqueness of the \'{e}tale lifts shows that the maps commute with $\varphi_E$ and with the action of $\Gal_K$, and these identities persist after completion.
\end{proof}

Put $B=W_{\cO_E}(\Cb)$.
Using the unique expansion $b=\sum_{i=0}^\infty\pi_E^i[x_i]$, identify the underlying set of $B$ with $\prod_{i=0}^\infty \Cb$, and give it the product topology where each factor has its valuation topology.
We call this topology on $B$ the \textit{weak topology}.
For $n\ge1$, we give $B/\pi_E^nB$ the topology transported from $(\Cb)^n$ by the first $n$ Teichm\"{u}ller coordinates.
It is also the quotient topology from $B$.
More generally, give a finite $B$-module $N$ the quotient topology from a surjection $B^r\to N$.
This topology, called \textit{weak topology}, is independent of the chosen generators, and every $B$-linear map between finite $B$-modules is continuous.
We use these topologies throughout the subsection.

\begin{lemma}\label{lem:padic_weak_topologies}
The weak topology on $B$ induces the $\pi_E$-adic topology on its subring $\cO_E$.
For every $n\ge1$, the subspace topology on $\cO_E/(\pi_E^n)\subset B/\pi_E^nB$ is discrete.
\end{lemma}
\begin{proof}
Fix $n\ge1$.
The space $B/\pi_E^nB\cong(\Cb)^n$ is Hausdorff because $\Cb$ is Hausdorff.
The finite subset $\cO_E/(\pi_E^n)$ therefore has the discrete subspace topology.
The weak topology on $B$ is the inverse limit topology for the weak topologies on $B/(\pi_E^n)$.
The induced topology on $\cO_E$ is the initial topology for the maps $\cO_E\to\cO_E/(\pi_E^n)$, with discrete targets, and hence is its $\pi_E$-adic topology.
For a finite $\cO_E$-module $V$, write $V\cong \cO_E^{\oplus r}\oplus(\oplus_i\cO_E/\pi_E^{e_i})$.
The natural map
\[V\to V\otimes_{\cO_E}B,\quad v\mapsto v\otimes 1\]
is the direct sum of the embeddings $\cO_E\to B$ and $\cO_E/\pi_E^{e_i}\to B/\pi_E^{e_i}$, which is a topological embedding.
\end{proof}

We now define the category of \'{e}tale $(\varphi,\Gamma_L)$-modules over $R_{H,L}$ by specifying a continuity condition of the $\Gamma_L$-action.
The period homomorphism $R_{H}^\sep\to W_{\cO_E}(\Cb)$ equips scalar extensions with a topology reflecting the valuation on $\Cb$; this gives a uniform continuity condition for all finite modules including both free and torsion modules.

\begin{definition}\label{def:phigammamod}
A \textit{$(\varphi,\Gamma_L)$-module} over $R_{H,L}$ is a finitely generated $R_{H,L}$-module $M$ with a $\varphi_E$-semilinear map $\varphi_M$ and a semilinear $\Gamma_L$-action commuting with $\varphi_M$, such that the diagonal $\Gal_L$-action on $M\otimes_{R_{H,L}}B$ is continuous for the finite module weak topology specified above.

Such a module $M$ is \textit{\'{e}tale} if its underlying $\varphi$-module is \'{e}tale.
Write $\textup{$(\varphi_E,\Gamma_L)$-Mod}_{R_{H,L}}$ for the category of \'{e}tale $(\varphi,\Gamma)$-modules, with $R_{H,L}$-linear morphisms commuting with $\varphi_M$ and the $\Gamma_L$-action.
\end{definition}

\begin{theorem}\label{thm:phigamma_equiv}
Let $K/E$ be a finite unramified extension of $p$-adic local fields, let $H$ be a one-dimensional formal $\cO_E$-module of finite height over $\cO_K$, and let $L/K$ be finite.
With the actions described below, the functors
\[\DD_{H,L}(V)=(V\otimes_{\cO_E}R_H^\sep)^{\Gal_{L_\infty}},\quad \VV_{H,L}(M)=(M\otimes_{R_{H,L}}R_H^\sep)^{\varphi=\id}\]
are mutually quasi-inverse equivalences
\[\Rep_{\cO_E}(\Gal_L)\simeq\textup{$(\varphi_E,\Gamma_L)$-Mod}_{R_{H,L}}.\]
\end{theorem}
\begin{proof}
Corollary \ref{cor:linfty_equiv} gives an equivalence
\[\DD_{H,L}:\Rep_{\cO_E}(\Gal_{L_\infty})\simeq\textup{$\varphi_E$-Mod}_{R_{H,L}},\quad V\mapsto (V\otimes_{\cO_E}R_H^\sep)^{\Gal_{L_\infty}}\]
with quasi-inverse $\VV_{H,L}$.
The diagonal $\Gal_L$-action on $V\otimes_{\cO_E}R_H^\sep$ preserves its $\Gal_{L_\infty}$-invariants and induces the required semilinear $\Gamma_L$-action on $\DD_{H,L}(V)$.
Conversely, the diagonal $\Gal_L$-action on $M\otimes_{R_{H,L}}R_H^\sep$ preserves its Frobenius fixed points and defines the action on $\VV_{H,L}(M)$.
We verify continuity below.
The comparison maps from Section \ref{sec:fds} are equivariant for these actions.

Fix $V\in\Rep_{\cO_E}(\Gal_L)$.
Proposition \ref{prop:functor_D_fds}, applied to $V\vert_{\Gal_{L_\infty}}$, gives the natural comparison isomorphism 
\[\DD_{H,L}(V)\otimes_{R_{H,L}}R_H^\sep\cong V\otimes_{\cO_E}R_H^\sep.\]
The multiplication description shows that it commutes with $\varphi_E$ and the diagonal action of $\Gal_L$.
Base change along $R_H^\sep\to B$ gives a $\Gal_L$-equivariant $B$-linear isomorphism
\[\DD_{H,L}(V)\otimes_{R_{H,L}}B\cong V\otimes_{\cO_E}B\]
which is a homeomorphism for the finite module weak topologies.
The action of $\Gal_L$ on $\Cb$ is continuous, so its coordinate-wise action on $B$ is continuous for the weak topology.
Together with continuity on the finite $\cO_E$-module $V$, this implies continuity of the diagonal action on $V\otimes_{\cO_E}B$, and hence on $\DD_{H,L}(V)\otimes_{R_{H,L}}B$.

Conversely, let $M$ be an \'{e}tale $(\varphi,\Gamma)$-module over $R_{H,L}$.
By Proposition \ref{prop:key_isom_V} and base change along $R_H^\sep\to B$, the canonical map gives a $\Gal_L$-equivariant isomorphism
\[\VV_{H,L}(M)\otimes_{\cO_E}B\isomto M\otimes_{R_{H,L}}B\]
which is a homeomorphism for the finite module weak topologies.
Definition \ref{def:phigammamod} makes the action on the right-hand side continuous, and the isomorphism transports the continuity to the left-hand side.
Since $B$ is faithfully flat over $\cO_E$, the natural equivariant map
\[\VV_{H,L}(M)\to \VV_{H,L}(M)\otimes_{\cO_E}B,\quad v\mapsto v\otimes1\]
is injective.
The finite module consequence of Lemma \ref{lem:padic_weak_topologies} identifies the subspace topology with the $\pi_E$-adic topology.
Restricting the continuous action to this topological subspace proves that $\Gal_L$ acts continuously on $\VV_{H,L}(M)$.
The natural unit and counit of the equivalence in Corollary \ref{cor:linfty_equiv} commute with the diagonal actions by their construction.
Therefore, $\VV_{H,L}\circ\DD_{H,L}$ is naturally isomorphic to the identity on $\Rep_{\cO_E}(\Gal_L)$, and $\DD_{H,L}\circ\VV_{H,L}$ is naturally isomorphic to the identity on $\textup{$(\varphi_E,\Gamma_L)$-Mod}_{R_{H,L}}$, proving the theorem.
\end{proof}

\begin{example}\label{ex:tatemodule}
Let $K/E$ be a finite unramified extension, and let $H$ be a one-dimensional formal $\cO_E$-module of finite height over $\cO_K$.
Its Tate module $T(H)$ is a finite free $\cO_E$-module with a continuous action of $\Gal_K$, defining an object in the category $\Rep_{\cO_E}(\Gal_K)$.
Since the action of $\Gal_{K_\infty}$ is trivial on $T(H)$, we have
\begin{align*}
    \DD_{H,K}(T(H))
    &=(T(H)\otimes_{\cO_E}R_H^\sep)^{\Gal_{K_\infty}}\\
    &\cong T(H)\otimes_{\cO_E}R_{H,K}
\end{align*}
where $\Gamma$ acts diagonally on $T(H)\otimes_{\cO_E}R_{H,K}$ and $\varphi$ acts by $\id\otimes\varphi_E$.
The finite free $\cA_H$-module $M_H\coloneqq T(H)\otimes_{\cO_E}\cA_H$ has commuting semilinear actions of $\varphi_E$ and $\Gamma$, and satisfies
\[M_H\otimes_{\cA_H}R_{H,K}\cong \DD_{H,K}(T(H)).\]
\end{example}
Example \ref{ex:tatemodule} raises the question of whether a similar descent property holds for lattices in general crystalline representations of $\Gal_K$.
More precisely, given a $\Gal_K$-stable $\cO_E$-lattice $T$ in a crystalline $E$-representation, does $\DD_{H,K}(T)$ admit a finite free model over $\cA_H$, or over another natural smaller base ring determined by $H$ with compatible Frobenius and $\Gamma$-actions?
Such models could provide a step toward the geometric interpretation discussed in Section \ref{subsec:proofstrategy}.
This question is also motivated by the theory of Wach modules in the cyclotomic setting \cite{Wach1996Representations,Berger2002Representationsa} and by the construction of Kisin-Ren in the Lubin-Tate setting \cite[Corollary 3.3.8]{Ren2009Galois}.

\subsection{Intrinsic continuity for finite free modules}
For an \'{e}tale $\varphi$-module which is finite free over $R_{H,L}$ or $R_{H,L}/(\pi_E^n)$ for some $n\ge1$, we now express Definition \ref{def:phigammamod} as continuity of the $\Gamma_L$-action on the module itself.

Let $f:X\to Y$ be a map from a set $X$ to a topological space $Y$.
The \textit{initial topology} on $X$ induced by $f$ is the coarsest topology making $f$ continuous.
If $f$ is injective, this is the subspace topology under the identification $X\simeq f(X)$.

\begin{definition}
Fix a finite extension $L/K$ and put $B=W_{\cO_E}(\Cb)$.
We give $R_{H,L}$ the initial topology induced by $R_{H,L}\to B$ where $B$ is endowed with the weak topology.
For each $n\ge1$, give $R_{H,L}/(\pi_E^n)$ the initial topology induced by $R_{H,L}/(\pi_E^n)\to B/(\pi_E^n)$.
A finite free module over either ring is given the product topology in a basis.
These module topologies are independent of the basis, since multiplication by a matrix over either topological ring is continuous.
We call all these topologies \textit{weak topologies}.
\end{definition}
\begin{lemma}\label{lem:initial_top_group_action}
Let a topological group $G$ act on a set $X$ and a topological space $Y$, and let $f:X\to Y$ be equivariant.
If the action map $G\times Y\to Y$ is continuous, then the action map $G\times X\to X$ is continuous for the initial topology induced by $f$.
\end{lemma}
\begin{proof}
Let $\alpha_X$ and $\alpha_Y$ denote the action maps.
By equivariance, the following diagram commutes.
\begin{center}
    \begin{tikzcd}
    G\times X\arrow[r,"\alpha_X"]\arrow[d,"\id_G\times f"']&X\arrow[d,"f"]\\
    G\times Y\arrow[r,"\alpha_Y"]&Y
    \end{tikzcd}
\end{center}
The composite $G\times X\xrightarrow{\id_G\times f} G\times Y\xrightarrow{\alpha_Y} Y$ is continuous.
The defining property of the initial topology now implies that $\alpha_X$ is continuous.
\end{proof}

\begin{proposition}
Let $M$ be an \'{e}tale $\varphi$-module over $R_{H,L}$ equipped with a semilinear $\Gamma_L$-action commuting with $\varphi_M$.
Assume that $M$ is a finite free module over $R_{H,L}$ or $R_{H,L}/(\pi_E^n)$ for some $n\ge1$, and give $M$ and its scalar extension to $B$ their weak topologies.
The following continuity conditions are equivalent.
\begin{enumerate}
    \item The diagonal action of $\Gal_L$ on $M\otimes_{R_{H,L}}B$ is continuous.
    \item The $\Gamma_L$-action on $M$ is continuous for its weak topology.
\end{enumerate}
\end{proposition}
\begin{proof}
Assume (i).
The map
\[j:M\to M\otimes_{R_{H,L}}B,\quad m\mapsto m\otimes1\]
is $\Gal_L$-equivariant where $\Gal_L$ acts on $M$ through $\Gamma_L$.
By definition, the topology on the base ring is initial with respect to $R_{H,L}\to B$ in the free case and $R_{H,L}/(\pi_E^n)\to B/(\pi_E^n)$ in the torsion case.
Choosing a basis identifies $j$ with a finite product of these ring maps, so the weak topology on $M$ is the initial topology induced by $j$.
Lemma \ref{lem:initial_top_group_action} makes the $\Gal_L$-action on $M$ continuous.
Since the continuous $\Gal_L$-action factors through $\Gamma_L$, it induces a continuous action of $\Gamma_L$, equipped with its quotient topology.

Conversely, assume (ii).
Since the maps $R_{H,L}\to B$ and $R_{H,L}/(\pi_E^n)\to B/(\pi_E^n)$ with $n\ge1$ are continuous, the diagonal action of $\Gal_L$ on $M\otimes_{R_{H,L}}B$ (resp. $M\otimes_{R_{H,L}/(\pi_E^n)}B/(\pi_E^n)$) is continuous.
Therefore, the condition (i) follows.
\end{proof}

\newpage

\section{Independence of periods and the structure of the base ring}\label{sec:period_alg}
In this section, we closely analyze the abstract period algebra $\cA_H$ and its reduced image $o_H$ which were constructed in Section \ref{subsec:model} and served as a crucial ingredient for constructing the base ring in Theorem \ref{thm:phigamma_equiv}.
We first prove that the reduced exponential periods attached to an $E_H$-basis of $V(H)$ are formally independent.
This identifies $o_H$ with a power series ring and determines regular parameters for $R_H$ and $R_{H,L}$.
We then prove formal independence over $\cO_{\breve{K}}$ of a finite family of Frobenius iterates of the exponential periods.
When $E=K$, this identifies the abstract period algebra $\cA_H$ with its image in $\AinfK$.

We will use the following notation throughout the section
\[h_0=\mathrm{ht}(H),\quad h=\mathrm{ht}_{\cO_E}(H),\quad n=\dim_{E_H}(V(H))\]
which denote the ordinary $p$-height, relative $\cO_E$-height, and the absolute height of $H$ respectively.
Furthermore, we denote
\[f=[K:E],\quad b=[E_H:E],\quad d=\deg(P),\quad q_E=\#\kappa_E.\]
\subsection{Independence of reduced periods and the period field}
Let $H$ be a one-dimensional formal group of finite height over $\cO_K$, and let $G$ be its special fiber.
Using the exponential period map $\widetilde{\tau}_H$, we write
\[\overline{\tau}_H=\red\circ \widetilde{\tau}_H:V(H)\to H(\AinfK)\to G(\OCb)\]
and call it the reduced exponential period map.
The independence statement holds for every one-dimensional formal group of finite height, without the unramified coefficient hypothesis. 
Set $K'=KE_H$, let $\kappa'$ be its residue field, and let $n$ denote the absolute height of $H$.
\begin{theorem}\label{thm:bH_injective}
Let $H$ be a one-dimensional formal group of finite height over $\cO_K$, and let $G$ be its special fiber.
Let $\alpha_1,\cdots,\alpha_n$ be an $E_H$-basis of $V(H)$, and consider the reduced period map $\overline{\tau}_H$.
Then evaluation defines an injective homomorphism
\[\epsilon:\Fpbar\llbracket t_1,\cdots,t_n\rrbracket\to\OCb,\quad t_i\mapsto \overline{\tau}_H(\alpha_i).\]
\end{theorem}

\begin{remark}
Caruso proved a related formal independence result for periods associated with the cyclotomic and Kummer towers \cite{Caruso2013Representations}.
Choose compatible primitive roots of unity $\zeta_{p^n}$, put $\epsilon=(1,\zeta_p,\zeta_{p^2},\cdots)\in\OCb$, and set $\eta=\epsilon-1\in\OCb$.
The field of norms of $\Qp(\zeta_{p^\infty})/\Qp$ is $\Fp(\!(\eta)\!)$.
Likewise, a compatible sequence of $p^n$-th roots of $p$ defines $p^\flat\in\Cb$, and the field of norms of the Kummer tower $\Qp(p^{1/p^\infty})/\Qp$ is identified with $\Fp(\!(p^\flat)\!)$ (see \cite[Section 2.1]{Breuil1999application}).
For odd $p$, Caruso's formal independence result \cite[Proposition 1.7]{Caruso2013Representations} gives an injective map
\[\Fpbar\llbracket t_1,\,\,t_2\rrbracket \to\OCb,\quad t_1\mapsto\eta, t_2\mapsto p^\flat.\]
The proof below adapts Caruso's use of Galois translates and Hasse derivatives to obtain arbitrarily large bounds for the valuations of potential relations.
\end{remark}

\begin{proof}[Proof of Theorem \ref{thm:bH_injective}]
Put $x_i=\overline{\tau}_H(\alpha_i)\in\frakm_{\Cb}\setminus\{0\}$.
Evaluation at $\mathbf{x}=(x_1,\cdots,x_n)$ defines a continuous map
\begin{equation}\label{eq:substitution}
    \epsilon:\Fpbar\llbracket t_1,\cdots,t_n \rrbracket \to \OCb,\quad g\mapsto g(\mathbf{x})
\end{equation}
Let $\frp=\ker(\epsilon)$.
For $g\in \Fpbar\llbracket t_1,\cdots,t_n \rrbracket$, the Hasse derivatives $\partial^{[\mathbf{k}]}g$ are defined by the formal identity
\[g(\mathbf{t}+\mathbf{d})=\sum_{\mathbf{k}\in\NN^n}(\partial^{[\mathbf{k}]}g(\mathbf{t}))\mathbf{d}^{\mathbf{k}},\quad \partial^{[\mathbf{0}]}g=g.\]
We claim that for every $g\in\frp$, every $1\le i\le n$ and every $k\ge1$
\begin{equation}\label{eq:f_partialder}
    \partial_{t_i}^{[k]}g\in\frp
\end{equation}
Assume this derivative-stability statement for the moment.
Applying the claim successively in the different coordinates gives
\[\partial^{[\mathbf{k}]}g=\partial_{t_1}^{[k_1]}\cdots\partial_{t_n}^{[k_n]}g\in\frp\]
for every multi-index $\mathbf{k}$.
Reduction of $\epsilon$ modulo $\frakm_{\Cb}$ is the augmentation $\Fpbar\llbracket t_1,\cdots,t_n \rrbracket\to\Fpbar$, so $\partial^{[\mathbf{k}]}g(\mathbf{0})=0$ for all $\mathbf{k}$.
Since these are the coefficients of $g$, we conclude $g=0$.

It remains to prove (\ref{eq:f_partialder}).
Fix a coordinate $i$ and a relation $g\in\frp$.
By Theorem \ref{thm:openimage}, the image of $\rho:I_K\to\GL_{E_H}(V(H))$ is open.
Hence, there is an integer $N_0$ depending on $H$ and the chosen basis such that for every $r\ge N_0$ there is $\gamma_{r,i}\in I_K$ satisfying $\rho(\gamma_{r,i})=\Id+p^r E_{i,i}$.
By equivariance and the trivial action of inertia on the coefficients,
\begin{align*}
    g(x_1,\cdots,[p^r+1]_G(x_i),\cdots,x_n)=\gamma_{r,i}\cdot g(\mathbf{x})=0
\end{align*}
Put $\delta_r(t)=[p^r+1]_G(t)-t$.
Write $h_0=\mathrm{ht}(H)$.
Since $G(0,y)=y$ and the coefficient of $x$ in the linear part is 1, there is a unit $U(x,y)\in\kappa_K\llbracket x,y\rrbracket^\times$ with $U(0,0)=1$ such that $G(x,y)-y=xU(x,y)$.
Consequently,
\[\delta_r(t)=[p^r]_G(t)U([p^r]_G(t),t)\]
where the second factor is a unit in $\kappa_K\llbracket t\rrbracket$.
Since $G$ has height $h_0$, the order of $[p^r]_G(t)$ is $p^{rh_0}$, so 
\[\mathrm{ord}_t(\delta_r)=p^{rh_0}.\]
Its nonzero coefficients lie in a finite field and therefore have valuation zero.
Since $v(x_i)>0$, its leading term has strictly smaller valuation than the remaining terms, so we have 
\[\delta_r(x_i)\neq0,\quad v(\delta_r(x_i))=p^{rh_0}v(x_i).\]
Suppose that there were a least $s\ge1$ such that $\partial_{t_i}^{[s]}g(\mathbf{x})\neq0$.
Taylor expansion in the $i$-th variable, followed by evaluation, gives
\[\sum_{k=s}^\infty \delta_r(x_i)^k\partial_{t_i}^{[k]}g(\mathbf{x})=0.\]
The series converges because $v(\delta_r(x_i))>0$.
Dividing by the nonzero element $\delta_r(x_i)^s$ gives
\begin{equation}\label{ineq:valuation}
    \partial_{t_i}^{[s]}g(\mathbf{x})=-\sum_{k=s+1}^\infty \delta_r(x_i)^{k-s}\partial_{t_i}^{[k]}g(\mathbf{x}).
\end{equation}
Since the Hasse derivative evaluations are integral, it follows that
\[v(\partial_{t_i}^{[s]}g(\mathbf{x}))\ge v(\delta_r(x_i))=p^{rh_0}v(x_i).\]
The left-hand side is independent of $r$, whereas $p^{rh_0}v(x_i)\to\infty$ as $r\to\infty$, so it is a contradiction.
This proves (\ref{eq:f_partialder}) and hence injectivity of $\epsilon$.
\end{proof}

\begin{remark}
The proof uses only the following consequence of Theorem \ref{thm:openimage}: for every $i$ and sufficiently large $r$, the image of inertia contains $I_n+p^rE_{i,i}$.
In particular, it is enough that the inertia image contain an open subgroup of the diagonal torus in the chosen $E_H$-basis.
\end{remark}

Using the independence of reduced exponential periods, we can now describe $o_H$ explicitly and determine the structure of the base ring $R_{H,L}$.
Let $K/E$ be finite unramified, and let $H$ be a one-dimensional formal $\cO_E$-module of finite height over $\cO_K$, and let $G$ be its special fiber.
By Proposition \ref{prop:end_int_closed}, $\Lambda_H=\cO_{E_H}$ and the finite torsion free $\Lambda_H$-module $T(H)$ is free of rank $n$.

Recall that the choice of a basis identifies 
\[\cA_H\cong\cO_{K'}\llbracket t_{k,i}:0\le k\le fd-1,1\le i\le n\rrbracket,\quad [\alpha_i^{(k)}]\mapsto t_{k,i}\]
with notation in Section \ref{subsec:model}, and the period homomorphism identifies with
\[\xi_H:\cA_H\to\AinfK,\quad t_{k,i}\mapsto\varphi_E^k(\tau_i).\]
Write $\mathrm{red}:\AinfK\to\OCb$ for the reduction modulo $\pi_E$.
Then the composite
\[\red\circ \xi_H:\cA_H\to \OCb,\quad t_{k,i}\mapsto \Frob_{q_E}^k(\overline{\tau}_i)\]
factors through the formal group ring
\[\mathcal{B}_H=\kappa'\llbracket T(H) \rrbracket _{G,\Lambda_H}\cong \kappa'\llbracket t_1,\cdots,t_n\rrbracket\]
where the map $\cA_H\to\mathcal{B}_H$ reduces coefficients modulo $\pi_E$ and sends $t_{k,i}\mapsto t_i^{q_E^k}$.
Therefore, the subalgebra $o_H$ identifies with the image of the continuous homomorphism
\[\overline{\xi}:\mathcal{B}_H\to \OCb,\quad t_i\mapsto \overline{\tau}_i.\]

Applying Theorem \ref{thm:bH_injective} to a $\Lambda_H$-basis of $T(H)$ gives the following consequence.
\begin{corollary}\label{cor:inj}
Let $K/E$ be a finite unramified extension, and let $H$ be a one-dimensional formal $\cO_E$-module of finite height over $\cO_K$.
Then 
\[\overline{\xi}:\mathcal{B}_H\to\OCb,\quad t_i\mapsto \overline{\tau}_i\]
is injective, and hence identifies $\mathcal{B}_H$ with $o_H$.
\end{corollary}
\begin{proof}
With the identification $\mathcal{B}_H\cong \kappa'\llbracket t_1,\cdots,t_n\rrbracket$, restricting the injective homomorphism $\epsilon$ in Theorem \ref{thm:bH_injective} to coefficients in $\kappa'$ gives $\overline{\xi}$.
Therefore, $\overline{\xi}$ is injective.
\end{proof}

\begin{corollary}\label{cor:description_kh}
Under the hypotheses of Corollary \ref{cor:inj}, put $K'=KE_H$, let $\kappa'$ be its residue field, and put $n=\dim_{E_H}V(H)=\mathrm{ht}(H)/[E_H:\Qp]$.
A choice of $\Lambda_H$-basis of $T(H)$ gives isomorphisms
\[o_H\cong \kappa'\llbracket t_1,\cdots,t_n \rrbracket,\quad k_H\cong\Frac(\kappa'\llbracket t_1,\cdots,t_n \rrbracket).\]
\end{corollary}
\begin{proof}
This follows from Corollary \ref{cor:inj}.
\end{proof}
\begin{remark}
For $n>1$, the field $k_H$ is not a characteristic $p$ local field.
Therefore, $k_H$ differs from the local field arising in the field of norms construction of \cite{Wintenberger1983corps}.
\end{remark}
\begin{remark}
By Corollary \ref{cor:description_kh}, we obtain an equation
\[\dim\left(\mathrm{im}(\overline{\xi})\right)=n\coloneqq \frac{\mathrm{ht}_{\cO_E}(H)}{[E_H:E]}\]
relating the Krull dimension of the subalgebra of $\OCb$ topologically generated by $\overline{\tau}_H$ and the absolute height of $H$.
\end{remark}

\begin{corollary}\label{cor:henselian_absheight_one}
Let $K/E$ be a finite unramified extension of $p$-adic local fields.
Let $H$ be a one-dimensional finite height formal $\cO_E$-module over $\cO_K$.
Then the following are equivalent.
\begin{enumerate}
    \item The natural injective homomorphism $\Gal_{K_\infty}\to\Gal_{k_H}$ is an isomorphism.
    \item The absolute height of $H$ is one.
\end{enumerate}
\end{corollary}
\begin{proof}
By Corollary \ref{cor:fullimage_equiv_hens}, condition (i) is equivalent to henselianity of the induced valuation on $k_H$.
If the absolute height $n$ of $H$ is one, then $k_H\cong\kappa_{K'}(\!(t)\!)$ is a complete valued field because $v_{\Cb}$ induces a multiple of the $t$-adic valuation, and it is in particular henselian.
Conversely, suppose $n>1$.
Corollary \ref{cor:description_kh} implies that 
\[k_H\cong\Frac(\kappa_{K'}\llbracket t_1,\cdots,t_n\rrbracket)\]
which is not henselian by \cite[Theorem 15.4.6, Lemma 15.5.4]{Fried2008Field}.
\end{proof}

\subsection{Coefficient rings and regular parameters}
Let $K/E$ be a finite unramified extension, and let $H$ be a one-dimensional formal $\cO_E$-module of finite height over $\cO_K$.
Put $K'=KE_H$.
We now use the independence of the reduced periods to describe $R_H$ explicitly.
Define 
\[\cT_{H}\coloneqq\cO_{K'}\llbracket T(H) \rrbracket_{H,\Lambda_H}\cong\cO_{K'}\llbracket t_{0,1},\cdots,t_{0,n}\rrbracket\]
and view it as the closed $\cO_{K'}$-subalgebra of $\cA_H$.
Lemma \ref{lem:semilinear_action} gives a natural $\Gal_K$-action on $\cT_H$ and the period homomorphism
\[\cT_H\to\AinfK,\quad t_{0,i}\mapsto \tau_H(\alpha_i)\]
is $\Gal_K$-equivariant.
Its reduction modulo $\pi_E$ is injective by Corollary \ref{cor:inj}.
Since the target is $\pi_E$-torsion free, every element of the kernel is successively divisible by every power of $\pi_E$.
The $\pi_E$-adic separation of $\cT_H$ therefore proves injectivity.
Corollary \ref{cor:inj} identifies $\cT_H/(\pi_E)$ with $o_H$, and therefore the kernel of $\cT_H\to o_H$ is $(\pi_E)$.
Define $\cC_H=((\cT_{H})_{(\pi_E)})^\wedge$ where the completion is $\pi_E$-adic.
The localization $(\cT_H)_{(\pi_E)}$ is a discrete valuation ring with uniformizer $\pi_E$ and residue field $k_H$, so the same holds after the completion.
Since every element inverted in this localization maps to a unit of $W_{\cO_K}(\Cb)$, the period homomorphism extends to $\cC_H\to W_{\cO_K}(\Cb)$ and reduces to the embedding $k_H\to\Cb$.
The extended homomorphism is injective because its reduction modulo $\pi_E$ is injective, its target is $\pi_E$-torsion free, and $\cC_H$ is $\pi_E$-adically separated.

\begin{proposition}\label{prop:coefficient_ring}
In the preceding setup, let $n=\mathrm{rank}_{\Lambda_H}T(H)$, let $P$ be the polynomial used to construct $\cA_H$ in Section \ref{subsec:model}, let $d=\deg P$, and let $f=[K:E]$.
Choosing a $\Lambda_H$-basis of $T(H)$ gives an isomorphism of complete local $\cC_H$-algebras
\[\cC_H\llbracket s_{k,i}:1\le k\le df-1,1\le i\le n \rrbracket\isomto R_H,\quad s_{k,i}\mapsto t_{k,i}-t_{0,i}^{q_E^k}\]
In particular, $R_H$ is a complete regular local ring of dimension $1+n(df-1)$, and $\pi_E$ together with the elements $t_{k,i}-t_{0,i}^{q_E^k}$ forms a regular system of parameters.
\end{proposition}

\begin{lemma}\label{lem:coef_ring}
Let $S$ be a complete discrete valuation ring with uniformizer $\pi$, and let $S\to R$ be a local homomorphism to a complete Noetherian regular local ring of dimension $D$.
Suppose that $S\to R$ induces an isomorphism of residue fields and that the image of $\pi$ belongs to $\frakm_R\setminus\frakm_R^2$.
Then $R$ is isomorphic to $S\llbracket X_1,\cdots,X_{D-1} \rrbracket$ as a complete local $S$-algebra.
\end{lemma}
We call the image of $S$ a \textit{coefficient ring} for $R$ in this setting.
Its uniformizer is $\pi$, which need not equal $p$.

\begin{proof}
Let $\kappa$ be the common residue field and choose $x_1,\cdots,x_{D-1}\in\frakm_R$ such that the classes of $\pi,x_1,\cdots,x_{D-1}$ form a $\kappa$-basis of $\frakm_R/\frakm_R^2$.
Since $R$ is complete, substitution $X_i\mapsto x_i$ defines a continuous local $S$-algebra homomorphism
\[\alpha:S\llbracket X_1,\cdots,X_{D-1} \rrbracket \to R.\]
The map $\alpha$ induces an isomorphism on residue fields, and Nakayama's lemma gives $\frakm_R=(\pi,x_1,\cdots,x_{D-1})R$.
Since the residue field map is surjective and the images of $\pi,X_1,\cdots,X_{D-1}$ generate $\frakm_R$, successive approximation modulo $\frakm_R^c$ produces a preimage for every element of $R$.
Completeness of both rings therefore makes $\alpha$ surjective.
The source is a Noetherian domain of Krull dimension $D$, so a nonzero kernel would give $\dim R<D$.
Therefore, $\alpha$ is injective.
\end{proof}

\begin{proof}[Proof of Proposition \ref{prop:coefficient_ring}]
Put $s_{k,i}=t_{k,i}-t_{0,i}^{q_E^k}$ for $1\le k\le fd-1$ and $1\le i\le n$.
Under the coordinate change
\[\cA_H\cong\cO_{K'}\llbracket t_{0,1},\cdots,t_{0,n},s_{k,i}:1\le k\le fd-1,1\le i\le n\rrbracket\]
reduction annihilates $\pi_E$ and all $s_{k,i}$.
Corollary \ref{cor:inj} identifies the remaining quotient with $o_H$.
Therefore, 
\[\frp_H=(\pi_E,s_{k,i}).\]
In these coordinates, $\pi_E$ and $s_{k,i}$ form a regular sequence of length $n(df-1)+1$ in $\cA_H$ and generate $\frp_H$.
It follows that $\mathrm{ht}(\frp_H)=1+n(df-1)$.
Localization and completion therefore give a regular local ring $R_H$ of this dimension whose maximal ideal is generated by $\pi_E$ and the $s_{k,i}$.
These generators form a regular system of parameters, so $\pi_E\not\in\frakm_{R_H}^2$.
The inclusion $\cT_H\to \cA_H$ satisfies $\frp_H\cap \cT_H=(\pi_E)$, so localization and completion give a local homomorphism $\cC_H\to R_H$ inducing the identity on the identified residue fields $k_H$.
Applying Lemma \ref{lem:coef_ring} gives the isomorphism
\[\cC_H\llbracket s_{k,i} \rrbracket \isomto R_H,\quad s_{k,i}\mapsto t_{k,i}-t_{0,i}^{q_E^k}\]
so $\cC_H$ is a coefficient ring for $R_H$ in the stated sense.
\end{proof}

Let $L/K$ be a finite extension.
We denote by $\cC_{H,L}$ the completion of the ind-(finite \'{e}tale) local $\cC_H$-algebra corresponding to the separable extension $k_{H,L}/k_H$.
We now deduce the regularity and dimension of $R_{H,L}$.
\begin{proposition}\label{prop:R_regular}
With the notation of Proposition \ref{prop:coefficient_ring}, there is an isomorphism
\[R_{H,L}\cong \cC_{H,L}\llbracket s_{k,i}\rrbracket\]
of complete local $\cC_{H,L}$-algebras.
Consequently, $R_{H,L}$ is a complete regular local ring of dimension $1+n(df-1)$ with regular parameters $\pi_K$ and the $s_{k,i}$.
\end{proposition}

\begin{proof}
For every finite subextension $\ell/k_{H}$ of $k_{H,L}/k_H$, let $\cC_\ell$ be the finite \'{e}tale local $\cC_H$-algebra with residue field $\ell$.
Since $R_H=\cC_H\llbracket s_{k,i}\rrbracket$,
\[\cC_\ell\otimes_{\cC_H}R_H\cong \cC_\ell\llbracket s_{k,i}\rrbracket.\]
This is a finite \'{e}tale local $R_H$-algebra with residue field $\ell$, so it is the corresponding local $R_H$-algebra $R_\ell$ in the construction of $R_{H,L}$.
These identifications are compatible with inclusions of finite residue field extensions, so passing to the filtered union and completing with respect to $(\pi_K,s_{k,i})$ gives an isomorphism
\[\cC_{H,L}\llbracket s_{k,i}\rrbracket\isomto R_{H,L}.\]
The ring $\cC_{H,L}$ is a complete discrete valuation ring because the strict henselization of $\cC_H$ is a discrete valuation ring.
The assertion now follows.
\end{proof}

\subsection{Independence of Frobenius iterates of periods}\label{subsec:indep_model}
We now prove a mixed-characteristic independence theorem for exponential periods.
More precisely, the periods attached to an $E_H$-basis of $V(H)$ together with their first $h-1$ Frobenius iterates are formally independent over $\cO_K$.
When $E=K$, this theorem identifies the abstract period algebra $\cA_H$ with its image in $\AinfK$.
The proof uses the invertibility of an associated crystalline period matrix, which we establish first.
Let $\ell_H$ denote the normalized formal logarithm of $H$.
\begin{lemma}\label{lem:full_period_matrix}
Let $K/E$ be a finite unramified extension of $p$-adic local fields, let $H$ be a one-dimensional formal $\cO_E$-module of finite height over $\cO_K$.
Fix an $E$-basis $\alpha_1,\cdots,\alpha_m$ of $V(H)$ and write $\tau_i=\widetilde{\tau}_H(\alpha_i)$.
Then 
\[(\varphi_E^k(\ell_H(\tau_i)))_{0\le k\le m-1,1\le i\le m}\in\GL_m(\BcrisK).\]
\end{lemma}
\begin{proof}
By functoriality of Dieudonn\'{e} modules, endomorphisms by $\cO_E$ induce a ring homomorphism $\iota:E\to \End_{K_0}(D)$.
Define the relative Dieudonn\'{e} module by
\[D_E=\{d\in D_K:\iota(a)d=ad\textup{ for all $a\in E$}\}.\]
Since $K/E$ is unramified, $D_E$ is a vector space over $K$ with a canonical endomorphism $\phi_E$, and its dimension is $m$.
Set $f_E=[E_0:\Qp]$.
The action of $E_0$ via $\iota$ gives a decomposition
\[D=\bigoplus_{i\in\ZZ/f_E\ZZ} D_i,\quad D_i\coloneqq\{d\in D:\iota(a)d=\sigma^i(a)d\textup{ for all $a\in E_0$}\}\]
which satisfies $\phi_p(D_i)=D_{i+1}$.
Then $D_0$ is a vector space over $E$, and the unramified condition on $K/E$ makes it a vector space over $K$ of dimension $m$.
The resulting object $(D_0,\phi_p^{f_E})$ is called the relative Dieudonn\'{e} module \cite[Section 4.3.1]{Fargues2018Courbes}.
We claim that $(D_0,\phi_p^{f_E})$ is simple.
If $W\subset D_0$ is a nonzero $K$-subspace stable under $\phi_p^{f_E}$ then
\[\bigoplus_{j=0}^{f_E-1}\phi_p^j(W)\subset D\]
is a nonzero $\phi_p$-stable $K_0$-subspace.
Simplicity of $D$ forces the subspace to be the entire $D$, so we have $W=D_0$, proving simplicity of $(D_0,\phi_p^{f_E})$.
Moreover, $(D_0,\phi_p^{f_E})$ is canonically identified with $(D_E,\phi_E)$, showing simplicity of $D_E$.
Therefore, $\lambda=[\ell_H],\phi_E(\lambda),\cdots,\phi_E^{m-1}(\lambda)$ form a $K$-basis of $D_E$.

On the other hand, $V(H)$ is a crystalline representation by Proposition \ref{prop:onedim_crystalline}.
The crystalline comparison now gives a canonical isomorphism
\[D_K\otimes_K\BcrisK\isomto\Hom_{\Qp}(V(H),\BcrisK).\]
By functoriality in formal groups, it restricts to an isomorphism
\[D_E\otimes_K\BcrisK\isomto\Hom_{E}(V(H),\BcrisK).\]
Under the isomorphism, we have
\[\langle \phi_E^k\lambda,\alpha_i\rangle=\varphi_E^k\langle\lambda,\alpha_i\rangle=\varphi_E^k(\ell_H(\widetilde{\tau}_H(\alpha_i)))\]
by Proposition \ref{prop:relation_to_padic_integration}.
Choosing bases $\lambda=[\ell_H],\phi_E(\lambda),\cdots,\phi_E^{m-1}(\lambda)\in D_E$ and $\alpha_1,\cdots,\alpha_m\in V(H)$ now gives the desired invertible matrix by Proposition \ref{prop:relation_to_padic_integration}.
\end{proof}
\begin{corollary}\label{cor:period_matrix}
Let $K/E$ be a finite unramified extension, and let $H$ be a one-dimensional formal $\cO_E$-module of finite height over $\cO_K$.
Fix an $E_H$-basis $\alpha_1,\cdots,\alpha_n$ of $V(H)$, write $\tau_i=\widetilde{\tau}_H(\alpha_i)$, and set $\ell=\ell_H$.
Then
\[M=(\varphi_{E_H}^{k}(\ell(\tau_i)))_{0\le k\le n-1,1\le i\le n}\in\GL_n(\BcrisK).\]
\end{corollary}
\begin{proof}
Apply Lemma \ref{lem:full_period_matrix} after replacing $K$ by $K'=KE_H$ and $E$ by $E_H$.
Since $K'/K$ is unramified, there is a canonical identification $\mathbb{B}_{\cris,K}\cong\mathbb{B}_{\cris,K'}$.
\end{proof}

Write $\breve{K}=\breve{E}$ for the completion of the common maximal unramified extension, and let $\sigma_E$ be its continuous automorphism lifting the $q_E$-power Frobenius on the residue field and fixing $E$.
\begin{theorem}\label{thm:injectivity_ainf}
Let $K/E$ be finite unramified.
Let $H$ be a one-dimensional formal $\cO_E$-module of finite height over $\cO_K$.
Let $h=\mathrm{ht}_{\cO_E}(H)$ and $n=\dim_{E_H}V(H)$, and choose an $E_H$-basis $\alpha_1,\cdots,\alpha_n$ of $V(H)$.
Write $\tau_i=\widetilde{\tau}_H(\alpha_i)$.
The continuous $\cO_{\breve{K}}$-algebra homomorphism
\[\xi_E:\cO_{\breve{K}}\llbracket t_{k,i}:0\le k\le h-1,1\le i\le n\rrbracket\to\AinfK,\quad t_{k,i}\mapsto \varphi_E^k(\widetilde{\tau}_H(\alpha_i))\]
is injective.
\end{theorem}

Write $h$ for the relative $\cO_E$-height of $H$, and $n=h/[E_H:E]$.
By Theorem \ref{thm:openimage}, the dimension of $\Gamma$ as an $E$-analytic Lie group is $n^2[E_H:E]=nh$.
Consequently, Theorem \ref{thm:injectivity_ainf} yields an equality
\[\dim(\mathrm{im}(\mathrm{\xi_E}))=nh+1=\dim_E(\Gamma)+1\]
relating the Krull dimension of the period subalgebra of $\AinfK$ to the dimension of $\Gamma$ as an $E$-analytic Lie group.

\begin{remark}
When $K/\Qp$ is unramified of degree $f$ and $H$ is a Lubin-Tate formal group, Theorem \ref{thm:injectivity_ainf} can be deduced from \cite[Corollary 3.7]{Berger2013Multivariable}.
In this case, the period homomorphism gives an injective homomorphism
\[\cO_{\breve{K}}\llbracket t_{k,1}:0\le k\le f-1\rrbracket\to\AinfK,\quad t_{k,1}\mapsto\varphi_{\Qp}^k(\tau).\]
The proof below adapts Berger's use of Artin's independence lemma \cite[Proposition 2.4]{Berger2013Multivariable} to handle the case that $\Lambda_H$ is strictly bigger than $\cO_E$.
\end{remark}

\begin{corollary}\label{cor:model_minimal}
Let $H$ be a one-dimensional formal $\cO_K$-module of finite height over $\cO_K$.
Then the period homomorphism
\[\xi_H:\cA_H\to\AinfK\]
is injective.
\end{corollary}
Therefore, when $E=K$, the abstract model $\cA_H$ identifies with its image in the period ring $\AinfK$, and the canonical Frobenius $\varphi_K$ on $\AinfK$ restricts to the endomorphism $\varphi_K$ on $\cA_H$.
\begin{proof}
For a one-dimensional formal $\cO_K$-module over $\cO_K$, the minimal polynomial $P(X)$ of $\xi_{q_K}$ over $K$ has degree $\mathrm{ht}_{\cO_K}(H)$ \cite[Theorem 3.6.1]{Cox1974Formal}, so $d=h$.
Therefore, the period homomorphism $\xi_H$ is the restriction of $\xi_K$ to the power series subring with coefficients in $\cO_{K'}$, so it is injective by Theorem \ref{thm:injectivity_ainf}.
\end{proof}

For the remainder of the subsection, we prove Theorem \ref{thm:injectivity_ainf}.
The proof has three steps.
We first linearize sufficiently small inertia actions in a fixed $p$-adically complete crystalline lattice.
Artin independence and the crystalline comparison matrix then show that the kernel of evaluation is stable under every partial derivative.
Finally, Hasse derivatives and reduction modulo $\pi_K$ force every coefficient of a relation to vanish.

The actions of $\Gal_K$ and $\varphi_E$ on $\AinfK$ commute.
Since $K/E$ is unramified, the natural identification $\AinfK\cong W_{\cO_E}(\OCb)$ identifies $\varphi_E$ with the Witt lift of $x\mapsto x^{q_E}$.
This commutes with $\Gal_K$ by functoriality.
The restriction of $\varphi_E$ on $\cO_K$ acts as $\sigma_E\vert_K$.

For $1\le a\le n$ and $0\le k\le h-1$, put $\tau_a=\widetilde{\tau}_H(\alpha_a)$ and $\tau_{k,a}=\varphi_E^k(\tau_a)$.
Fix indices $1\le i,j\le n$.
By Theorem \ref{thm:openimage}, there is an integer $N_0$ which depends on the formal group $H$ and the basis $\alpha_1,\cdots,\alpha_n$ of $V(H)$ such that, for every $r\ge N_0$ and $\lambda\in\Lambda_H$, there is $\gamma=\gamma_{r,\lambda,i,j}\in I_K$ satisfying 
\[\rho(\gamma)=\Id+p^r\lambda E_{ji}.\]
With the chosen basis, this means
\[\gamma\cdot \alpha_a=\begin{cases}
\alpha_i+p^r\lambda \alpha_j&a=i\\ \alpha_a&a\neq i
\end{cases}\]
We first compute the first-order change of $\tau_{k,i}$ as $r\to\infty$, using a fixed $p$-adically separated crystalline lattice.

Define a crystalline lattice by $\AcrisK=\Acris\widehat{\otimes}_{\cO_{K_0}}\cO_K$.
This ring is $p$-adically separated and complete, and is $p$-torsion free.
Set $f_E=[E_0:\Qp]$.
Relative Frobenius extends to $\AcrisK$ by $\varphi_E(a\otimes c)=\varphi_p^{f_E}(a)\otimes\sigma_E(c)$, and it preserves $\AcrisK$, fixes $p$ and so preserves every $p^r\AcrisK$.

For $x\in\frakm_{\Cb}$ and $r\ge1$ such that $v_{\Cb}(x^r)\ge p$, 
\begin{equation}\label{eq:power_acrisk}
    [x]^r\in p\AcrisK.
\end{equation}
Put $\xi=[p^\flat]-p\in\ker\theta$.
Since $\xi^p\in p\Acris$, we have $[p^\flat]^p\in p\AcrisK$.
The assertion (\ref{eq:power_acrisk}) now follows because $v_{\Cb}(p^\flat)=1$.

\begin{lemma}\label{lem:tau_acrisk}
Put $e=e(K/\Qp)$, $h_0=\mathrm{ht}(H)$, and fix $N\ge0$.
Set 
\[N_1=p^{Nh_0+1+\lceil \log_p(e)\rceil}.\]
Let $\alpha\in p^{-N}T(H)$ and $\tau=\widetilde{\tau}_H(\alpha)$.
Then
\begin{equation}
    \tau^{r}\in p\AcrisK
\end{equation}
for every $r\ge N_1$.
\end{lemma}
\begin{proof}
Let $\tau=\widetilde{\tau}_H(\alpha)$, and let $\overline{\tau}$ be its reduction modulo $\pi_K$.
Then we can write $\tau=[\overline{\tau}]+\pi_Kz$ with $z\in\AinfK$.
If $\alpha\in T(H)$, then we have $\theta_K(\tau)=0$, so $\overline{\tau}^\#\in \pi_K\cO_C$, showing that 
\[v_{\Cb}(\overline{\tau})=v_C(\overline{\tau}^\#)\ge v_C(\pi_K)=\frac{1}{e}.\]
Since $p^N\alpha\in T(H)$, the reduced exponential period of $p^N\alpha$ has valuation at least $1/e$.
The reduction of $[p^N]_H$ has order $p^{Nh_0}$ and nonzero leading coefficient in the residue field.
Therefore,
\[v_{\Cb}([p^N]_H(\overline{\tau}))=p^{Nh_0}v_{\Cb}(\overline{\tau})\ge\frac{1}{e}\]
and consequently $v_{\Cb}(\overline{\tau})\ge1/(ep^{Nh_0})$.
For $i\ge1$, we have
\[\tau^{p^i}-[\overline{\tau}]^{p^i}-\pi_K^{p^i}z^{p^i}\in p\AcrisK.\]
By (\ref{eq:power_acrisk}), the term $[\overline{\tau}]^{p^i}$ belongs to $p\AcrisK$ if $p^iv_{\Cb}(\overline{\tau})\ge p$, and the other term $\pi_K^{p^i}z^{p^i}$ belongs to $p\AcrisK$ if $p^i\ge e$.
These conditions are ensured by $p^i\ge ep^{Nh_0+1}$.
\end{proof}
The power series $[p]_H(T)$ has linear term $pT$, so its reduction modulo $p$ has order at least 2.
Therefore, for each $r\ge1$, the reduction of $[p^r]_H$ modulo $p$ has order at least $2^r$, and so
\[[p^r]_H(\tau)\in (\tau^{2^r},p)\subset\AcrisK.\]
Consequently, by putting $s_1=\lceil \log_2(N_1)\rceil$, Lemma \ref{lem:tau_acrisk} implies that, for every $r\ge s_1$,
\begin{equation}\label{eq:tau_pacris}
    [p^r]_H(\tau)\in p\AcrisK.
\end{equation}

Let $\ell=T+\sum_{m=2}^\infty b_mT^m$ denote the normalized formal logarithm of $H$.
Since the normalized invariant differential $\ell'(T)dT$ has integral coefficients, we have $mb_m\in \cO_K$ for every $m$.
Therefore, Lemma \ref{lem:tau_acrisk} implies
\[b_m\tau^m\in p^{\lfloor m/N_1\rfloor-v_p(m)}\AcrisK.\]
The exponent is uniformly bounded from below and diverges as $m\to\infty$, so $\ell(\tau)$ converges inside a fixed fractional lattice of $\AcrisK$.

\begin{lemma}\label{lem:p^r_calculation}
Set $s_1=\lceil \log_2(N_1)\rceil$.
For every $r\ge s_1$ and $\lambda\in\Lambda_H$,
\[p^r\lambda\ell(\tau)\in p^{r-s_1+1}\AcrisK,\quad [p^r\lambda]_H(\tau)-p^r\lambda\ell(\tau)\in p^{2r-2s_1+1}\AcrisK.\]
\end{lemma}

\begin{proof}
We begin by making a general observation about evaluating a formal logarithm.
Fix $r\ge1$.
If $z\in p^r\AcrisK$, then
\begin{equation}\label{eq:ell_estimate}
    \ell(z)-z\in p^{2r-1}\AcrisK,\quad \ell(z)\in p^r\AcrisK.
\end{equation}
In fact, for each $m\ge2$ we have
\[b_mz^m\in p^{rm-v_p(m)}\AcrisK.\]
An elementary argument shows that, for $r\ge1$ and $m\ge2$,
\[rm-v_p(m)\ge 2r-1\]
showing the first inclusion in (\ref{eq:ell_estimate}).
The second inclusion is a consequence of the first.

For every $\lambda\in\Lambda_H$, $[\lambda]_H\in\cO_{K'}\llbracket T\rrbracket$ and $\cO_{K'}\subset\AinfK\subset\AcrisK$ because $K'/K$ is unramified.
Since $[p]_H$ has linear term $pT$, for every $r\ge1$,
\[[\lambda]_H(p^r\AcrisK)\subset p^r\AcrisK,\quad [p]_H(p^r\AcrisK)\subset p^{r+1}\AcrisK.\]
For $r\ge s_1$, (\ref{eq:tau_pacris}) implies
\begin{equation}\label{eq:pnlambda_estimate}
    [p^r\lambda]_H(\tau)=[p^{r-s_1}\lambda]_H([p^{s_1}]_H(\tau))\in p^{r-s_1+1}\AcrisK.
\end{equation}
On the other hand, 
\begin{equation}\label{eq:ell_pn}
    \ell([p^r\lambda]_H(\tau))=p^r\lambda\ell(\tau).
\end{equation}

Fix $r\ge s_1$ and $\lambda\in\Lambda_H$.
By (\ref{eq:pnlambda_estimate}), we have
\[[p^r\lambda]_H(\tau)\,\,\in p^{r-s_1+1}\AcrisK.\]
The second inclusion in (\ref{eq:ell_estimate}) implies
\[\ell([p^r\lambda]_H(\tau))=p^r\lambda\ell(\tau)\,\,\in p^{r-s_1+1}\AcrisK\]
where we also used (\ref{eq:ell_pn}).
Applying the first inclusion in (\ref{eq:ell_estimate}) to $[p^r\lambda]_H(\tau)$ now gives
\[p^r\lambda\ell(\tau)-[p^r\lambda]_H(\tau)\in p^{2r-2s_1+1}\AcrisK,\]
showing the assertion.
\end{proof}

\begin{lemma}\label{lem:linearization}
For every $r\ge \max\{s_1,N_0\}$ and $\lambda\in\Lambda_H$, $\gamma=\gamma_{r,\lambda,i,j}$ satisfies
\[\gamma\cdot\tau_{k,i}- \tau_{k,i}\in p^{r-s_1+1}\AcrisK\]
and
\[\gamma\cdot\tau_{k,i}-\tau_{k,i}-p^r\sigma_E^k(\lambda)\varphi_E^k(\ell(\tau_j))a_{k,i}\in p^{2r-2s_1+1}\AcrisK.\]
Here, 
\[a_{k,i}=\varphi_E^k(a_i),\quad a_i\coloneqq \frac{\partial H}{\partial Y}(\tau_i,0).\]
Moreover, every $a_{k,i}$ is a unit of $\AinfK$.
\end{lemma}

\begin{proof}
Equivariance gives
\[\gamma\cdot\tau_{k,i}=\varphi_E^k(H(\tau_i,[p^r\lambda]_H(\tau_j))).\]
Therefore,
\[\gamma\cdot\tau_{k,i}- \tau_{k,i}=\varphi_E^k(a_{i})\varphi_E^k([p^r\lambda]_H(\tau_j))+\sum_{m=2}^\infty \varphi_E^k(a_i^{(m)})(\varphi_E^k([p^r\lambda]_H(\tau_j)))^m\]
where 
\[a_i^{(m)}\coloneqq \partial_Y^{[m]}H(\tau_i,0).\]
By the containment (\ref{eq:pnlambda_estimate}) in Lemma \ref{lem:p^r_calculation},
\[[p^r\lambda]_H(\tau_j)\in p^{r-s_1+1}\AcrisK.\]
Therefore,
\[\gamma\cdot\tau_{k,i}- \tau_{k,i}-\varphi_E^k(a_{i})\varphi_E^k([p^r\lambda]_H(\tau_j))\in p^{2r-2s_1+2}\AcrisK.\]
The second inclusion in Lemma \ref{lem:p^r_calculation} implies
\[\gamma\cdot\tau_{k,i}- \tau_{k,i}-\varphi_E^k(a_{i})\varphi_E^k(p^r\lambda\ell(\tau_j))\in p^{2r-2s_1+1}\AcrisK.\]
In particular, 
\[\gamma\cdot\tau_{k,i}- \tau_{k,i}\in p^{r-s_1+1}\AcrisK.\]

It remains to check that $a_{k,i}$ is a unit.
The identity $\ell(H(X,Y))=\ell(X)+\ell(Y)$ gives (see \cite[Proposition 1]{Honda1968Formal})
\[\ell'(X)\frac{\partial H}{\partial Y}(X,0)=1.\]
Therefore, evaluating the invertible power series $\frac{\partial H}{\partial Y}(X,0)$ at $\tau_i$ gives a unit $a_i\in\AinfK$.
Its $\varphi_E$-iterate $a_{k,i}$ is therefore a unit as well.
\end{proof}

\begin{proof}[Proof of Theorem \ref{thm:injectivity_ainf}]
Let $\boldsymbol{\tau}=(\tau_{k,r})_{0\le k\le h-1,1\le r\le n}$ so that $\xi_E(g)=g(\boldsymbol{\tau})$.
Let $g\in\ker(\xi_E)$, so $g(\boldsymbol{\tau})=0$.
Choose $N\ge0$ such that $p^N\alpha_i\in T(H)$ for all $1\le i\le n$.
Set $N_1=p^{Nh_0+1+\lceil \log_p(e)\rceil}$ and $s_1=\lceil \log_2(N_1)\rceil$ as in Lemma \ref{lem:tau_acrisk} and Lemma \ref{lem:p^r_calculation}.
For a sufficiently large $r$ and every $\lambda\in\Lambda_H$, choose $\gamma=\gamma_{r,\lambda,i,j}\in I_K$ as above.
Inertia fixes $\cO_{\breve{K}}$ pointwise, so equivariance and continuity of evaluation give
\[0=\gamma\cdot g(\boldsymbol{\tau})=g(\gamma\cdot\boldsymbol{\tau})\]
Put $D_{k,i}=(\partial g/\partial t_{k,i})(\boldsymbol{\tau})$.
Consider the Taylor expansion of $g(\gamma\cdot\boldsymbol{\tau})-g(\boldsymbol{\tau})$.
Every component of $\gamma\boldsymbol{\tau}-\boldsymbol{\tau}$ lies inside $p^{r-s_1+1}\AcrisK$ by Lemma \ref{lem:linearization}, so every term in Taylor expansion of degree at least two lies in $p^{2r-2s_1+2}\AcrisK$.
Therefore, we can approximate the Taylor series with linear terms; Lemma \ref{lem:linearization} gives a formula for the linearization modulo $p^{2r-2s_1+1}\AcrisK$, implying
\begin{align*}
    g(\gamma\cdot\boldsymbol{\tau})-g(\boldsymbol{\tau})-p^r\sum_{k=0}^{h-1} D_{k,i} \sigma_E^k(\lambda)\varphi_E^k(\ell(\tau_j))a_{k,i}\in p^{2r-2s_1+1}\AcrisK
\end{align*}
Since evaluation at $\boldsymbol{\tau}$ and $\gamma\cdot\boldsymbol{\tau}$ both vanish, division by $p^r$ in $\AcrisK[1/p]$ gives
\[\sum_{k=0}^{h-1}D_{k,i}\sigma_E^k(\lambda)\varphi_E^k(\ell(\tau_j))a_{k,i}\in p^{r-2s_1+1}\AcrisK.\]
The left-hand side is independent of $r$.
Since $\AcrisK$ is $p$-adically separated, letting $r\to\infty$ yields
\begin{equation}\label{eq:vanishing}
    \sum_{k=0}^{h-1}D_{k,i}\sigma_E^k(\lambda)\varphi_E^k(\ell(\tau_j))a_{k,i}=0
\end{equation}
for all $\lambda\in\Lambda_H$.
Put $b=[E_H:E]$.
The extension $E_H/E$ is unramified.
Since $h=\dim_EV(H)$ and $n=\dim_{E_H}V(H)$, we have $h=bn$.
The restrictions $\sigma_E^s\vert_{E_H}$ for $0\le s\le b-1$ are the distinct elements of $\Gal(E_H/E)$.
Grouping the terms according to $k\pmod{b}$, we obtain
\[\sum_{s=0}^{b-1}\left(\sum_{\substack{0 \le k \le h-1 \\ k \equiv s \;(\mathrm{mod}\,b)}}D_{k,i}\varphi_E^k(\ell(\tau_j))a_{k,i}\right)\sigma_E^s(\lambda)=0\]
By multiplying elements of $E_H$ by sufficiently large powers of a uniformizer of $E$, the identity extends to every $\lambda\in E_H$.
Independence of the distinct $E$-embeddings of $E_H$ (Artin's lemma) therefore gives 
\begin{equation}\label{eq:artin}
    \sum_{\substack{0 \le k \le h-1 \\ k \equiv s \;(\mathrm{mod}\,b)}}D_{k,i}\varphi_E^k(\ell(\tau_j))a_{k,i}=0
\end{equation}
for every $0\le s\le b-1$.

For $0\le s\le b-1$, we put
\[M_s=(\varphi_E^{s+rb}(\ell(\tau_j)))_{1\le j\le n,0\le r\le n-1},\quad \mathbf{v}_{s,i}=(D_{s+rb,i}a_{s+rb,i})_{0\le r\le n-1}^\intercal\]
Varying $j$ in (\ref{eq:artin}) gives a matrix equality
\[M_s\mathbf{v}_{s,i}=0.\]
By Corollary \ref{cor:period_matrix},
\[M_0=(\varphi_E^{rb}(\ell(\tau_j)))_{1\le j\le n,0\le r\le n-1}\in\GL_n(\BcrisK).\]
The endomorphism $\varphi_E$ extends to $\BcrisK$, acting on $K$ by $\sigma_E\vert_K$.
Hence, $M_s=\varphi_E^s(M_0)$ is invertible with inverse $\varphi_E^s(M_0^{-1})$.
It follows that
\[\mathbf{v}_{s,i}=0\]
for every $0\le s\le b-1$.
Therefore,
\[D_{k,i}a_{k,i}=0\]
for every $0\le k\le h-1$ and $1\le i\le n$.
Since each $a_{k,i}$ is a unit by Lemma \ref{lem:linearization},
\[D_{k,i}=0\]
for every $k$ and $i$.

We have proved that $\ker(\xi_E)$ is stable under every ordinary partial derivative $\partial_{k,i}$.
For a multi-index $I=(m_{k,i})$, let $D^{[I]}$ denote the Hasse derivative, so that 
\[\partial^I=I!D^{[I]},\quad I!=\prod_{k,i}m_{k,i}!.\]
Iterating differential stability gives
\[I!\xi_E(D^{[I]}g)=0.\]
Since $\AinfK$ is a domain, multiplication by $I!$ is injective.
Hence,
\[D^{[I]}g\in\ker(\xi_E)\]
for every $I$.
The composite
\[\cO_{\breve{K}}\llbracket t_{k,i}\rrbracket\xrightarrow{\xi_E}\AinfK\to\OCb/\frakm_{\Cb}\cong\Fpbar\]
sends every variable $t_{k,i}$ to zero, and reduces the coefficients modulo $\pi_K$.
Therefore, the constant term of every element of $\ker(\xi_E)$ belongs to $\pi_K\cO_{\breve{K}}$.
Since every $D^{[I]}g$ lies in the kernel, each coefficient $D^{[I]}g(\mathbf{0})$ is divisible by $\pi_K$.
Therefore, $g=\pi_Kg_1$ for some $g_1\in\cO_{\breve{K}}\llbracket t_{k,i}\rrbracket$.
Since $\AinfK$ is $\pi_K$-torsion free, $\xi_E(g_1)=0$.
Repeating this argument gives $g\in\pi_K^m\cO_{\breve{K}}\llbracket t_{k,i}\rrbracket$ for every $m$.
The coefficient ring $\cO_{\breve{K}}$ is $\pi_K$-adically separated, so $g=0$.
This proves the theorem.
\end{proof}

\newpage
\section{Potentially Lubin-Tate $(\varphi,\Gamma)$-modules}\label{sec:pot_lubintate}
Let $L/K$ be a finite unramified extension of degree $f$ and fix a uniformizer $\pi_K$ of $K$.
We study the formal $\cO_K$-module $H$ constructed below, whose base change to $\cO_L$ is a Lubin-Tate formal $\cO_L$-module for $\pi_K$.
Its torsion extension $K_\infty=K(H[p^\infty](\overline{K}))$ has Galois group $\Gamma=\Gal(K_\infty/K)\cong\cO_L^\times\rtimes\Gal(L/K)$.
The construction of Section \ref{sec:phigamma} gives a multivariable theory for $\cO_K$-representations of $\Gal_K$, while a one-variable base ring describes semilinear $\cO_L$-representations.
We compare these theories at the end of the section.
The comparison has two parts.
For the common Frobenius $\varphi_L=\varphi_K^f$, extension of scalars from the one-variable ring $\cC_H$ to the multivariable ring $R_H$ is an equivalence.
Passing from $\varphi_K$ to $\varphi_L$, however, corresponds to extending representation coefficients from $\cO_K$ to $\cO_L$, with its natural semilinear Galois action.

\subsection{Potentially Lubin-Tate formal groups}
Let $K/\Qp$ be a finite extension, and put $q=\#\kappa_K$.
We recall the Lubin-Tate construction \cite{Lubin1965Formal} before explaining its descent along an unramified extension.
Fix a uniformizer $\pi_K$ of $K$, and put
\[\cF_{\pi_K}=\{P(T)\in \cO_K\llbracket T\rrbracket_0: P'(0)=\pi_K,\,\, P(T)\equiv T^q\pmod{\pi_K}\}.\]
\begin{lemma}[{\cite[Lemma 1]{Lubin1965Formal}}]\label{lem:lubintate}
Let $P(T),Q(T)\in \cF_{\pi_K}$, and let $a_1,\cdots,a_n\in \cO_K$.
There is a unique series $F(X_1,\cdots,X_n)\in \cO_K\llbracket X_1,\cdots,X_n\rrbracket$ satisfying the following two conditions.
\begin{enumerate}
    \item $F(X_1,\cdots,X_n)\equiv a_1X_1+\cdots+a_nX_n\pmod{\deg 2}$
    \item $P(F(X_1,\cdots,X_n))=F(Q(X_1),\cdots,Q(X_n))$
\end{enumerate}
\end{lemma}
Fix $P\in\cF_{\pi_K}$.
Applying Lemma \ref{lem:lubintate} with $P=Q$, $n=2$, and the linear term $X+Y$ gives a unique series $H_P(X,Y)$ satisfying 
\[H_P(X,Y)\equiv X+Y\pmod{\deg 2},\quad P(H_P(X,Y))=H_P(P(X),P(Y)).\]
The series $H_P$ defines a one-dimensional formal $\cO_K$-module over $\cO_K$ with $[\pi_K]_{H_P}=P(T)$, whose isomorphism class is independent of the choice of $P\in\cF_{\pi_K}$ \cite[Theorem 1]{Lubin1965Formal}.
We call $H_P$ a \textit{Lubin-Tate formal group} associated to $\pi_K$.

Let $L/K$ be a finite unramified extension of degree $f$, so its residue field has cardinality $q_L=q^f$, and let $\pi_K$ be a uniformizer of $K$.
Choose an element $P(T)\in\cO_K\llbracket T\rrbracket_0$ with $P'(0)=\pi_K$ and $P(T)\equiv T^{q_L}\pmod{\pi_K}$.
Since $\pi_K$ is also a uniformizer of $L$, the Lubin-Tate construction (Lemma \ref{lem:lubintate} applied to $L$) produces a formal $\cO_L$-module $H_P$ associated to $\pi_K$.
\begin{lemma}\label{lem:potentially_lubintate}
With $L/K$ and $\pi_K$ fixed as above, the group law $H_P$ and its $\cO_K$-action are defined over $\cO_K$, and its isomorphism class over $\cO_K$ is independent of the choice of $P$.
\end{lemma}
\begin{proof}
Let $\sigma\in\Gal(L/K)$ be the automorphism inducing $x\mapsto x^q$ on the residue field of $L$.
Since $P$ has coefficients in $\cO_K$, applying $\sigma$ to the coefficients of $H_P$ preserves its defining identities.
Uniqueness in Lemma \ref{lem:lubintate} therefore gives $H_P^\sigma=H_P$, so $H_P$ has coefficients in $\cO_K$.
Let $P$ and $Q$ be two series satisfying the conditions.
For every $a\in \cO_L$, Lemma \ref{lem:lubintate} gives a unique series $[a]_{P,Q}\in \cO_{L}\llbracket T\rrbracket_0$ satisfying
\[[a]_{P,Q}'(0)=a,\quad P\circ [a]_{P,Q}=[a]_{P,Q}\circ Q\]
The series $([a]_{P,Q})^\sigma$ defines a homomorphism $H_Q\to H_P$ \cite[Theorem 1 (8)]{Lubin1965Formal}.
Applying $\sigma$ to the defining identities and using uniqueness gives
\begin{equation}\label{eq:sigma}
    ([a]_{P,Q})^\sigma=[\sigma(a)]_{P,Q}.
\end{equation} 
If $a\in \cO_K$, then $[a]_{P,Q}$ has coefficients in $\cO_K$.
Taking $Q=P$ shows that the $\cO_K$-action on $H_P$ descends to $\cO_K$.
Finally, $[1]_{P,Q}$ gives an isomorphism between $H_P$ and $H_Q$ over $\cO_K$ with inverse $[1]_{Q,P}$.
\end{proof}

We call a group obtained in Lemma \ref{lem:potentially_lubintate} a \textit {potentially Lubin-Tate} formal group for $L/K$ associated to $\pi_K$.
Fix the generator $\sigma\in \Gal(L/K)$ lifting the $q$-power Frobenius.

\begin{lemma}\label{lem:pot_lubintate_galois_group}
Let $H$ be a potentially Lubin-Tate formal group over $\cO_K$ associated with $L/K$.
Set $K_\infty=K(H[p^\infty](\overline{K}))$.
There is an isomorphism of profinite groups
\[\Gal(K_\infty/K)\cong (\cO_{L})^\times\rtimes_{\iota} (\ZZ/f \ZZ),\quad \iota(i)(a)=\sigma^i(a)\]
whose construction depends on a choice of $\cO_L$-basis of $T(H)$.
\end{lemma}
\begin{proof}
Put $K_1=K(H[\pi_K](\overline{K}))$.
Let $\alpha$ be a nonzero element of $H[\pi_K]$, a rank-one $\cO_L/(\pi_K)$-module.
For $g\in\Gal_{K_1}$ and $a\in\cO_L$,
\[[g(a)]_H(\alpha)=g([a]_H(\alpha))=[a]_H(\alpha)\]
so $[g(a)-a]_H(\alpha)=0$.
The annihilator of $\alpha$ in $\cO_L$ is $(\pi_K)$, so $g(a)-a\in\pi_K\cO_L$ for every $a\in\cO_L$.
Therefore, $g$ acts trivially on $\kappa_L$.
Since $L/K$ is unramified, $g$ acts trivially on $L$, proving the containment $L\subset K_1$.

Therefore, $K_\infty/L$ is the Lubin-Tate torsion extension associated to $\pi_K$.
Choose an $\cO_L$-basis $\varpi$ of $T(H)$.
Lubin-Tate theory identifies \cite[Theorem 2]{Lubin1965Formal}
\[\Gal(K_\infty/L)\to (\cO_{L})^\times,\quad \gamma(\varpi)=\chi(\gamma)\varpi.\]
Choose any lift $s\in\Gal(K_\infty/K)$ of $\sigma$, and write $s(\varpi)=u\varpi$ with $u\in\cO_L^\times$.
Let $\gamma_{u^{-1}}\in\Gal(K_\infty/L)$ act on $\varpi$ by $u^{-1}$, and put $\widetilde{\sigma}=\gamma_{u^{-1}}s$.
Then $\widetilde{\sigma}$ restricts to $\sigma$ and fixes $\varpi$, and it is the unique lift with these properties.
The automorphism $\widetilde{\sigma}^f$ fixes both $L$ and $\varpi$, so it is the identity.
Therefore, $\sigma\mapsto\widetilde{\sigma}$ defines a section of
\[1\to\Gal(K_\infty/L)\to \Gal(K_\infty/K)\to\Gal(L/K)\to 1.\]
This section expresses $\Gal(K_\infty/K)$ as a semidirect product of $\Gal(K_\infty/L)$ with $\Gal(L/K)$.
It remains to compute the action of $\widetilde{\sigma}$ on $\Gal(K_\infty/L)$ by conjugation.
For $a\in (\cO_{L})^\times$, let $\gamma_a\in \Gal(K_\infty/L)$ be characterized by $\gamma_a(\varpi)=a\varpi$.
By (\ref{eq:sigma}),
\begin{align*}
    \tsigma\gamma_a\tsigma^{-1}(\varpi)= \tsigma \gamma_a(\varpi)=\sigma(a)\varpi.
\end{align*}
Hence, $\tsigma\gamma_a\tsigma^{-1}=\gamma_{\sigma(a)}$, giving the asserted semidirect product action.
\end{proof}
\subsection{The multivariable theory over $R_H$}\label{subsec:plt_phigamma}
Let $L/K$ be a finite unramified extension of degree $f$, put $q=\#\kappa_K$, and let $H$ be a potentially Lubin-Tate formal group associated to $\pi_K$.
Set $K_\infty=K(H[p^\infty](\overline{K}))$ and $\Gamma=\Gal(K_\infty/K)$.
The construction gives $\cO_L\subset \Lambda_H$ and $\mathrm{ht}(H)=[L:\Qp]$.
Since $[E_H:\Qp]\le\mathrm{ht}(H)$, this forces $E_H=L$.
Therefore, $\Lambda_H=\cO_L$.
In summary, $H$ has $\cO_K$-height $f$, ordinary height $[L:\Qp]$, and absolute height one.

Fix an $\cO_L$-basis $\alpha$ of $T(H)$, and put $\tau=\tau_H(\alpha)$.
Since the $q$-power Frobenius of the special fiber satisfies $\xi_{q_K}^f=[\pi_K]$, its minimal polynomial over $K$ is the Eisenstein polynomial $X^f-\pi_K$.
By Corollary \ref{cor:model_minimal}, the canonical period homomorphism is injective, and in the coordinates determined by $\alpha$, is given by
\[\cA_H\cong\cO_L\llbracket t_{k,1}:0\le k\le f-1\rrbracket\hookrightarrow\AinfK,\quad t_{k,1}\mapsto \varphi_K^k(\tau).\]
The endomorphism $\varphi_K$ commutes with its $\Gamma$-action.
The endomorphism $\varphi_K$ restricts to $\sigma$ on $\cO_L$ and satisfies
\[\varphi_K(t_{k,1})=\begin{cases}
    t_{k+1,1}&0\le k\le f-2\\
    [\pi_K]_H(t_{0,1})&k=f-1
\end{cases}\]

Define $\chi_\alpha:\Gamma\to(\cO_L)^\times$ by $\gamma(\alpha)=\chi_\alpha(\gamma)\alpha$.
It satisfies $\chi_\alpha(\gamma\delta)=\chi_\alpha(\gamma)\gamma(\chi_\alpha(\delta))$ and its restriction to $\Gamma_L=\Gal(K_\infty/L)$ is the Lubin-Tate character.
For $\gamma\in\Gamma$, its action on $\cA_H$ restricts to $\gamma\vert_L$ on $\cO_L$ and satisfies
\[\gamma\cdot t_{k,1}=[\sigma^k(\chi_\alpha(\gamma))]_H(t_{k,1}).\]
The subring $o_H\subset\OCb$ is identified with $\FF_{q^f}\llbracket t\rrbracket$, where $t$ corresponds to $\overline{\tau}_H(\alpha)$.
Reduction of the period homomorphism gives
\[\cA_H\to o_H,\quad t_{k,1}\mapsto t^{q^k}\,(0\le k<f).\]
Writing $s_k=t_{k,1}-t_{0,1}^{q^k}$ for $1\le k\le f-1$, we obtain
\[\frp_H=\ker(\cA_H\to o_H)=(\pi_K,s_1,\cdots,s_{f-1}).\]
Recall that $R_H=((\cA_H)_{\frp_H})^\wedge$ with completion taken at the maximal ideal.
Proposition \ref{prop:coefficient_ring} gives an isomorphism of complete local rings
\[R_H\cong \left(\cO_L\llbracket t\rrbracket [\tfrac{1}{t}]\right)^\wedge\llbracket s_1,\cdots,s_{f-1}\rrbracket\]
where the completion is $\pi_K$-adic.
We will use the resulting coefficient ring inclusion in Section \ref{subsec:comparison_plt}.

The field $k_H=\Frac(o_H)\cong\FF_{q^f}(\!(t)\!)$ is complete for the valuation inherited from $\Cb$, which is a positive multiple of the $t$-adic valuation, so it is henselian.
Therefore, $k_{H,K}=k_H$, and the construction gives $R_{H,K}=R_H$.
Theorem \ref{thm:phigamma_equiv} now specializes as follows.
\begin{corollary}\label{cor:pLT}
There is an equivalence of categories
\[\DD_{H,K}:\Rep_{\cO_K}(\Gal_K)\simeq\textup{$(\varphi,\Gamma)$-Mod}_{R_H},\quad V\mapsto (V\otimes_{\cO_K}R_H^\sep)^{\Gal_{K_\infty}}.\]
\end{corollary}

\subsection{The one-variable theory over $\cC_H$}
Keep the notation of the previous subsection; in particular $L/K$ is unramified of degree $f$ and $H$ is potentially Lubin-Tate.
We now describe continuous semilinear $\cO_L$-representations of $\Gal_K$ using a one-variable base ring and the Frobenius $\varphi_K^f$.

\begin{definition}
A \textit{semilinear $\cO_{L}$-representation of $\Gal_{K}$} is a finitely generated $\cO_{L}$-module $V$ with a continuous $\Gal_{K}$-action satisfying $g(av)=g(a)g(v)$ for $g\in\Gal_K$, $a\in\cO_L$ and $v\in V$.
Let $\Rep^{\mathrm{semi}}_{\cO_{L}}(\Gal_K)$ denote this category, with $\cO_L$-linear $\Gal_K$-equivariant maps as morphisms.
\end{definition}

Recall that $\cT_H=\cO_{L}\llbracket T(H)\rrbracket_{H,\cO_L}$.
The period homomorphism
\[\cT_H\to\AinfK,\quad [\alpha]\mapsto\tau_H(\alpha)\]
is injective and $\Gal_K$-equivariant.
Since $L\subset K_\infty$ and $\Gal_{K_\infty}$ acts trivially on $T(H)$, the $\Gal_K$-action on $\cT_H$ factors through $\Gamma$.
If $f>1$, then $\cT_H\cong\cO_L\llbracket t\rrbracket$ is not stable under $\varphi_K$ because, under the embedding $\cO_L\llbracket t\rrbracket\subset\cA_H,t\mapsto t_{0,1}$, we have $\varphi_K(t_{0,1})=t_{1,1}\not\in\cO_L\llbracket t_{0,1}\rrbracket$.
Nonetheless, the form of the minimal polynomial $P(X)=X^f-\pi_K$ implies
\[\varphi_K^f(t_{0,1})=[\pi_K]_H(t_{0,1})\]
so $\cT_H$ is stable under $\varphi_K^f$.
Retain the $\cO_L$-basis $\alpha$ of $T(H)$ and the cocycle $\chi_\alpha:\Gamma\to\cO_L^\times$ from Section \ref{subsec:plt_phigamma}.
Put $\varphi_L=\varphi_K^f$, and let $a\in \cO_L$.
The commuting actions of $\varphi_L$ and $\Gamma$ satisfy
\[\varphi_L(t)=[\pi_K]_H(t),\quad \varphi_L(a)=a\]
and
\[\gamma(t)=[\chi_\alpha(\gamma)]_H(t),\quad \gamma(a)=\gamma\vert_L(a).\]
Let $\cC_H=((\cT_H)_{(\pi_K)})^\wedge$ where the completion is $\pi_K$-adic.
Then
\[\cC_H\cong \left(\cO_L\llbracket t\rrbracket [\tfrac{1}{t}]\right)^\wedge\]
is a complete discrete valuation ring with uniformizer $\pi_K$ and residue field $k_H=\FF_{q^f}(\!(t)\!)$.
As in Example \ref{ex:lubintate}, the period map extends to an injection $\cC_H\hookrightarrow W_{\cO_L}(\Cb)\cong W_{\cO_K}(\Cb)$, identifying $\cC_H$ with the usual Lubin-Tate base ring for $L$ \cite{Ren2009Galois}.
Since $H$ is defined over $\cO_K$, the action of $\Gamma_L$ extends to $\Gamma\cong\Gamma_L\rtimes\Gal(L/K)$, with the splitting determined in Lemma \ref{lem:pot_lubintate_galois_group}.

\begin{proposition}\label{prop:equiv_plt}
Let $L/K$ be a finite unramified extension of $p$-adic local fields.
Let $H/\cO_K$ be a potentially Lubin-Tate group for $L/K$, and set $K_\infty=K(H[p^\infty](\overline{K}))$ and $\Gamma=\Gal(K_\infty/K)$.
The functor
\[\DD_{\cC_H}(V)=(V\otimes_{\cO_L}\cC_H^\sep)^{\Gal_{K_\infty}}\]
induces an equivalence
\[\DD_{\cC_H}:\Rep^{\mathrm{semi}}_{\cO_{L}}(\Gal_K)\isomto\textup{$(\varphi_L,\Gamma)$-Mod}_{\cC_H}\]
where the module category has the continuity condition of Definition \ref{def:phigammamod}.
\end{proposition}
\begin{proof}
The pair $(\cC_H,\varphi_L)$ is a flat $F$-dynamical system over $\cO_L$.
Theorem \ref{thm:FDS_cat_equiv} therefore gives the equivalence
\[\DD_{\cC_H}:\Rep_{\cO_L}(\Gal_{K_\infty})\simeq \textup{$\varphi_L$-Mod}_{\cC_H}, \quad V\mapsto (V\otimes_{\cO_L}\cC_H^\sep)^{\Gal_{K_\infty}}.\]
By the proof of Theorem \ref{thm:phigamma_equiv}, the functor $\DD_{\cC_H}$ gives the asserted equivalence. 
\end{proof}
Let $\Gamma_L=\Gal(K_\infty/L)\subset\Gamma$.
Restricting the base ring action to $\Gamma_L$ gives the classical Lubin-Tate equivalence for $L$ \cite{Ren2009Galois}
\[\DD_{\LT}:\Rep_{\cO_{L}}(\Gal_{L})\simeq \textup{$(\varphi_L,\Gamma_L)$-Mod}_{\cC_H}.\]
The invariant formula gives a natural isomorphism
\[\res_{\Gamma,\Gamma_L}\circ\DD_{\cC_H}\simeq\DD_\LT\circ\res_{K,L}\]
where the two restriction functors are those associated with $\Gamma_L\subset\Gamma$ and $\Gal_L\subset\Gal_K$, making the following diagram commute.
\begin{center}
\begin{tikzcd}
\Rep^{\mathrm{semi}}_{\cO_{L}}(\Gal_K)\arrow[d,"\res_{K,L}"']\arrow[r,"\DD_{\cC_H}"]&\textup{$(\varphi_L,\Gamma)$-Mod}_{\cC_H}\arrow[d,"\res_{\Gamma,\Gamma_L}"]\\
\Rep_{\cO_{L}}(\Gal_{L})\arrow[r,"\DD_{\LT}"]&\textup{$(\varphi_L,\Gamma_L)$-Mod}_{\cC_H}
\end{tikzcd}
\end{center}

\begin{remark}\label{rmk:tension_with_berger}
Under the hypotheses of Appendix \ref{app:onevariable}, an infinite group acting faithfully by coefficient-linear invertible power series and commuting with the specified nonzero contracting power series is commutative.
For the formal group considered here, $\cT_H\cong\cO_L\llbracket t\rrbracket$ carries commuting actions of $\varphi_L$ and $\Gamma\cong\cO_L^\times\rtimes\Gal(L/K)$, and $\Gamma$ is noncommutative when $f>1$.
The obstruction does not apply because the $\Gamma$-action is semilinear over $\cO_{L}$.
With the chosen generator $\alpha\in T(H)$, let $\widetilde{\sigma}\in\Gamma$ be the order $f$ lift of $\sigma$ which fixes $\alpha$ as in the proof of Lemma \ref{lem:pot_lubintate_galois_group}.
The assignment $\gamma\mapsto\gamma(t)$ has fiber $\langle \widetilde{\sigma}\rangle$ over the identity series $t$, and is therefore noninjective for $f>1$.
It is not a homomorphism for ordinary composition of power series, because $\Gamma$ also acts on their coefficients.
For $f>1$, the action is therefore outside the coefficient-linear setting of Appendix \ref{app:onevariable}, although the action of $\Gamma$ on the full ring $\cT_H$ is faithful.
\end{remark}

\subsection{Comparison}\label{subsec:comparison_plt}
By Proposition \ref{prop:coefficient_ring}, the natural homomorphism $j:\cC_H\to R_H$ identifies $\cC_H$ with a coefficient ring of $R_H$ and gives an isomorphism
\[\cC_H\llbracket s_1,\cdots,s_{f-1}\rrbracket\isomto R_H,\quad s_k\mapsto t_{k,1}-(t_{0,1})^{q^k}.\]
The map $j$ is equivariant for $\varphi_L$ and $\Gamma$ and induces the identity on the common residue field $k_H$.
We first compare the theories over $\cC_H$ and $R_H$ with the common Frobenius $\varphi_L$, and then relate iteration of Frobenius to extension of coefficients from $\cO_K$ to $\cO_L$.

\begin{theorem}\label{thm:pLT_LT}
Base change along $j:\cC_H\to R_H$ induces an equivalence
\[j^*:\textup{$(\varphi_L,\Gamma)$-Mod}_{\cC_H}\to \textup{$(\varphi_L,\Gamma)$-Mod}_{R_H}.\]
If we write
\[\DD_{R_H}^{(f)}(V)=(V\otimes_{\cO_L}R_H^\sep)^{\Gal_{K_\infty}}\]
then there is a natural isomorphism
\[j^*\circ\DD_{\cC_H}\simeq\DD_{R_H}^{(f)}\]
making the following diagram commute
\begin{center}
\begin{tikzcd}
    \Rep_{\cO_L}^{\mathrm{semi}}(\Gal_K)\arrow[r,"{\DD_{\cC_H}}"]\arrow[rd,"{\DD_{R_H}^{(f)}}"']&\textup{$(\varphi_L,\Gamma)$-Mod}_{\cC_H}\arrow[d,"j^*"]\\
    &\textup{$(\varphi_L,\Gamma)$-Mod}_{R_H}
\end{tikzcd}
\end{center}
All three arrows in the diagram are equivalences.
\end{theorem}
\begin{proof}
Applying Proposition \ref{prop:fds_compatibiliy} to $j:\cC_H\to R_H$ as a morphism of flat $F$-dynamical systems over $\cO_L$, both equipped with $\varphi_L$, gives the commutative diagram
\begin{center}
\begin{tikzcd}
    \Rep_{\cO_L}(\Gal_{K_\infty})\arrow[d,"\DD_{\cC_H}"]\arrow[r,equal]&\Rep_{\cO_L}(\Gal_{K_\infty})\arrow[d,"\DD^{(f)}_{R_H}"]\\
    \textup{$\varphi_L$-Mod}_{\cC_H}\arrow[r,"j^*"]&\textup{$\varphi_L$-Mod}_{R_H}
\end{tikzcd}
\end{center}
The natural comparison isomorphisms are $\Gal_K$-equivariant, so they respect the residual $\Gamma$-actions.
Moreover, using the compatible maps $\cC_H\to R_H\to B\coloneqq W_{\cO_K}(\Cb)$, we obtain a canonical equivariant isomorphism
\[(M\otimes_{\cC_H}R_H)\otimes_{R_H}B\cong M\otimes_{\cC_H}B.\]
It is a homeomorphism for the finite-module weak topologies.
Therefore, the continuity condition of Definition \ref{def:phigammamod} is preserved and reflected by $j^*$.
This proves the assertion for the categories of \'{e}tale $(\varphi,\Gamma)$-modules.
\end{proof}

Let $M$ be an \'{e}tale $\varphi$-module over $(R_H,\varphi_K)$.
Iterating the invertible Frobenius linearization shows that $M$ with $\varphi_M^f$ is an \'{e}tale $\varphi$-module over $(R_H,\varphi_L)$.
Therefore, iteration defines
\[I_f(M,\varphi_M)=(M,\varphi_M^f)\]
both on \'{e}tale $\varphi$-modules and \'{e}tale $(\varphi,\Gamma)$-modules:
\[I_f:\textup{$(\varphi_K,\Gamma)$-Mod}_{R_H}\to \textup{$(\varphi_L,\Gamma)$-Mod}_{R_H}\]
Iteration of Frobenius corresponds to the coefficient extension functor $V\mapsto V\otimes_{\cO_K}\cO_L$ with diagonal semilinear $\Gal_K$-action.

\begin{proposition}
For $V\in \Rep_{\cO_K}(\Gal_K)$, there is a natural isomorphism
\[I_f(\DD_{R_H}(V))\cong\DD_{R_H}^{(f)}(V\otimes_{\cO_K}\cO_L)\]
making the following diagram commute.
\begin{center}
\begin{tikzcd}
    \Rep_{\cO_K}(\Gal_K)\arrow[d,"\DD_{R_H}"]\arrow[r,"-\otimes_{\cO_K}\cO_L"]&\Rep^{\mathrm{semi}}_{\cO_L}(\Gal_K)\arrow[d,"\DD^{(f)}_{R_H}"]\\
    \textup{$(\varphi_K,\Gamma)$-Mod}_{R_H}\arrow[r,"I_f"]&\textup{$(\varphi_L,\Gamma)$-Mod}_{R_H}
\end{tikzcd}
\end{center}
\end{proposition}
\begin{proof}
Fix $V\in \Rep_{\cO_K}(\Gal_K)$.
By definition, $M=\DD_{R_H}(V)$ consists of $\Gal_{K_\infty}$-invariant vectors inside
\[V\otimes_{\cO_K}R_H^\sep=(V\otimes_{\cO_K}\cO_L)\otimes_{\cO_L}R_H^\sep\]
so its underlying module is naturally identified with $M'=\DD_{R_H}^{(f)}(V\otimes_{\cO_K}\cO_L)$.
The endomorphism $\varphi_{M}$ is induced by $\id_V\otimes\varphi_K$, whereas $\varphi_{M'}$ is induced by $\id_{(V\otimes_{\cO_K}\cO_L)}\otimes\varphi_L$.
Therefore, $I_f(M)=M'$, proving the assertion.
\end{proof}

Therefore, passage from $\cC_H$ to $R_H$ preserves the category for the common Frobenius $\varphi_L$.
On the other hand, iteration from $\varphi_K$ to $\varphi_L$ corresponds to extension of representation coefficients from $\cO_K$ to $\cO_L$, and it is not an equivalence for $f>1$.

\newpage 
\appendix

\section{Ramified Witt vectors}\label{app:witt}
We recall the construction and basic properties of ramified Witt vectors, following \cite[\S 1.2]{Fargues2018Courbes}.

Let $K/\Qp$ be finite, with ring of integers $\cO_K$, uniformizer $\pi_K$, and residue field $\Fq$.
For each $n\ge0$, define the ramified Witt polynomial
\[w_n(X)=\sum_{i=0}^n \pi_K^i X_i^{q^{n-i}}\in\cO_K[X_0,\cdots,X_n].\]
There is a unique functor $W_{\pi_K}$ from $\cO_K$-algebras to $\cO_K$-algebras whose underlying set-valued functor is $R\mapsto R^\NN$ for which the ghost map
\[w_{\pi_K}:W_{\pi_K}(R)\to R^{\NN},\quad (a_i)_i\mapsto (w_n(a_0,\cdots,a_n))_n\]
is a natural $\cO_K$-algebra homomorphism, where the target has coordinatewise ring operations and the diagonal $\cO_K$-algebra structure.

\begin{lemma}[cf. {\cite[Lemma 1.2.1]{Fargues2018Courbes}}]\label{lem:Phi}
Let $\Phi(X,Y)\in \cO_K[X,Y]$.
There is a unique sequence
\[\Phi_i(X,Y)\in \cO_K[X_0,\cdots,X_i,Y_0,\cdots,Y_i],\quad i\ge0\]
such that, for every $n\ge0$,
\[\Phi(w_n(X),w_n(Y))=w_n(\Phi_0(X,Y),\cdots,\Phi_n(X,Y)).\]
\end{lemma}

Applying Lemma \ref{lem:Phi} to $\Phi(X,Y)=X+Y$ and $\Phi(X,Y)=XY$ gives the addition and multiplication polynomials in $\pi_K$-Witt coordinates.

Let $\pi_K$ and $\pi_K'$ be two uniformizers of $\cO_K$.
There is a  unique natural $\cO_K$-algebra isomorphism 
\[u_R:W_{\pi_K}(R)\isomto W_{\pi_K'}(R)\]
such that $w_{\pi_K'}\circ u_R=w_{\pi_K}$.
These comparison isomorphisms satisfy the cocycle condition, so they define a functor $W_{\cO_K}$ together with its ghost map, independently of the choice of the uniformizer.

The multiplicative Teichm\"{u}ller map 
\[R\to W_{\cO_K}(R),\quad r\mapsto [r]\]
is given in $\pi_K$-Witt coordinates by $[r]=(r,0,0,\cdots)$, independently of $\pi_K$, and has ghost coordinates $(r,r^q,r^{q^2},\cdots)$.

\underline{Frobenius.}
There is a unique natural $\cO_K$-algebra endomorphism of $W_{\cO_K}$ satisfying
\[w_n(\varphi_K(x))=w_{n+1}(x)\]
for $n\ge0$.
Equivalently, the ghost map intertwines $\varphi_K$ with the left shift $(a_i)_{i\ge0}\mapsto (a_{i+1})_{i\ge0}$, making the following diagram commute.
\begin{center}
    \begin{tikzcd}
    W_{\cO_K}(R)\arrow[r,"\varphi_K"]\arrow[d,"w"]&W_{\cO_K}(R)\arrow[d,"w"]\\
    R^{\NN}\arrow[r,"\textup{shift}"]&R^\NN
    \end{tikzcd}
\end{center}
We call $\varphi_K$ the \textit{Frobenius} endomorphism.
In $\pi_K$-Witt coordinates,
\[(\varphi_K(X))_i=X_i^q+\pi_K f_i(X_0,\cdots,X_{i+1})\]
for $i\ge0$, where $f_i\in\cO_K[X_0,\cdots,X_{i+1}]$ is weighted homogeneous of degree $q^{i+1}$ for the weights $\deg X_j=q^j$.

\underline{Verschiebung.}
For each uniformizer $\pi_K$, the \textit{Verschiebung} is the natural $\cO_K$-linear map
\[V_{\pi_K}:W_{\cO_K}(R)\to W_{\cO_K}(R)\]
given in $\pi_K$-Witt coordinates by
\[(a_0,a_1,\cdots)\mapsto (0,a_0,a_1,\cdots).\]
It is generally not an algebra homomorphism.
For every $\cO_K$-algebra $R$,
\[\varphi_KV_{\pi_K}=\pi_K\]
where the right-hand side denotes multiplication by $\pi_K$.
If the $\cO_K$-algebra structure on $R$ factors through $\Fq$, then also
\[V_{\pi_K}\varphi_K=\pi_K.\]
In this case, we have
\[\varphi_K((r_i)_i)=(r_i^q)\]
and multiplication by $\pi_K$ is given in $\pi_K$-Witt coordinates by 
\[(a_0,a_1,a_2,\cdots)\mapsto (0,a_0^q,a_1^q,\cdots).\]
\begin{lemma}\label{lem:pth_power_Witt}
Use the $\pi_K$-Witt coordinates associated with the chosen uniformizer of $K$.
For each $n\ge0$, there is a unique polynomial $f_n\in\Fq[X_0,\cdots,X_n]$ such that for every $\Fq$-algebra $R$ and $r=(r_0,r_1,\cdots)\in W_{\cO_K}(R)$, the $n$-th Witt coordinate of $r^p$ is $f_n(r_0,\cdots,r_n)$.
Moreover, $f_0=X_0^p$, $f_1=0$, and $f_n$ is independent of $X_n$ for every $n\ge2$.
\end{lemma}
\begin{proof}
Lemma \ref{lem:Phi} applied to $\Phi(X)=X^p$ gives polynomials $\Phi_n(X)\in\cO_K[X_0,\cdots,X_n]$ satisfying 
\[(w_n(X_0,\cdots,X_n))^p=w_n(\Phi_0,\cdots,\Phi_n).\]
Reducing $\Phi_n$ modulo $\pi_K$ gives the required polynomial $f_n$, whereas uniqueness follows by evaluating the operation on the universal tuple in $\Fq[X_0,\cdots,X_n]$.
Since $w_0(X)=X_0$, we have $\Phi_0=X_0^p$ and therefore $f_0=X_0^p$.

Before we proceed with the proof, we prepare a standard approximation lemma.
Fix an integer $\ell\ge1$.
If $\alpha,\beta$ lie inside a polynomial ring over $\cO_K$, then
\begin{equation}\label{eq:little_lemma}
    \alpha\equiv\beta\bmod{\pi_K^\ell}\Rightarrow \alpha^p\equiv\beta^p\bmod{\pi_K^{\ell+1}}
\end{equation}
Put $\delta=\beta-\alpha\in\pi_K^\ell\cO_K[X]$.
For $1\le j<p$, the term $\binom{p}{j}\alpha^{p-j}\delta^j$ is divisible by $\pi_K^{\ell+1}$, so $\beta^p\equiv \alpha^p+\delta^p\bmod{\pi_K^{\ell+1}}$.
Since $p \ell\ge\ell+1$, the term $\delta^p$ also vanishes modulo $\pi_K^{\ell+1}$.
This proves (\ref{eq:little_lemma}).

By definition of the polynomials $\Phi_n$, 
\[w_n(X)^p=w_n(\Phi_0,\cdots,\Phi_n)=\sum_{i=0}^n \pi_K^i(\Phi_i(X))^{q^{n-i}}.\]
Choose a polynomial lift $\widetilde{f}_i\in\cO_K[X_0,\cdots,X_i]$ of each $f_i$.
Repeated application of (\ref{eq:little_lemma}) gives
\[\pi_K^i(\Phi_i(X))^{q^{n-i}}\equiv \pi_K^i(\widetilde{f}_i)^{q^{n-i}}\bmod{\pi_K^{n+1}}.\]
For $n\ge1$, removing the summand $\pi_K^nX_n$ changes $w_n(X)$ by an element of $\pi_K^n\cO_K[X]$, so (\ref{eq:little_lemma}) gives 
\[(w_n(X))^p\equiv \left(\sum_{i=0}^{n-1}\pi_K^iX_i^{q^{n-i}}\right)^p\bmod{\pi_K^{n+1}}.\]
Therefore, for $n\ge1$,
\begin{equation}\label{eq:witt}
    \sum_{i=0}^n \pi_K^i(\widetilde{f}_i)^{q^{n-i}}\equiv \left(\sum_{i=0}^{n-1}\pi_K^iX_i^{q^{n-i}}\right)^p\bmod{\pi_K^{n+1}}.
\end{equation}
Therefore, $f_n$ is independent of $X_n$.
For $n=1$, the congruence reduces to
\[X_0^{pq}+\pi_K\widetilde{f}_1\equiv X_0^{pq}\bmod{\pi_K^2}\]
hence $f_1=0$.
\end{proof}

\newpage
\section{An obstruction to coefficient-linear actions in one variable}\label{app:onevariable}
Berger proved that a coefficient-linear lift of the field of norms satisfying the finite-height condition forces $\Gamma$ to be abelian \cite{Berger2014Lifting}.
We isolate the formal power series argument underlying this obstruction.
\begin{assumption}\label{assum:ps}
Let $A$ be a complete discrete valuation ring of characteristic zero, with maximal ideal $\frakm_A$, and let $\Gamma$ be an infinite set.
Suppose that $P(T)\in A[[T]]_0\setminus\{0\}$ satisfies $P'(0)\in\frakm_A$ and that $\{F_{\gamma}(T)\in A[[T]]_0\}_{\gamma\in\Gamma}$ is a family of pairwise distinct power series, each invertible under composition, such that
\[P\circ F_{\gamma}=F_{\gamma}\circ P\] 
for every $\gamma\in\Gamma$.
\end{assumption}

\begin{proposition}\label{prop:injectivity}
Under Assumption \ref{assum:ps}, the map 
\[\Gamma\to A,\quad\gamma\mapsto F_{\gamma}'(0)\]
is injective.
In particular, an infinite group acting faithfully on $A\llbracket T\rrbracket$ by $A$-algebra automorphisms fixing the origin and commuting with $T\mapsto P(T)$ must be abelian.
\end{proposition}

The rest of this appendix is devoted to the proof of Proposition \ref{prop:injectivity}.
Let $K$ be a field, and let $f(T)\in K\llbracket T\rrbracket_0$.
Define
\[\Comm_K(f)=\{g(T)\in K\llbracket T\rrbracket_0:f\circ g=g\circ f\}\]
and write $\Comm_K(f)^\times$ for the subgroup consisting of series with $g'(0)\neq0$.

\begin{proposition}[{\cite[Proposition 1.1]{Lubin1994NonArchimedean}}]\label{prop:stable}
Let $f(T)\in K\llbracket T\rrbracket_0$.
Suppose that $a=f'(0)$ is nonzero and is not a root of unity\footnote{These two conditions say that $f(T)$ is stable in Lubin's terminology \cite{Lubin1994NonArchimedean}}.
The map 
\[\Comm_K(f)\to K,\quad g(T)\mapsto g'(0)\]
is injective.
\end{proposition}

If $\mathrm{char}(K)=0$ and $f\neq0$, the existence of infinitely many invertible elements of $\Comm_K(f)$ forces $f'(0)\neq0$.
The following lemma abstracts the leading-degree argument used in the proof of \cite[Lemma 4.5]{Berger2014Lifting}.
\begin{lemma}\label{lem:deg2_finiteneness}
Let $K$ be a field of characteristic zero.
Let $f(T) \in K\llbracket T\rrbracket_0\setminus\{0\}$ have order $k\ge2$.
The derivative at the origin defines an injective group homomorphism
\[\Comm_K(f)^\times\to \mu_{k-1}(K),\quad g\mapsto g'(0)\]
Therefore, 
\[\#\Comm_K(f)^\times\le k-1.\]
\end{lemma}
\begin{remark}
\begin{enumerate}
\item The restriction to invertible elements is essential for its finiteness conclusion.
For example, $f(T)=T^2$ commutes with $g_d(T)=T^d$ for every $d\ge1$, so its full centralizer is infinite.

\item 
The characteristic zero hypothesis is also crucial.
If $K$ is a field of characteristic $p>0$, then every $g(T)\in\Fp\llbracket T\rrbracket_0$ commutes with $f(T)=T^p$.
\end{enumerate}
\end{remark}
\begin{proof}[Proof of Lemma \ref{lem:deg2_finiteneness}]
Write $f(T)=c_kT^k+\cdots$ with $c_k\neq0$, and let $g(T)=\lambda T+\cdots\in \Comm_K(f)^\times$, so that $\lambda\neq0$.
Comparing coefficients of $T^k$ in $f\circ g=g\circ f$ gives $c_k\lambda^k=c_k\lambda$, so $\lambda^{k-1}=1$.

Suppose that $g$ and $g+\Delta$ are two elements of $\Comm_K(f)^\times$ with the same derivative.
If $\Delta\neq0$, then $\Delta(T)=bT^d+\cdots$ with $b\neq0$ and $d\ge2$.
Commutativity gives
\[f(g+\Delta)-f(g)=\Delta(f).\]
The left-hand side has order $k-1+d$ by the Taylor expansion, whereas the right-hand side has order $kd$.
Therefore, $k-1+d=kd$, contradicting $(k-1)(d-1)>0$.
Therefore, $\Delta=0$.
\end{proof}

We obtain the following immediate consequence.
\begin{corollary}\label{cor:infinite_inv_comm_then_linear}
Let $K$ be a field of characteristic zero.
If $f(T)\in K\llbracket T\rrbracket_0\setminus\{0\}$ and $\Comm_K(f)^\times$ is infinite, then $f'(0)\neq0$.
\end{corollary}
We now prove the claimed assertion.
\begin{proof}[Proof of Proposition \ref{prop:injectivity}]
Applying Corollary \ref{cor:infinite_inv_comm_then_linear} over $K=\Frac(A)$ gives $P'(0)\neq0$.
Since $P'(0)\in\frakm_A\setminus\{0\}$, it is not a root of unity.
Proposition \ref{prop:stable} therefore makes the derivative map
\[\Comm_{K}(P)\to K,\quad g\mapsto g'(0)\]
injective.
Restricting to the distinct series $F_\gamma$ proves the assertion.
\end{proof}

\newpage
\bibliographystyle{plain}
\bibliography{ref}
\end{document}